\documentclass[11pt]{article}
\usepackage{amssymb}
\usepackage{mathrsfs}
\usepackage{latexsym,amsmath,amssymb,amsfonts,epsfig,graphicx,cite,psfrag}
\usepackage{eepic,color,colordvi,amscd}
\usepackage{ebezier}
\usepackage{verbatim}
\usepackage{subcaption}
\usepackage{enumitem}
\usepackage{algorithm}
\usepackage{algorithmic}
\usepackage{hyperref}
\usepackage{amsthm}
\usepackage{longtable}
\usepackage{comment}
\usepackage{color}
\usepackage{xcolor}
\usepackage{graphicx} %use graph format
\usepackage{epstopdf}
\usepackage{longtable}
\usepackage{booktabs}
\usepackage{mathtools}
\usepackage{multicol, multirow}
\usepackage{capt-of}
\usepackage{longtable}
\usepackage{amsopn}
\usepackage{tikz}
\definecolor{blue}{rgb}{0,0,0.9}
\definecolor{red}{rgb}{0.9,0,0}
\definecolor{green}{rgb}{0,0.9,0}
\theoremstyle{plain}
\newtheorem{remark}{Remark}
\newtheorem{example}{Example}

\newtheorem{assumption}{Assumption}

\newtheorem{theorem}{Theorem}

\newtheorem{lemma}{Lemma}

\def\H{\mathcal{H}}
\def\<{\big\langle}
\def\>{\big\rangle}
\def\E{\mathcal{E}}
\def\M{\mathcal{M}}

\def\A{ \mathcal{A} }
\def\B{\mathcal{B}}
\def\C{\mathcal{C}}
\def\K{\mathcal{K}}
\def\L{\mathcal{L}}
\def\F{\mathcal{F}}

\def\I{\mathcal{I}}

\def\R{\mathbb{R}}
\def\bN{\mathbb{N}}

\def\G{\mathcal{G}}
\def\N{\mathcal{N}}

\def\S{\mathbb{S}}

\def\O{\mathcal{O}}

\def\({\left(}
\def\){\right)}
\def\hR{\widehat{R}}

\def\tb#1{\textbf{#1}}

\def\RR{{\widehat R\widehat R^\top}}
\def\tol{\tt{tol}}
\def\timelimit{\tt{TimeLimit}}
\DeclareMathOperator{\diag}{diag}

\let\svthefootnote\thefootnote
\newcommand\blankfootnote[1]{
\let\thefootnote\svthefootnote
}
\usepackage[many]{tcolorbox}
\newtcolorbox{boxA}{
    fontupper = \bf,
    boxrule = 1.5pt,
    colframe = black % frame color
}
\newcommand{\xdownarrow}[1]{%
  {\left\downarrow\vbox to #1{}\right.\kern-\nulldelimiterspace}
}

\begin{document}
        \title{RiNNAL+: a Riemannian ALM Solver for SDP-RLT Relaxations of Mixed-Binary Quadratic Programs}
        \author{
        Di Hou\thanks{Department of Mathematics, National
         University of Singapore, Singapore
         119076 ({\tt dihou@u.nus.edu}).
         }, \quad 
	Tianyun Tang\thanks{Institute of Operations Research and Analytics, National
         University of Singapore, Singapore
         119076 ({\tt ttang@u.nus.edu}).
         }, \quad 
	  Kim-Chuan Toh\thanks{Department of Mathematics, and 
         Institute of Operations Research and Analytics, 
         National University of Singapore, 
         Singapore
         119076 ({\tt mattohkc@nus.edu.sg}). 
         }
 	}
	\date{\today}
\maketitle

%%%%%%%%%%%%%%%%%%%%%%%%%%%%%%%%%%%%%%%%%%%%%%%%%%%%%%%%%%%%%%%%%%%%%%%%%%%%%%%
%%%%%%%%%%%%%%%%%%%%%%%%%%%%%%%%%%%%%%%%%%%%%%%%%%%%%%%%%%%%%%%%%%%%%%%%%%%%%%%

\begin{abstract}

%% SDP-RLT and DNN
Doubly nonnegative (DNN) relaxation usually provides a tight lower bound for a mixed-binary quadratic program (MBQP). However, solving DNN problems is challenging because: (1) the problem size is $\Omega((n+l)^2)$ for an MBQP with $n$ variables and $l$ inequality constraints, and (2) the rank of optimal solutions cannot be estimated a priori due to the absence of theoretical bounds. In this work, we propose RiNNAL+, a Riemannian augmented Lagrangian method (ALM) for solving DNN problems. 
We prove that the DNN relaxation of an MBQP, with matrix dimension $(n+l+1)$, is equivalent to the SDP-RLT relaxation (based on the reformulation-linearization technique) with a smaller matrix dimension $(n+1)$.
In addition, we develop a hybrid method that alternates between two phases to solve the ALM subproblems.
% Rie
In phase one, we apply low-rank matrix factorization and random perturbation to transform the feasible region into a lower-dimensional manifold so that we can use the Riemannian gradient descent method.
% PG
In phase two, we apply a single projected gradient step to update the rank of the 
underlying variable and escape from spurious local minima arising in the first phase if necessary. 
To reduce the computation cost of the projected gradient step,
we develop pre-processing and warm-start techniques for acceleration.
Unlike traditional rank-adaptive methods that require extensive parameter tuning, our hybrid method requires minimal tuning.
%% numerical 
Extensive experiments confirm the efficiency and robustness of RiNNAL+ in solving various classes of large-scale DNN problems.
\end{abstract}

%%%%%%%%%%%%%%%%%%%%%%%%%%%%%%%%%%%%%%%%%%%%%%%%%%%%%%%%%%%%%%%%%%%%%%%%%%%%%%%
%%%%%%%%%%%%%%%%%%%%%%%%%%%%%%%%%%%%%%%%%%%%%%%%%%%%%%%%%%%%%%%%%%%%%%%%%%%%%%%

\section{Introduction}\label{sec-intro}

%%%%%%%%%%%%%%%%%%%%%%%%%%%%%%%%%%%%%%%%%%%%%%%%%%%%%%%%%%%%%%%%%%%%%%%%%%%%%%%

\subsection{Mixed-binary nonconvex quadratic program}

In this paper, we consider the following mixed-binary quadratic program:
\begin{equation}\label{MBQP}
v^{P_1}:=\min\left\{x^{\top} Q x+2 c^{\top}x :  Ax=b,\ Gx\leq d,\ x_i \in\{0,1\},\ \forall i \in B,\ x\in \R^{n}_+\right\},\tag{MBQP}
\end{equation}
where $Q\in \mathbb{S}^{n},\,c\in\R^n,\, A \in \R^{m\times n},\, G \in \R^{l\times n},\, b \in \R^m,\, d \in \R^l$, $B\subseteq[n]$ is the index set of $p$ binary variables. Without loss of generality, we assume that $A$ has full row rank and $b\geq 0.$ Problem \eqref{MBQP} is general because it includes both binary and continuous nonnegative variables, as well as equality and inequality constraints. The objective function features both quadratic and linear terms, allowing it to model a broad range of optimization problems. This formulation encompasses key problems such as 0-1 mixed-integer programming (MIP), non-convex quadratic programming (QP), binary integer nonconvex quadratic programming (BIQ), and more.

Since \eqref{MBQP} is generally nonconvex and NP-hard, solving it to global optimality is computationally intractable in most cases. As a result, numerous convex relaxations have been developed to approximately solve it efficiently \cite{bomze2002solving,bundfuss2009adaptive}. 
Among these, semidefinite programming (SDP) relaxations have been particularly effective due to their ability to provide tight lower bounds and, in some cases, exact solutions under specific conditions \cite{goemans1995improved,wang2022tightness,burer2020exact,arima2024exact}.
Beyond stand-alone approximations, convex relaxations also play a crucial role in global optimization frameworks, such as branch-and-bound and branch-and-cut methods \cite{lawler1966branch,mitchell2002branch}. In these frameworks, relaxations are used to derive valid bounds that help prune the search space, significantly improving computational efficiency. The strength of the relaxation directly impacts the overall performance of such global approaches, making SDP relaxations a key ingredient in state-of-the-art exact algorithms for nonconvex quadratic optimization \cite{buchheim2013semidefinite,gusmeroli2022biqbin,krislock2017biqcrunch}.

In the next subsection, we introduce several widely used SDP relaxations, with a particular emphasis on the SDP-RLT relaxation. We discuss their formulations, the quality of their respective bounds, and numerical challenges associated with their implementation in deriving tight lower bounds for \eqref{MBQP}.

\begin{remark}\label{remark-QCQP}
The algorithm we propose in this paper is not limited to just solving the DNN relaxation of \eqref{MBQP}. 
It can also handle the relaxations of QCQP problems 
with additional general quadratic constraints:
\begin{equation}\label{eq-QCQP}
\begin{aligned}
x^\top A_i x + b_i^\top x + c_i &\leq 0, \quad i=1,\dots,m_1, \\
x^\top A_j x + b_j^\top x + c_j &= 0,     \quad j=m_1+1,\dots,m_2,
\end{aligned}
\end{equation}
where $A_i\not=0_{n\times n}$ for all $i\in[m_2]$. By introducing the new variable $X\in \S^{n}_+$ to represent $xx^\top$, we obtain the following convex relaxation of \eqref{eq-QCQP}:
\begin{align}
    \langle A_i, X \rangle + b_i^\top x + c_i &\leq 0, \quad i = 1, \dots, m_1, \label{ieq-QCQP-lift} \\
    \langle A_j, X \rangle + b_j^\top x + c_j &= 0, \quad j = m_1 + 1, \dots, m_2. \label{eq-QCQP-lift}
\end{align}
For instance, complementarity constraints are a special class of quadratic constraints defined by $x_ix_j = 0$ for $(i,j) \in E$, where $E \subseteq \{(i,j) \mid 1 \leq i < j \leq n\}$ represents the set of 
complementary index pairs. By lifting these constraints using \eqref{eq-QCQP-lift}, we obtain $X_{ij}=0,\, \forall (i,j)\in E$.
By incorporating \eqref{ieq-QCQP-lift} and \eqref{eq-QCQP-lift} into the set $\I$ and $\E$ defined in the subsequent subsection, respectively, our results for \eqref{MBQP} extend naturally to general QCQP problems.
For simplicity, we exclude additional quadratic constraints of the form \eqref{eq-QCQP} in \eqref{MBQP} throughout this paper.
\end{remark}

%%%%%%%%%%%%%%%%%%%%%%%%%%%%%%%%%%%%%%%%%%%%%%%%%%%%%%%%%%%%%%%%%%%%%%%%%%%%%%%

\subsection{SDP-RLT relaxations of (\ref{MBQP})}\label{subsec-SDPRLT}

\noindent\tb{Shor relaxation.} A commonly used convex relaxation of general QCQP problems is the Shor relaxation \cite{shor1990dual}, which is obtained by replacing the quadratic term $xx^\top$ with a matrix $X$ and introducing a positive semidefinite constraint $X-xx^\top\succeq 0$. The explicit formulation of the Shor relaxation \eqref{SHOR} of \eqref{MBQP} is given in Subsection \ref{subsec-relaxation}. While the Shor relaxation can be exact under certain conditions \cite{wang2022tightness,burer2020exact}, it often yields weak lower bounds for most instances of \eqref{MBQP}. In some cases, the bound can even be unbounded below. Therefore, additional constraints are necessary to strengthen the relaxation and improve bound quality.
\medskip

\noindent\tb{DNN relaxation.} To address this limitation, researchers have explored the doubly nonnegative (DNN) relaxation, which extends the Shor relaxation by incorporating nonnegativity constraints on $X$. Initially proposed by Burer \cite{burer2009copositive, burer2010optimizing} and further analyzed in \cite{kim2022doubly, bomze2017fresh, RiNNAL}, the DNN relaxation generally provides a significantly tighter relaxation than the Shor relaxation. The explicit formulation of the DNN relaxation \eqref{DNN} of \eqref{MBQP} is given in Subsection \ref{subsec-relaxation}. However, deriving \eqref{DNN} requires converting all inequality constraints into equality constraints via the introduction of slack variables. This transformation can substantially increase both the matrix dimension and the number of constraints in the resulting relaxation. For instance, if the number of inequality constraints $l$ equals the variable dimension $n$, the matrix dimension of the DNN relaxation increases to $2n$, significantly increasing computational complexity.
\medskip

\noindent\tb{SDP-RLT relaxation.} To address the high dimensionality of \eqref{DNN} while maintaining its bound quality, we consider the SDP-RLT relaxation of \eqref{MBQP}. Unlike \eqref{DNN}, which increases the problem size by introducing additional slack variables, the SDP-RLT relaxation preserves the original problem dimension by deriving quadratic constraints from the linear constraints of \eqref{MBQP}. Specifically, it integrates the reformulation-linearization technique (RLT) \cite{sherali2007rlt,qiu2024polyhedral} with the Shor relaxation \eqref{SHOR}, yielding a tighter bound than \eqref{SHOR}. In fact, we will further show that it achieves the same bound as \eqref{DNN}.

To systematically construct the SDP-RLT relaxation of \eqref{MBQP}, we first apply the RLT technique to obtain the quadratic constraints. These constraints are derived by multiplying pairs of linear inequalities, including $x \geq 0$, and by multiplying each equality constraint with a decision variable. Next, we replace the quadratic term $xx^\top$ with $X$ and introduce the variable $z = 1$ to homogenize all constraints. The derivation process is shown as follows:
\begin{tcolorbox}[colframe=black, colback=white, boxrule=1pt, arc=2mm, top=0pt]
\begin{gather*}
 Ax - b = 0, \quad d - Gx \geq 0, \quad x \geq 0  \\
 \Big\downarrow  \text{ (RLT)}\phantom{\text{.}}  \\
 (Ax - b)x^\top = 0, \quad (d - Gx)x^\top \geq 0, \quad (d - Gx)(d - Gx)^\top \geq 0, \quad xx^\top \geq 0 \\
  \Big\downarrow \text{ (LIFT)}  \\
 AX-bx^\top = 0, \quad dx^\top-GX \geq 0, \quad \G(Y) \geq 0, \quad X \geq 0,\quad z=1,
\end{gather*}
\end{tcolorbox}
\noindent where $Y:=\begin{bmatrix}z& x^\top \\ x & X\end{bmatrix}\in  \mathbb{S}^{n+1}$, and  the mapping $\G:\S^{n+1}\rightarrow\S^{l}$ is defined as 
\begin{equation}\label{def-G}
\G(Y):=GXG^\top -Gxd^\top -dx^\top G^\top+zdd^\top.
\end{equation}
Finally, by incorporating these constraints into the Shor relaxation, we obtain the SDP-RLT relaxation as follows:

\begin{equation}
\label{SDP-RLT}
v^{\text{SDP-RLT}}:=\min\left\{
\<C,Y\>  :\  
Y\in \F\cap\I_R\cap \mathbb{S}_{+}^{n+1} \cap\mathbb{N}^{n+1}
\right\},
\tag{SDP-RLT}
\end{equation}
where the cost matrix $C=[0,c^\top; c, Q]\in\S^{n+1}$. The sets defined by equality and inequality constraints, denoted by $\F$ and $\I_R$ respectively, are given by 
\begin{align*}
\F&:= \left\{\begin{bmatrix}z& x^\top \\ x & X\end{bmatrix}\in  \mathbb{S}^{n+1} :  Ax=b,\ AX=bx^\top,\ x_i =X_{ii}, \  \forall i \in B,\ z=1\right\},\\
\I_R    &:= \left\{\begin{bmatrix}z& x^\top \\ x & X\end{bmatrix}\in  \mathbb{S}^{n+1} :  \G(Y)\geq 0,\ dx^\top-GX\geq 0,\ zd-Gx\geq 0 \right\}.
\end{align*}
Here $\mathbb{S}_{+}^{n+1}$ denotes the cone of positive semidefinite matrices in $\mathbb{S}^{n+1}$, and $\mathbb{N}^{n+1}$ denotes the cone of nonnegative matrices in $\mathbb{R}^{(n+1) \times(n+1)}$. Note that the last constraint $zd-Gx\geq 0$ in $\I_R$ is redundant, provided that at least one component of $x$ is bounded in the feasible set $\{x\in\R^n_+:\ Ax=b,\ Gx\leq d\}$, see \cite[Proposition 8.1]{sherali2013reformulation}.
Compared with the lower bounds provided by \eqref{SHOR} and \eqref{DNN}, we prove in Theorem \ref{thm-1} that
\begin{equation*}
v^{\text{SHOR}} \leq 
v^{\text{DNN}} =
v^{\text{SDP-RLT}}  \leq 
v^{P_1}.
\end{equation*}
Thus, \eqref{SDP-RLT} achieves the same lower bound as \eqref{DNN} while maintaining the same variable dimension as \eqref{SHOR}. This result can be applied to two important problem classes: the strengthened MBQP and sparse QP with an $l_0$-norm constraint, as discussed in subsection~\ref{subsec-application}. 
A detailed theoretical and numerical comparison of these relaxations is provided in Section \ref{sec-relaxations} and Subsection \ref{subsec-tightness-numerical}, respectively. Our analysis demonstrates that \eqref{SDP-RLT} is generally more computationally efficient than \eqref{DNN}. Given this advantage, the primary focus of this paper is to develop an efficient method for solving the \eqref{SDP-RLT} relaxation. The next subsection reviews existing algorithms for \eqref{SDP-RLT} and identifies their limitations.

%%%%%%%%%%%%%%%%%%%%%%%%%%%%%%%%%%%%%%%%%%%%%%%%%%%%%%%%%%%%%%%%%%%%%%%%%%%%%%%

\subsection{Challenges in solving (\ref{SDP-RLT})}

Our main question is how to efficiently solve \eqref{SDP-RLT}. Renowned SDP solvers like SDPT3 \cite{TTT}, SeDuMi \cite{sturm1999using}, and DSDP \cite{benson2008algorithm}, which utilize interior point methods, are rarely used for solving SDP-RLT or DNN problems due to their high computational costs per iteration, scaling as $\mathcal{O}(n^6)$. Instead, first-order methods based on the alternating direction method of multipliers (ADMM) \cite{SDPNALp1,chen2017efficient} are preferred for these problems. Although solvers such as SDPNAL+ \cite{SDPNAL,SDPNALp1,SDPNALp2}, which employ the augmented Lagrangian method (ALM), have been quite effective in solving medium-size problems (with $n \leq 2000$), solving large-scale instances (say with $n \geq 3000$) remains a highly challenging task. This difficulty arises primarily from the costly eigenvalue decompositions required by ADMM-type or ALM-type methods to perform projections onto $\S^n_+$, as well as slow convergence issues caused by the degeneracy of solutions. 

A closely related work is \cite{RiNNAL}, where the authors proposed RNNAL (which we rename as RiNNAL here for ease of pronunciation), a method for solving DNN problems by leveraging their solutions' potential low-rank property. RiNNAL is a globally convergent Riemannian ALM that penalizes the nonnegativity and complementarity constraints while preserving all other constraints in the ALM subproblem. After applying the low-rank decomposition to the ALM subproblem, the resulting feasible region becomes an algebraic variety with favorable geometric properties. In \cite{RiNNAL}, it was demonstrated that RiNNAL can substantially outperform other
state-of-the-art solvers in solving large-scale DNN problems. 
However, RiNNAL still has several limitations:

\begin{enumerate} 
\item RiNNAL is not applicable for solving general \eqref{SDP-RLT} problems with the inequality constraints imposed by $Y\in\I_R$, which arise from the conditions $Gx\leq d$ and $x\geq 0$. While RiNNAL can solve \eqref{DNN}
(whose equivalent reformulation is \eqref{SDP-RLT}), the variable dimension of \eqref{DNN} increases significantly with the number of inequality constraints, resulting in higher computational costs.

\item Most low-rank decomposition algorithms  \cite{wang2023decomposition, wang2023solving, tang2024feasible,RiNNAL,monteiro2024low} including RiNNAL require frequently tuning the rank of the factorized variable. On one hand, when the dual infeasibility of the KKT system is large, the rank needs to be increased to escape from the saddle points of the factorized ALM subproblem. However, the appropriate rank increment is a challenging hyperparameter to tune: overly large increments will unnecessarily enlarge the problem size, while 
insufficient increments will require repeated updates to achieve convergence. On the other hand, rank reduction is equally important for saving memory and reducing computational costs, typically achieved by dropping near-zero singular values and their corresponding singular vectors from the factored variable. However, selecting the threshold for this reduction is another difficult hyperparameter to tune. A large threshold may cause significant changes to the iteration matrix, leading to jumps in the objective function and harming convergence, while a small threshold results in slow rank reduction and incurs unnecessary computational costs. 
Both the frequency and magnitude of rank updates are often performed heuristically, varying across different cases, and the parameter selection can greatly affect the performance of RiNNAL.

\item When encountering non-smooth points in the Riemannian gradient descent inner loop for solving the ALM subproblem, RiNNAL needs to reformulate the DNN relaxation problem equivalently into a higher-dimensional problem to ensure the linearly independent constrained qualification (LICQ). However, solving the higher-dimensional reformulated problem significantly increases the corresponding computational time.
\end{enumerate}

The goal of this paper is to propose a suite of techniques, from various perspectives, to resolve the above-mentioned issues and enhance the computational performance of our previous algorithm RiNNAL. We will state our techniques and contributions in the next two subsections.

%%%%%%%%%%%%%%%%%%%%%%%%%%%%%%%%%%%%%%%%%%%%%%%%%%%%%%%%%%%%%%%%%%%%%%%%%%%%%%%

\subsection{A hybrid method for solving ALM subproblems}

In Section \ref{sec-alg}, we propose RiNNAL+, an enhanced version of our previous algorithm RiNNAL to solve general SDP relaxation problems including both \eqref{SDP-RLT} and \eqref{DNN}. The similarities and distinctions in applying RiNNAL+ to these relaxations are discussed in Subsections \ref{subsec-alg-equivalence} and \ref{subsec-pp}. For clarity, we use \eqref{SDP-RLT} as a representative example to illustrate the core ideas of RiNNAL+ in this subsection. The primary innovation in RiNNAL+ is its hybrid approach, which consists of two phases for solving the following ALM subproblem:
\begin{equation}
    \min\left\{\<C,Y\>+\frac{\sigma}{2}\| \Pi_{+}( \sigma^{-1}\mu-(\C(Y)-l)) \|^2 :\ Y\in \F\cap \S^{n+1}_+\right\},\label{CVX}
    \tag{CVX}
\end{equation}
where $\C(Y)\geq l$ denotes all the inequality constraints of \eqref{SDP-RLT} imposed by $Y\in \I_R\cap \mathbb{N}^{n+1}$, $\mu$ is the corresponding Lagrangian 
dual multiplier, and $\sigma $ is the penalty parameter. RiNNAL+ switches between two phases to efficiently solve \eqref{CVX}.

\medskip

%%%%%%%%%%%%%%%%%%%%%%%%%%%%%%%%%%%%%%%%%%%%%%%%%%%%%%%%%%%%%%%%%%%%%%%%%%%%%%%

\noindent\tb{Low-rank phase.} Suppose that the subproblem \eqref{CVX} has an optimal solution of rank $r$, where $r$ is a positive integer. To fully utilize the potential low-rank property of the solutions to \eqref{CVX}, we apply a special low-rank factorization proposed in \cite{RiNNAL} to the variable $Y$ and simplify \eqref{CVX} to the following equivalent model:
\begin{equation}
    \min\left\{\<C,\hR\hR^\top\>+\frac{\sigma}{2}\| \Pi_{+}( \sigma^{-1}\mu-(\C(\hR\hR^\top)-l)) \|^2 :\ R\in \M_r \right\},\label{LR}
\tag{LR}
\end{equation}
where $R\in \R^{n\times r}$ is the matrix variable, $\hR:=[e_1^\top;R]\in\R^{(n+1)\times r}$ is the factor matrix with $e_1$ being the first standard unit vector in $\R^r$, and $\M_r$ (derived from the low-rank formulation
of ${\cal F}\cap \S^{n+1}_+$) is defined as 
\begin{equation*}
%\label{algebraic-variety-M}
\mathcal{M}_r := \left\{R \in \mathbb{R}^{n \times r} :\  AR=be_1^\top,\ \operatorname{diag}_B(RR^{\top})=R_B e_1 \right\}.
\end{equation*}
We refer the reader to Subsection~\ref{subsec-notations} for the meaning of the notation $\operatorname{diag}_B(\cdot)$ and $R_B$.
Here and in other parts of this paper, given two matrices $P$ and 
$Q$ with the same number of 
columns, the notation $[P;Q]$ denotes the matrix that is obtained by appending $Q$ to the 
last row of $P$. The set $\M_r$ has many favorable properties so that the corresponding 
projection and retraction can be computed efficiently. Although this has been discussed in detail in \cite{tang2024solving,RiNNAL}, for reader's convenience, we summarize these properties below: 
\begin{enumerate}
    \item $\M_r$ only contains only $mr$ linear constraints and $p$ 
    spherical constraints, making it easier to handle than the original one 
    ($\mathcal{F}\cap\mathbb{S}^{n+1}_+$) consisting of a positive semidefinite constraint and $m(n+1)+p+1$ equality constraints.
    \item Reformulating constraints into $\M_r$ helps mitigate the violation of Slater’s condition for the primal \eqref{SDP-RLT} problem.
    \item The metric projection onto the algebraic variety $\M_r$, although non-convex, can be transformed into a tractable convex optimization problem under the LICQ condition.
\end{enumerate}

Based on the good geometric properties of $\M_r$, RiNNAL applies the Riemannian gradient descent (RGD) method to solve \eqref{LR}. However, as mentioned in the previous subsection, the reformulation technique to ensure LICQ property will increase the dimension and reduce the algorithm's efficiency. To overcome this issue, in RiNNAL+, we use a random perturbation strategy to directly achieve smoothness without increasing the dimensionality. This approach, initially studied in \cite{tang2023feasible}, is detailed further in Subsection \ref{subsec-perturnation}. The low-rank phase plays a central role in reducing the objective value of \eqref{CVX}.

\medskip

%%%%%%%%%%%%%%%%%%%%%%%%%%%%%%%%%%%%%%%%%%%%%%%%%%%%%%%%%%%%%%%%%%%%%%%%%%%%%%%

\noindent\tb{Convex lifting phase.} 
Once the iterate $R_t$ in the low-rank phase reaches near-stationarity, the algorithm transitions to the convex lifting phase. In this phase, we perform a projected gradient (PG) step on \eqref{CVX}, initializing from $Y_t = \widehat{R}_t \widehat{R}^\top_t$. The PG step employs a semi-smooth Newton (SSN) method to compute the next iterate $Y_{t+1}$. After that, we factorize $Y_{t+1}$ to get $R_{t+1}$,  which serves as the starting point for the next low-rank phase. This convex lifting phase has two advantages:
\begin{enumerate}
    \item Unlike the rank-tuning strategies discussed in the previous subsection, the PG step automatically updates the rank without requiring parameter tuning. In our numerical experiments, we observe that this approach performs remarkably well and typically identifies the correct rank after just a few PG steps. For instance, we test this on a BIQ problem with dimension $n = 500$. As shown in Figure~\ref{fig-rank}, the evolution of the rank demonstrates that the PG method converges to the correct rank significantly faster than traditional rank-tuning strategies.
    \begin{figure}[h!]
    \centering
    % First subfigure
    \begin{subfigure}{0.49\textwidth}
        \centering
        \includegraphics[width=\linewidth]{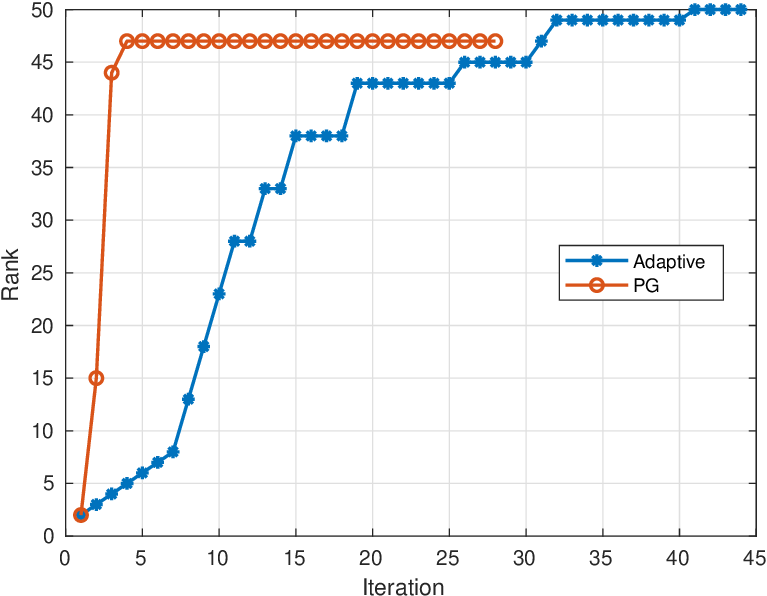}
        \caption{Initial rank $r=2$.}
        \label{fig:pp-QKP}
    \end{subfigure}
    \hfill
    % Second subfigure
    \begin{subfigure}{0.49\textwidth}
        \centering
        \includegraphics[width=\linewidth]{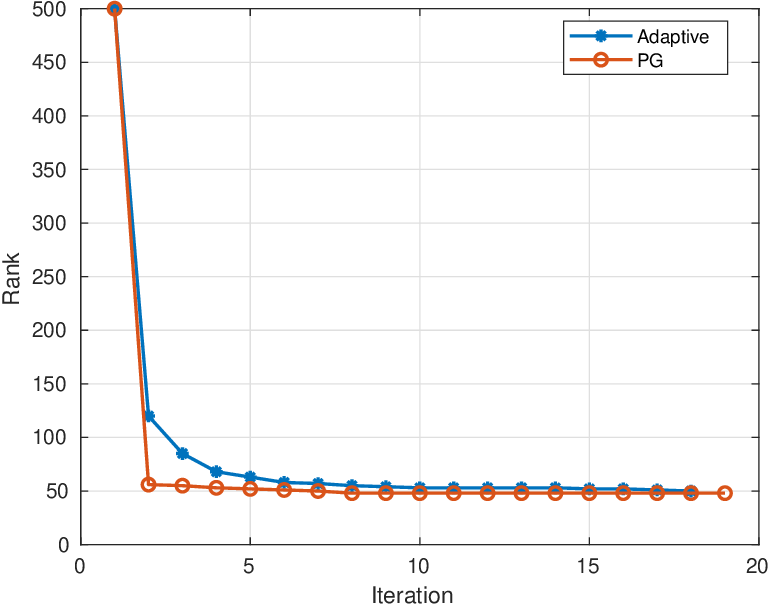}
        \caption{Initial rank $r=500$.}
        \label{fig:pp-BIQ}
    \end{subfigure}
    \caption{Comparison of rank evolution between PG and traditional rank-tuning strategies.}
    \label{fig-rank}
    \end{figure}
    \item The PG step consistently decreases the function value, whereas the rank-tuning method may increase the function value when we truncate small singular values. This monotonic decrease ensures the convergence of the subproblem.
\end{enumerate}
The idea of using a PG step to update rank has been explored in the earlier work \cite{lee2022escaping}. In this paper, we further develop a preprocessing technique, which will be introduced in subsection \ref{subsec-preprocess}, to significantly reduce the computational cost of the PG step. 
This preprocessing step is particularly beneficial for SDP problems with constraints $Y\in \F\cap\S^{n+1}_+$, as it enables the metric projection onto the feasible region in the PG step to be efficiently solved via the SNN method with significantly fewer iterations. Since the set $\F\cap\S^{n+1}_+$ is widely encountered in SDP relaxations with RLT constraints, this technique can be used as a subroutine in various solvers to efficiently handle such constraints.
Beyond preprocessing, we also develop a warm-start technique, which will be discussed in subsection~\ref{subsec-recover-dual}. This warm-start technique recovers the dual variable from the low-rank phase and uses it as the initial value for the PG step in the convex lifting phase. This warm-start technique can substantially reduce the time required to solve the projection subproblem, further improving the efficiency of RiNNAL+.

%%%%%%%%%%%%%%%%%%%%%%%%%%%%%%%%%%%%%%%%%%%%%%%%%%%%%%%%%%%%%%%%%%%%%%%%%%%%%%%

\subsection{Summary of our contributions}

Our contributions in this paper are summarized as follows:
\begin{enumerate}
    \item We provide a comprehensive comparison of the bound tightness, constraint type, and numerical behavior among several commonly used SDP relaxations of \eqref{MBQP}, namely, \eqref{SHOR}, \eqref{SDP-RLT}, \eqref{DNN} and \eqref{COMP}. In particular, we establish the theoretical equivalence between the \eqref{SDP-RLT} and \eqref{DNN} relaxations.
    \item We introduce RiNNAL+,  a Riemannian ALM to solve general SDP relaxation problems \eqref{prob-DNN-general}, which encompasses \eqref{SDP-RLT}, \eqref{DNN}, and \eqref{COMP} as special cases. 
    Our approach employs a hybrid two-phase framework to efficiently solve the ALM subproblem: the low-rank phase reduces the objective value, while the convex lifting phase automatically adjusts the rank of iterates to reduce computational costs and ensure global convergence.
    \item Unlike prior methods that require reformulating \eqref{SDP-RLT} into a higher-dimensional SDP problem \cite{RiNNAL}, we circumvent nonsmoothness issues by introducing a small random perturbation to the constraints in $\M_r$, thereby improving computational efficiency without increasing dimensionality.
    \item In contrast to existing low-rank algorithms that adaptively adjust the solution rank \cite{wang2023decomposition, wang2023solving, tang2024feasible,RiNNAL}, we employ a single PG step to automatically tune the rank for the low-rank phase. This strategy reduces variable dimensionality, helps escape saddle points, and ensures global convergence. Additionally, we enhance the efficiency of the PG step through specially designed preprocessing and warm-start techniques.
    \item We conduct numerous numerical experiments to evaluate the performance of our RiNNAL+ algorithm for solving SDP relaxations of various MBQP problem classes. While previous SDP solvers based on low-rank factorization perform well only when the problem has a low-rank optimal solution, RiNNAL+ demonstrates strong performance even for \eqref{SDP-RLT} and \eqref{DNN} relaxations whose optimal solution ranks are not necessarily small.
\end{enumerate}

%%%%%%%%%%%%%%%%%%%%%%%%%%%%%%%%%%%%%%%%%%%%%%%%%%%%%%%%%%%%%%%%%%%%%%%%%%%%%%%

\subsection{Organization}

This paper is structured as follows.
Section~\ref{sec-relaxations} introduces several SDP relaxations of MBQP, including \eqref{SHOR}, \eqref{SDP-RLT}, and \eqref{DNN}, and establishes the theoretical equivalence between the last two relaxations.
Section~\ref{sec-alg} presents RiNNAL+, a hybrid augmented Lagrangian method designed to solve general SDP relaxations of QCQP. 
Section~\ref{sec-acceleration} introduces computational enhancements, such as preprocessing, random perturbations, and warm-start techniques, to improve the efficiency and stability of RiNNAL+.
Section~\ref{sec-experiments} provides extensive numerical experiments to evaluate the performance of RiNNAL+.
Finally, we conclude the paper in Section~\ref{sec-conclusion}.

%%%%%%%%%%%%%%%%%%%%%%%%%%%%%%%%%%%%%%%%%%%%%%%%%%%%%%%%%%%%%%%%%%%%%%%%%%%%%%%

\subsection{Notations}\label{subsec-notations}

Let $\langle A, B\rangle := \operatorname{Tr}\left(A B^{\top}\right)$ denote the matrix inner product and $\|\cdot\|$ be its induced Frobenius norm in $\mathbb{S}^n$. Define $e$ as a column vector of all ones, and $e_1$ as a column vector with 1 as its first entry and zero otherwise. 
Let $\R^n_+$ and $\R^n_{++}$ denote the sets of nonnegative and positive real vectors in $\R^n$, respectively, and let $\Pi_+(\cdot)$ be the projection onto $\R^n_+$.
For a matrix $X \in \mathbb{R}^{m \times n}$, $\operatorname{vec}(X)$ denotes the vector in $\mathbb{R}^{mn}$ formed by stacking the columns of $X$. 
We use $\circ$ to denote the element-wise multiplication operation between two vectors/matrices of the same size. 
We use $\delta_{\mathcal{C}}(\cdot)$ to denote the indicator function of a set $\mathcal{C}$. Let $[n] := \{1,2, \ldots, n\}$ for any positive integer $n$. For a matrix $X\in\S^{n+1}$, we denote its block decomposition as follows:
\begin{equation}\label{notation-block}
    X=\begin{bmatrix}
        X_{11}&X_{12}\\
        X_{21}&X_{22}
    \end{bmatrix}\in \begin{bmatrix}
        \mathbb{R}&\mathbb{R}^{1\times n}\\
        \mathbb{R}^{n\times 1}&\S^{n}
    \end{bmatrix}.
\end{equation}
Next, we define some operators. Given an index set $B \subseteq[n]$ with its cardinality denoted by $|B|$, define $\operatorname{diag}_B: \mathbb{R}^{n \times n} \rightarrow \mathbb{R}^{|B|}$ such that $\operatorname{diag}_B(X)=\left(X_{ii}\right)_{i \in B}$. 
The index set $B$ is omitted if $B=[n]$. 
For a matrix $R \in \mathbb{R}^{n \times r}$, let $R_i\in\R^{1\times r}$ denote its $i$-th row, and let $R_B \in \mathbb{R}^{|B| \times r}$ be the submatrix of $R$ consisting of the rows indexed by $B$. Define $\widehat R:= (e^\top_1;R)$, which augments $R$ with the first standard basis row vector.

%%%%%%%%%%%%%%%%%%%%%%%%%%%%%%%%%%%%%%%%%%%%%%%%%%%%%%%%%%%%%%%%%%%%%%%%%%%%%%%
%%%%%%%%%%%%%%%%%%%%%%%%%%%%%%%%%%%%%%%%%%%%%%%%%%%%%%%%%%%%%%%%%%%%%%%%%%%%%%%

\section{Relaxations}\label{sec-relaxations}

In this section, we first introduce three commonly used SDP-type relaxations: the Shor relaxation, the SDP-RLT relaxation, and the DNN relaxation in subsection \ref{subsec-relaxation}. 
We then demonstrate that the latter two relaxations are actually equivalent 
in subsection \ref{Tightness comparison}. Finally, we apply these findings to two significant problem classes: the strengthened MBQP and sparse quadratic programming (QP) with an $l_0$-norm constraint
in subsection \ref{subsec-application}.

%%%%%%%%%%%%%%%%%%%%%%%%%%%%%%%%%%%%%%%%%%%%%%%%%%%%%%%%%%%%%%%%%%%%%%%%%%%%%%%

\subsection{Relaxation formulation}
\label{subsec-relaxation}

\noindent\tb{Shor relaxation.}
By replacing the quadratic term $xx^\top$ with a matrix $X$ and introducing a positive semidefinite constraint $X-xx^\top\succeq 0$, we obtain the Shor relaxation \cite{shor1990dual} of \eqref{MBQP} as follows:
\begin{equation}\label{SHOR}
    v^{\text{SHOR}}:=\min\left\{  
    \<C,Y\> :\  
    Y\in\F_S\cap\I_S\cap\mathbb{S}_{+}^{n+1}
    \right\},\tag{SHOR}
\end{equation}
where $C$ is defined after \eqref{SDP-RLT}, $\F_S$ and $\I_{S}$ are defined by
\begin{align*}
\F_S&:= \left\{\begin{bmatrix}z& x^\top \\ x & X\end{bmatrix}\in  \mathbb{S}^{n+1} :\  Ax=b,\ x_i =X_{ii}, \  \forall i \in B,\ z=1\right\},\\
\I_S&:= \left\{\begin{bmatrix}z& x^\top \\ x & X\end{bmatrix}\in  \mathbb{S}^{n+1} :\ zd-Gx\geq 0,\ x\geq 0 \right\}.
\end{align*}
The Shor relaxation provides a lower bound $v^{\text{SHOR}}$ of the optimal value $v^{P_1}$ of \eqref{MBQP}. However, the gap between $v^{\text{SHOR}}$ and $v^{P_1}$ may be large, and in certain cases, the lower bound $v^{\text{SHOR}}$ can even be unbounded below. Thus, we need to add extra constraints to make the relaxation tighter.

\bigskip

%%%%%%%%%%%%%%%%%%%%%%%%%%%%%%%%%%%%%%%%%%%%%%%%%%%%%%%%%%%%%%%%%%%%%%%%%%%%%%%

\noindent\tb{SDP-RLT relaxation.}
The SDP-RLT relaxation integrates the reformulation-linearization technique  (RLT) \cite{sherali2007rlt} with Shor relaxation \eqref{SHOR}, hence providing a tighter bound. The RLT technique generates additional quadratic constraints implied by the linear constraints of \eqref{MBQP}. The explicit formulation \eqref{SDP-RLT} is provided in Subection~\ref{subsec-SDPRLT}.

\bigskip

%%%%%%%%%%%%%%%%%%%%%%%%%%%%%%%%%%%%%%%%%%%%%%%%%%%%%%%

\noindent\tb{DNN relaxation.}
To derive the DNN relaxation of \eqref{MBQP}, we first convert all the inequality constraints in \eqref{MBQP} into equality constraints by introducing nonnegative slack variables. The resulting problem is
\begin{equation*}
\min\left\{x^{\top} Q x+2 c^{\top}x : \ Ax=b,\ Gx+s= d,\ x_i \in\{0,1\},\ \forall i \in B,\ x\in \R^{n}_+,\ s\in  \R^{l}_+\right\},
\end{equation*}
which can be expressed in the following compact form:
\begin{equation}\label{prob-MBQP-2}
%\label{prob-MBQP-compact}
v^{P_2}:=\min\left\{x'^{\top} Q' x'+2 c'^{\top}x' : \ A'x'=b',\ x'_i \in\{0,1\},\ \forall i \in B,\ x'\in \R^{n+l}_+\right\},
\tag{$P_2$}
\end{equation}
where
\[
Q':=\begin{bmatrix}
Q & 0_{n\times l}\\ 
0_{l\times n} & 0_{l\times l}
\end{bmatrix},\
A':=\begin{bmatrix}
A&0_{m\times l}\\
 G&I_l
\end{bmatrix},\ 
c':=\begin{bmatrix}
c\\
0_{l\times 1}
\end{bmatrix},\
b':=\begin{bmatrix}
b\\
d
\end{bmatrix},\ 
x':=\begin{bmatrix}
x\\
s
\end{bmatrix}.
\]
Then the DNN relaxation of \eqref{prob-MBQP-2} is given by  
\begin{equation}\label{DNN}
v^{\text{DNN}}:=\min\left\{
\<C',Y'\>  :\  \
Y'\in \F_D\cap \mathbb{S}_{+}^{n+l+1} \cap\mathbb{N}^{n+l+1}
\right\},\tag{DNN}
\end{equation}
where $C':=[0,(c')^\top; c', Q']$, and $\F_D$ is defined as
\begin{equation*}
\begin{aligned}
\F_D &:= \left\{\begin{bmatrix}
    z'&(x')^\top\\
    x'&X'
\end{bmatrix}\in  \mathbb{S}^{n+l+1} :\ 
A' x'  = b',\
A' X' = b'(x')^\top,\ \operatorname{diag}_B(X')=x'_B,\ z'=1
\right\}.
\end{aligned}
\end{equation*}
There are several equivalent reformulations of $\F_D$, but we adopt this particular form because it preserves the sparsity structure of the constraint matrices and has the smallest duality gap among other DNN reformulations, see \cite{RiNNAL,bomze2017fresh} for more detailed explanations. \eqref{DNN} can be also viewed as the SDP-RLT relaxation of \eqref{prob-MBQP-2}.

%%%%%%%%%%%%%%%%%%%%%%%%%%%%%%%%%%%%%%%%%%%%%%%%%%%%%%%%%%%%%%%%%%%%%%%%%%%%%%%

\subsection{Tightness comparison}
\label{Tightness comparison}

In this subsection, we compare the tightness of the lower bounds provided by the three different relaxations: \eqref{SHOR}, \eqref{SDP-RLT}, and \eqref{DNN}. Specifically, we prove in Theorem~\ref{thm-1} that \eqref{SDP-RLT} and \eqref{DNN} provide the same bound, which is tighter than that of \eqref{SHOR}. To prove this result, we first denote
\begin{equation*}
\begin{aligned}
{\mathrm{F}_R}&:=\F\cap\I_R\cap \mathbb{S}_{+}^{n+1} \cap\mathbb{N}^{n+1}&&\subseteq\S^{n+1},\\
{\mathrm{F}_D}&:=\F_D\cap \mathbb{S}_{+}^{n+l+1} \cap\mathbb{N}^{n+l+1}&&\subseteq\S^{n+l+1},
\end{aligned}
\end{equation*}
which are the feasible regions of \eqref{SDP-RLT} and \eqref{DNN}, respectively. Next, define the linear map $\Phi:\S^{n+1}\rightarrow\S^{n+l+1}$ such that for any $Y=\begin{bmatrix}z& x^\top \\ x & X\end{bmatrix}\in\S^{n+1}$, it holds that
     {\small
     \begin{equation}
     \label{phi}
     \begin{aligned}
     \Phi\left(Y\right):=\begin{bmatrix}
             1&0_{1\times n}\\
             0_{n\times 1}&I_n\\
             d&-G
         \end{bmatrix}Y\begin{bmatrix}
             1&0_{1\times n}\\
             0_{n\times 1}&I_n\\
             d&-G
         \end{bmatrix}^\top
         =\begin{bmatrix}
             z&x^\top&(zd-Gx)^\top\\
             x&X&(dx^\top-GX)^\top\\
             zd-Gx&dx^\top-GX&\G(Y)
         \end{bmatrix},
     \end{aligned}
     \end{equation}
     }\normalsize
where the mapping $\G$ is defined in \eqref{def-G}.
The following lemma characterizes the one-to-one correspondence between ${\mathrm{F}_R}$ and ${\mathrm{F}_D}$.
\begin{lemma}\label{lemma-F1-F2}
The mapping $\Phi$ is bijective from ${\mathrm{F}_R}$ to ${\mathrm{F}_D}$.
\end{lemma}
\begin{proof}
It is sufficient to prove that:
\begin{enumerate}
\item[(1)] For any $Y:=\begin{bmatrix}
z&x^\top\\
x&X
\end{bmatrix}\in{\mathrm{F}_R}$, it holds that $\Phi(Y)\in {\mathrm{F}_D}$.
\item[(2)] $\Phi$ is surjective.
\item[(3)] $\Phi$ is injective.
\end{enumerate}
Proof of (1): assume that $Y\in{\mathrm{F}_R}=\F\cap\I_R\cap \mathbb{S}_{+}^{n+1} \cap\mathbb{N}^{n+1}$. Since $z=1$, we have that  
\[
\Phi(Y):=\begin{bmatrix}
1&x^\top&s^\top\\
x&X&Z^\top\\
s&Z&W
\end{bmatrix}=\begin{bmatrix}
             1&x^\top&(d-Gx)^\top\\
             x&X&(dx^\top-GX)^\top\\
             d-Gx&dx^\top-GX&GXG^\top -Gxd^\top -dx^\top G^\top +dd^\top
         \end{bmatrix}.
\]
First, by the definition of $\Phi(Y)$ in \eqref{phi}, $Y\succeq 0$ implies that $\Phi(Y)\succeq 0$. Next, the relation $Y\in\I_R$ implies that $s\geq 0,\ Z\geq 0,\ W\geq 0$. Also, the relation $Y\in \mathbb{N}^{n+1}$ implies that $x\geq 0,\ X\geq 0$. Thus, we have $\Phi(Y)\in \bN^{n+l+1}$. Finally, denote $x':=[x;s]$. 
Then
\begin{equation*}
\begin{aligned}
A'x'
=\begin{bmatrix}
A&0_{m\times l}\\
 G&I_l
\end{bmatrix}\begin{bmatrix}
x\\s
\end{bmatrix}
=\begin{bmatrix}
Ax\\
Gx+s
\end{bmatrix}
=\begin{bmatrix}
b\\
d
\end{bmatrix}
=b', 
\end{aligned}
\end{equation*}
where the third equality follows from the definition of $s$. Similarly,
\[
A'X'
=\begin{bmatrix}
A&0_{m\times l}\\
 G&I_l
\end{bmatrix}\begin{bmatrix}
X&Z^\top\\
Z&W
\end{bmatrix}
=\begin{bmatrix}
AX&AZ^\top\\
GX+Z&GZ^\top+W
\end{bmatrix}
=\begin{bmatrix}
bx^\top&bs^\top\\
dx^\top&ds^\top
\end{bmatrix}
=b'(x')^\top,
\]
where the third equality follows from the definitions of $s$, $Z$ and $W$. Furthermore, 
\[
\operatorname{diag}_B(X')=\operatorname{diag}_B(X)=x_B=x'_B.
\] 
Thus, we have $\Phi(Y)\in \F_D$. According to the relations above, we proved that $\Phi(Y)\in{\mathrm{F}_D}$.
\smallskip

\noindent Proof of (2): we prove that for any $Y'\in{\mathrm{F}_D}$, there exists some $Y\in{\mathrm{F}_R}$ such that $\Phi(Y)=Y'$. Denote
\[
Y':=\begin{bmatrix}
1&x^\top&s^\top\\
x&X&Z^\top\\
s&Z&W
\end{bmatrix}\in {\mathrm{F}_D},
\] 
then it is sufficient to prove that the submatrix $Y_0:=[1, x^\top; x, X] \in{\mathrm{F}_R}$ and it satisfies $\Phi(Y_0)=Y'$. By the relation $Y'\in\F_D$, we have 
\begin{equation}\label{eq-Y'-F'}
\begin{aligned}
\begin{bmatrix}
Ax\\
Gx+s
\end{bmatrix}
=\begin{bmatrix}
b\\
d
\end{bmatrix}
,\quad \begin{bmatrix}
AX&AZ^\top\\
GX+Z&GZ^\top+W
\end{bmatrix}
=\begin{bmatrix}
bx^\top&bs^\top\\
dx^\top&ds^\top
\end{bmatrix},
\end{aligned}
\end{equation}
which implies that
\[
\Phi(Y_0)=\begin{bmatrix}
             1&x^\top&(d-Gx)^\top\\
             x&X&(dx^\top-GX)^\top\\
             d-Gx&dx^\top-GX&GXG^\top -Gxd^\top -dx^\top G^\top +dd^\top
         \end{bmatrix}
=\begin{bmatrix}
1&x^\top&s^\top\\
x&X&Z^\top\\
s&Z&W
\end{bmatrix}=Y'.
\]
Thus, we only need to show that $Y_0\in{\mathrm{F}_R}$.
First, the relation $Y_0\in\S^{n+1}_+\cap\mathbb{N}^{n+1}$ holds because $Y'\in\S^{n+l+1}_+\cap\mathbb{N}^{n+l+1}$. Second, denote $x':=[x;s]$, the relation $Y_0\in\F$ follows from the first $m$ rows of \eqref{eq-Y'-F'} and 
\[
\operatorname{diag}_B(X)=\operatorname{diag}_B(X')=x'_B=x_B.
\]
Finally, the relation $Y_0\in\I_R$ is equivalent to $s\geq0, Z\geq 0, W\geq 0$, which holds because $Y'\geq 0$. Thus, we have $Y_0\in{\mathrm{F}_R}$.
\smallskip

\noindent Proof of (3): we prove that if $\Phi(Y_1)=\Phi(Y_2)$ for any $Y_1,Y_2\in{\mathrm{F}_R}$, then $Y_1=Y_2$.
This can be directly proven by noting that $Y_1$ and $Y_2$ are the submatrices of $\Phi(Y_1)$ and $\Phi(Y_2)$.
\end{proof}
With Lemma \ref{lemma-F1-F2}, we can prove the equivalence between \eqref{SDP-RLT}, and \eqref{DNN}.
\begin{theorem}\label{thm-1} The optimal values of different reformulations and relaxations of \eqref{MBQP} satisfy
    \begin{equation*}
    v^{\text{SHOR}} \leq 
    v^{\text{DNN}} =
    v^{\text{SDP-RLT}} \leq 
    v^{P_1} =
    v^{P_2}.
    \end{equation*}
\end{theorem} 
\begin{proof}
Since \eqref{MBQP} and \eqref{prob-MBQP-2} differ only by the slack variables, it follows that $v^{P_1}=v^{P_2}$. Additionally, since \eqref{SHOR} shares the same objective function as \eqref{SDP-RLT} but has a larger feasible region, we have $v^{\text{SHOR}}\leq v^{\text{SDP-RLT}}$. Since \eqref{DNN} is a relaxation of \eqref{prob-MBQP-2}, we have $v^{\text{DNN}}\leq v^{P_2}$. Thus, we only need to prove that $v^{\text{SDP-RLT}} = v^{\text{DNN}}$. By Lemma \ref{lemma-F1-F2}, there exists a bijective mapping $\Phi$ from the feasible set ${\mathrm{F}_R}$ of \eqref{SDP-RLT} to the feasible set ${\mathrm{F}_D}$ of \eqref{DNN}. Furthermore, for any pair $Y\in{\mathrm{F}_R}$ and $Y'\in{\mathrm{F}_D}$ such that $\Phi(Y)=Y'$, the objective function values of $\eqref{DNN}$ and $\eqref{SDP-RLT}$ are the same, i.e., $\<C',Y'\>=\<C,Y\>$. Thus, we have $v^{\text{SDP-RLT}} =v^{\text{DNN}}$.
\end{proof}
\begin{remark}[Applications]
The equivalence outlined in Theorem \ref{thm-1} between \eqref{SDP-RLT} and \eqref{DNN} generalizes several established results in the literature. For example, \cite{bao2011semidefinite,sherali2013reformulation} establish this equivalence for \eqref{MBQP} restricted to box constraints. \cite{anstreicher2012convex,bomze2019notoriously} mentioned similar equivalence results under the mapping $\Phi$, but their DNN relaxation does not incorporate slack variables, making it different from \eqref{DNN}. Theorem \ref{thm-1} also serves as an efficient tool to directly verify the equivalence between different relaxations. For example, we can directly confirm the equivalence between the DNN relaxation of a sparse standard QP and its reduced-dimensional form demonstrated in \cite[Section 3]{bomze2024tighter}. Similarly, we can confirm the equivalence between the DNN relaxation of a concave QP and its reduced-dimensional form demonstrated in \cite[subsection 2.2]{qu2023globally}.
\end{remark}
\begin{remark}[Tightness]
    Numerical examples, presented in Subsection \ref{subsec-tightness-numerical} and discussed in \cite[Section 6]{kim2022doubly}, demonstrate that the inequality $v^{\text{SHOR}} \leq v^{\text{SDP-RLT}}$ may indeed be strict. Additionally, Subsection \ref{subsec-tightness-numerical} further compares the relative gap of these different relaxations.
\end{remark}
\begin{remark}[Constraints]
Theorem \ref{thm-1} indicates that \eqref{SDP-RLT} and \eqref{DNN} are theoretically equivalent. However, these two relaxations differ in variable dimensions, as well as the type and number of constraints, which will affect the numerical performance of algorithms used to solve these relaxations. As shown in Table \ref{tab-comparison}, although \eqref{SDP-RLT} has a smaller variable dimension than \eqref{DNN}, it contains more inequality constraints, especially when the number of inequalities of \eqref{MBQP} is large. The numerical performance comparison of RiNNAL+ and SDPNAL+ for different relaxation formulations is further studied in subsection \ref{subsec-pp}.
\end{remark}	

\begin{center}
\begin{table}[h!]
\begin{center}
    \def\arraystretch{1.5}
    \setlength{\tabcolsep}{5pt} 
    \begin{tabular}{|c|c|c|c|} 
     \hline
    \multicolumn{1}{|c|}{Problem} & \# conic constraints  & \# equality constraints & \# inequality constraints   \\ 
     \hline
     \eqref{SHOR}  & $\S_+^{n+1}$   &   $m+p+1$ & $l+n$ \\ 
     \hline
     \eqref{SDP-RLT}  & $\S_+^{n+1}\cap\bN^{n+1}$   &   $m(n+1)+p+1$ & ${l(2n+l+1)}/{2}$ \\ 
     \hline
     \eqref{DNN} & $\S_+^{n+l+1}\cap\bN^{n+l+1}$  &  $(m+l)(n+l+1)+p+1$ &  $0$ \\ 
     \hline    
      \multicolumn{4}{|c|}{$n$: variable dimension\quad $p$: binary variables\quad $m$: equalities\quad $l$: inequalities}\\
     \hline
     \end{tabular}
    \caption{Comparison of constraints for different formulations.}
    \label{tab-comparison}
\end{center}
\end{table}
\end{center}

%%%%%%%%%%%%%%%%%%%%%%%%%%%%%%%%%%%%%%%%%%%%%%%%%%%%%%%%%%%%%%%%%%%%%%%%%%%%%%%

\subsection{Examples}\label{subsec-application}

In this subsection, we highlight two significant classes of \eqref{MBQP}
derived from the strengthened equality-constrained MBQP 
and sparse QP with $l_0$-norm constraint. 
\begin{example}[Strengthened equality-constrained MBQP]\label{example-SMBQP}
Consider the following equality-constrained MBQP without inequality constraints:
\begin{equation}\label{prob-MBQP-eq}
\min\left\{x^{\top} Q x+2 c^{\top}x : \ Ax=b,\ x_B\in\{0,1\}^{p},\ x\in \R^{n}_+\right\}.
\tag{MBQP-E}
\end{equation}
The DNN and SDP-RLT relaxations of \eqref{prob-MBQP-eq} are the same and have been extensively analyzed in prior works \cite{RiNNAL,bomze2017fresh,burer2010optimizing}. However, they may still lack sufficient tightness. To further strengthen these relaxations, we can introduce redundant constraints based on the binary condition $x_B\in\{0,1\}^p$ \cite{padberg1989boolean,locatelli2024fix}. Specifically, we add the redundant bound constraint $x_B\leq e$ to \eqref{prob-MBQP-eq}, leading to the following equivalent problem:
\begin{equation}\label{SMBQP}
v^{P_1}:=\min\left\{x^{\top} Q x+2 c^{\top}x : \ Ax=b,\ x_B\leq e,\ x_B\in\{0,1\}^{p},\ x\in \R^{n}_+\right\}.\tag{SMBQP-E}
\end{equation}
Although \eqref{prob-MBQP-eq} and \eqref{SMBQP} are equivalent, the SDP-RLT relaxation of the latter is expected to be tighter due to the additional bound constraint $ x_B \leq e $. This is confirmed in subsection~\ref{subsec-tightness-numerical}, where we compare the SDP-RLT lower bounds of both formulations and demonstrate that incorporating this redundant constraint significantly strengthens the SDP-RLT relaxations.
To obtain the DNN relaxation of \eqref{SMBQP}, we introduce a slack variable $s$ and reformulate \eqref{SMBQP} as follows:
\begin{equation}\label{SMBQP-slack}
v^{P_2}:=\min\left\{x^{\top} Q x+2 c^{\top}x : \ Ax=b,\ x_B+s= e,\ x_B\in\{0,1\}^{p},\ x\in \R^{n}_+,\ s\in\R_+^{p}\right\},
\end{equation}
which aligns with the structure of \eqref{prob-MBQP-2}, enabling us to obtain the corresponding DNN relaxation. Furthermore, we can replace the binary constraints $x_B\in\{0,1\}^{p}$ in \eqref{SMBQP-slack} by the complementarity constraints $x_B\circ s=0$, leading to another formulation as follows:
\begin{equation}\label{SMBQP-comp}
v^{P_3}:=\min\left\{x^{\top} Q x+2 c^{\top}x : \ Ax=b,\ x_B+s= e,\ x_B\circ s=0,\ x\in \R^{n}_+,\ s\in\R_+^{p}\right\}.
\tag{$P_3$}
\end{equation}
This formulation is in the form of QCQP, differing from the structure of \eqref{MBQP} due to the additional complementarity constraints $x_B \circ s = 0$. However, as discussed in Remark \ref{remark-QCQP}, we can still derive the SDP-RLT relaxation as follows:
\begin{equation}\label{COMP}
v^{\text{COMP}}:=\min\left\{
\<C',Y'\>  :\  \
Y'\in \F_C\cap\E_C\cap \mathbb{S}_{+}^{n+p+1} \cap\mathbb{N}^{n+p+1}
\right\},\tag{COMP}
\end{equation}
where the sets $\F_C$ and $\I_C$ are defined as 
\begin{equation*}
\begin{aligned}
\F_C &:= \left\{
\begin{bmatrix}
    z'&(x')^\top\\
    x'&X'
\end{bmatrix}\in  \mathbb{S}^{n+p+1} :\ 
A' x'  = b',\
A' X' = b'(x')^\top,\ z'=1
\right\}\\
\E_C &:= \left\{\begin{bmatrix}
    z'&(x')^\top\\
    x'&X'
\end{bmatrix}\in\mathbb{S}^{n+p+1}  :\ X'_{ij}=0\ \text{for}\ i\in B,\ j\in \{n+1,\dots,n+p\} \right\}
.
\end{aligned}
\end{equation*}
Denote the optimal values of the Shor, SDP-RLT and DNN relaxations of \eqref{SMBQP} as $v^{SHOR}$, $v^{SDP-RLT}$ and $v^{DNN}$, respectively. The following result shows that the three different relaxations are all equivalent.
\begin{theorem}\label{thm-tightness-SMBQP}
The optimal values of different reformulations and relaxations of \eqref{SMBQP} satisfy
    \begin{equation*}
    v^{\text{SHOR}} \leq 
    v^{\text{SDP-RLT}} =
    v^{\text{DNN}} = 
    v^{\text{COMP}} \leq 
    v^{P_1} =
    v^{P_2} = 
    v^{P_3}.
    \end{equation*}
\end{theorem}
\begin{proof}
    Since \eqref{SMBQP}, \eqref{SMBQP-slack} and \eqref{SMBQP-comp} are equivalent, we have $v^{P1}=v^{P2}=v^{P3}$. Also, it is proven in \cite{ito2018equivalences} that $v^{DNN}=v^{COMP}$. Combined with Theorem \ref{thm-1}, we complete the proof.
\end{proof}
\begin{remark}
Table \ref{tab-comparison-SMBQP} presents a comparison of the constraints across the four relaxations of \eqref{SMBQP}. Notably, while \eqref{DNN} and \eqref{COMP} share the same variable dimensions and number of constraints, the nature of these constraints differs, which has a substantial impact on the performance of our proposed algorithm. A detailed performance comparison is provided in subsection \ref{subsec-pp}, demonstrating that \eqref{SDP-RLT} is the preferred choice for our algorithm, as compared to either \eqref{DNN} or \eqref{COMP}. Moreover, \eqref{DNN} is preferred over \eqref{COMP}.
\end{remark}

\begin{center}
\begin{table}[h!]
\begin{center}
    \def\arraystretch{1.5}
    \begin{tabular}{|c|c|c|c|} 
     \hline
    \multicolumn{1}{|c|}{Problem} & \# conic constraints  & \# equality constraints & \# inequality constraints   \\ 
     \hline
     \eqref{SHOR}  & $\S_+^{n+1}$   &   $m+p+1$ & $p+n$ \\ 
     \hline
     \eqref{SDP-RLT}  & $\S_+^{n+1}\cap\bN^{n+1}$   &   $m(n+1)+p+1$ & ${p(2n+p+1)}/{2}$ \\ 
     \hline
     \eqref{DNN} & $\S_+^{n+p+1}\cap\bN^{n+p+1}$  &  $(m+p)(n+p+1)+p+1$ &  $0$ \\ 
     \hline
     \eqref{COMP} & $\S_+^{n+p+1}\cap\bN^{n+p+1}$  &  $(m+p)(n+p+1)+p+1$ &  $0$  \\ 
     \hline    
      \multicolumn{4}{|c|}{$n$: variable dimension\quad $p$: binary variables\quad $m$: equalities}\\
      \hline
     \end{tabular}
    \caption{Comparison of constraints for different relaxations of \eqref{SMBQP}.}
    \label{tab-comparison-SMBQP}
\end{center}
\end{table}
\end{center}

\end{example}

\begin{example}[$l_0$-norm constraint]\label{example-L0} 
Consider the following bounded set with $l_0$-norm constraint:
\begin{equation*}
S_1:=\{x\in\R_+^n:\ x\leq d,\|x\|_0\leq \rho \},
\end{equation*}
where $\rho \in (0, n]$ is a positive integer limiting the maximum number of nonzero elements, and $d \in \mathbb{R}_{++}^n$ provides a positive upper bound for the variable $x$. Without loss of generality, we assume $d = e$, as we can always scale the variable $x$. Constraints of the form $S_1$ are extensively studied in the literature; see, for instance, \cite{atamturk2021sparse,bomze2024tighter}. However, directly addressing optimization problems with the constraints in $S_1$  is typically challenging. An alternative approach is to introduce an upper-bound variable $u$ and consider the following Big-M set:
\begin{equation*}
S_2:=\{(x,u)\in\R_+^n\times \R_+^n:\ u\leq e,\ x\leq u,\ e^\top u=\rho,\ u\in \{0,1\}^n \}.
\end{equation*}
The constraint $u\leq e$ in $S_2$ is redundant, but helps to provide a tighter SDP-RLT relaxation. Besides the Big-M set, we can also use the variable $u$ to formulate the complementarity set: 
\begin{equation}\label{L0-comp}
S_3:=\{(x,u)\in\R_+^n\times \R_+^n:\  u\leq e,\ x\leq e,\ x^\top(u-e)=0,\ e^\top u=\rho,\ u\in \{0,1\}^n \}.
\end{equation}
Define the projection of a set $S$ onto the $x$-component as
\[
\Pi_x(S):=\{x\in \R^n:\ (x,u)\in S \text{ for some } u\in\R^n \}.
\]
Then the following relationship holds:
\[
S_1=\Pi_x(S_2)=\Pi_x(S_3). 
\]
Thus, for any sparse optimization problem with the constraint $x \in S_1$, we can replace $x \in S_1$ with $(x, u) \in S_2$. If the objective function is quadratic, the resulting problem then becomes a special case of \eqref{MBQP}, allowing all the established relaxations and algorithms to be applied directly.
Alternatively, we can substitute $x \in S_1$ with $(x, u) \in S_3$. This reformulation yields a special case of \eqref{MBQP} that includes additional quadratic constraints, specifically $x^\top(u - e) = 0$, which can be transformed into complementarity constraints by changing variables. Based on the discussion in Remark \ref{remark-QCQP}, we can extend the relaxations and algorithms to address the resulting QCQP problem, as demonstrated in subsection \ref{subsec-SStQP}.

\end{example}

%%%%%%%%%%%%%%%%%%%%%%%%%%%%%%%%%%%%%%%%%%%%%%%%%%%%%%%%%%%%%%%%%%%%%%%%%%%%%%%
%%%%%%%%%%%%%%%%%%%%%%%%%%%%%%%%%%%%%%%%%%%%%%%%%%%%%%%%%%%%%%%%%%%%%%%%%%%%%%%

\section{Algorithm framework}\label{sec-alg}

In this section, we give a detailed description of the proposed method RiNNAL+ to solve the following general SDP problem:
\begin{equation}\label{prob-DNN-general}
\min\left\{
\<C,Y\>  :\  \
Y\in \F\cap \E\cap\I\cap \mathbb{S}_{+}^{n+1}
\right\},\tag{$P$}
\end{equation}
 where $\F$, $\E$ and $\I$ are defined as
 \begin{equation*}
\begin{aligned}
 \F&:=\left\{Y\in\S^{n+1} :  \mathcal{A}(Y)= {d}  \right\},\,&&\quad\A:\S^{n+1}\rightarrow\R^{d_0},\\
\E&:=\left\{Y\in  \mathbb{S}^{n+1} : \,\B(Y)={g} \right\},\,&&\quad\B\,:\S^{n+1}\rightarrow\R^{d_1},\\
\I &:=\left\{Y\in  \mathbb{S}^{n+1} : \,\C(Y)\,\geq h \right\},&&\quad\C\,\,:\S^{n+1}\rightarrow\R^{d_2},
 \end{aligned}
 \end{equation*}
and $\mathcal{A}$, $\mathcal{B}$ and $\mathcal{C}$ are linear maps.
We focus on the specific $\F$ mentioned in \eqref{SDP-RLT} with
\begin{equation*}
    d_0=m+mn+p+1,\quad
    \mathcal{A}(Y) := \begin{bmatrix}
        AY_{21}\\
        \operatorname{vec}(AY_{22}-bY_{12})\\
        \diag_B(Y_{22})-(Y_{21})_B\\
        Y_{11}
    \end{bmatrix},
    \quad  {d}  := \begin{bmatrix}
        b\\
        0_{mn}\\
        0_{p}\\
        1
    \end{bmatrix}\in\R^{d_0}, 
    \end{equation*}
where $Y_{11},Y_{12},Y_{21},Y_{22}$ are the submatrices of $Y$ defined in \eqref{notation-block}. 
We reuse the notation $d$ in this section with a different meaning from its earlier usage. The distinction should be clear from the context.
Problem \eqref{prob-DNN-general} is general and encompasses the SDP relaxations proposed in this paper, as summarized in Table~\ref{tab-general}.

\begin{center}
\begin{table}[h!]
\begin{center}
    \def\arraystretch{1.5}
    \setlength{\tabcolsep}{5pt} 
    \begin{tabular}{|c|c|c|c|} 
     \hline
    \multicolumn{1}{|c|}{Problem} & $\F$ & $\E$ & $\I$   \\ 
     \hline
     \eqref{SHOR} & $ \F_S $   &   $\S^{n+1}$ & $ \I_S$ \\
     \hline  
     \eqref{SDP-RLT}  & $ \F $   &   $\S^{n+1}$ & $ \I_R \cap \mathbb{N}^{n+1}$ \\ 
     \hline
     \eqref{DNN}  & $\F_D$   &   $\S^{n+l+1}$ & $\mathbb{N}^{n+l+1}$ \\ 
     \hline
     \eqref{COMP} & $\F_C$   &   $\E_C$ & $\I_C\cap\mathbb{N}^{n+l+1}$ \\ 
     \hline    
     \end{tabular}
    \caption{Constraint correspondence to \eqref{prob-DNN-general} for different relaxations.}
    \label{tab-general}
\end{center}
\end{table}
\end{center}

The optimality conditions, also known as the Karush-Kuhn-Tucker (KKT) conditions, for \eqref{prob-DNN-general} are given by:
\begin{equation}\label{KKT}
\begin{aligned}
\text{(primal feasibility)}&\quad \A (Y)=d,\ \B (Y)= g,\ \C (Y)\geq h,\ Y\succeq 0,\\[2pt]
\text{(compementarity)}&\quad
\< \C (Y)-h, \mu \> = 0,\ \<Y,S\>=0,\\[2pt]
\text{(dual feasibility)}&\quad C - \A^*(y) - \B^*(\lambda)-\C^* (\mu) -S= 0,\ S\succeq 0,\ \mu\geq 0,
\end{aligned}
\end{equation}
where $(y,\lambda,\mu,S)\in \R^{d_0}\times \R^{d_1}\times \R^{d_2}\times \mathbb{S}^{n+1} $ are dual variables.
We make the following assumption throughout the paper. 
\begin{assumption}\label{assumption-mild}
    The problem \eqref{prob-DNN-general} admits an optimal solution satisfying the KKT conditions \eqref{KKT}, and its objective function is bounded from below.
\end{assumption}

RiNNAL+ is an augmented Lagrangian method for solving \eqref{prob-DNN-general}, incorporating two phases for solving the ALM subproblem: a low-rank phase utilizing Riemannian optimization and a convex lifting phase employing the projected gradient (PG) method. More specifically,
\begin{itemize}
    \item \tb{(Low-rank phrase)} We leverage possible low-rank property of the 
    optimal solution  to accelerate the computation of the optimal solution of the ALM subproblem
    through a Riemannian gradient descent method.
    
    \item \tb{(Convex lifting phrase)} We run one step of the PG method 
    to automatically adjust the variable's rank and escape from a saddle point, if necessary.
    To solve the PG subproblem, we apply a semismooth Newton-CG (SSN) method 
    together with a warm-start technique (described in
    section \ref{sec-acceleration}) to accelerate
    the computation.
\end{itemize}

In this section, we first introduce the general ALM framework of RiNNAL+ to solve \eqref{prob-DNN-general} in subsection~\ref{subsec-ALM}. We then give a detailed description of the low-rank and convex lifting phases for solving the ALM subproblem in subsection~\ref{subsec-LR-phase} and \ref{subsec-CVX-phase}, respectively. Finally, we elaborate the equivalence between the ALM subproblems of RiNNAL+ for \eqref{SDP-RLT} and \eqref{DNN} in subsection~\ref{subsec-alg-equivalence}.

%%%%%%%%%%%%%%%%%%%%%%%%%%%%%%%%%%%%%%%%%%%%%%%%%%%%%%%%%%%%%%%%%%%%%%%%%%%%%%%

\subsection{Augmented Lagrangian method}\label{subsec-ALM}

RiNNAL+ is an augmented Lagrangian method for solving \eqref{prob-DNN-general}. We first equivalently express \eqref{prob-DNN-general} in the following form:
\begin{align}
    &\min\left\{
\<C,Y\> +\delta_{\F\cap \mathbb{S}_{+}^{n+1}}(Y) :\  
Y\in \E\cap\I
\right\}\nonumber\\
=\,\,&\min\left\{
\<C,Y\> +\delta_{\F\cap \mathbb{S}_{+}^{n+1}}(Y) :\  
\B(Y)=g,\ C(Y)\geq h
\right\}.\label{prob-DNN-general-reform}
\end{align}
Let $\sigma>0$ be a given penalty parameter. The augmented Lagrange function is defined by
 \begin{equation*}
 L_\sigma(Y;\lambda,\mu):= \<C,Y\>+\frac{\sigma}{2}\|\sigma^{-1}\lambda-(\B(Y)-g) \|^2+\frac{\sigma}{2}\| \Pi_{+}( \sigma^{-1}\mu-(\C(Y)-h)) \|^2.
 \end{equation*}
We can apply the following augmented Lagrangian method to solve \eqref{prob-DNN-general-reform}. Specifically, given the initial penalty parameter $\sigma_0>0$ and dual variables $\lambda^0\in \R^{d_1}$ and $\mu^0\in \R^{d_2}_+$, perform the following steps at the $(k+1)$-th iteration:
\begin{align}
Y^{k+1}&=\arg\min\left\{L_{\sigma_k} (Y;\lambda^{k},\mu^{k}) :\ Y\in \F\cap \S^{n+1}_+\right\},\label{ALM-sub-Y}
\tag{CVX}\\
%\tag{ALM-sub}\\
\lambda^{k+1}&=\lambda^{k}-\sigma_k (\B (Y^{k+1})-g),\\
\mu^{k+1}&=\Pi_+(\mu^{k}-\sigma_k (\C (Y^{k+1})-h)),\label{ALM-mu}
\end{align}
where $\sigma_k \uparrow \sigma_{\infty} \leq+\infty$ are positive penalty parameters. For a comprehensive discussion on the augmented Lagrangian method applied to convex optimization problems and beyond, see \cite{hestenes1969multiplier, powell1969method, rockafellar1976augmented}. 

The main challenge in the ALM lies in solving the convex ALM subproblem \eqref{ALM-sub-Y}. In the following subsections, we introduce 
a two-phase method to address this: the low-rank phase, utilizing Riemannian optimization, and the convex lifting phase, employing the PG method. We then demonstrate how to efficiently solve the subproblem \eqref{ALM-sub-Y} by combining these two phases. The algorithmic framework is listed in Algorithm \ref{alg1}, where $Y^{i}$ is obtained by the factorization described in the next subsection.

\begin{algorithm}
\linespread{1.1}\selectfont
\caption{The RiNNAL+ method}
\label{alg1}
\begin{algorithmic}[1]
\STATE {\bf Parameters:} Given $\sigma_0>0$, integer $r_0>0$, and initial point $R^0\in\M_{r_0}$.
\STATE $k\gets 0$, $i\gets 0$, $\lambda^0=0_{d_1}$, $\mu^0=0_{d_2}$.
\WHILE{\eqref{prob-DNN-general} is not solved to required accuracy}
    \WHILE{\eqref{ALM-sub-Y} is not solved to required accuracy}
        \STATE Obtain $R^{i+1}$ by solving \eqref{ALM-sub-LR-M} using the Riemannian gradient descent method.
        \STATE Obtain $Y^{i+1}$ from $R^{i+1}$ by \eqref{lr-R}.
        \STATE Obtain $Y^{i+1}$ by solving \eqref{ALM-sub-Y} using one step of the projected gradient method.
        \STATE Obtain $R^{i+1}$ from $Y^{i+1}$ by \eqref{lr-R}.
        \STATE $i\gets i+1$.
    \ENDWHILE
    \STATE $Y^{k+1}\gets Y^{i}$. % not i+1 because i<- i+1 already
    \STATE $\lambda^{k+1}=\lambda^{k}-\sigma_k (\B (Y^{k+1})-g)$.
    \STATE $\mu^{k+1}=\Pi_+(\mu^{k}-\sigma_k (\C (Y^{k+1})-h))$.
    \STATE {Update $\sigma_k$}.
    \STATE $k\gets k+1$,\;\; $i\gets 0$.
\ENDWHILE
\end{algorithmic}
\end{algorithm}

Algorithm~\ref{alg1} is a double-loop method, where the outer loop follows the same structure as the ALM in \cite{RiNNAL}, with convergence analysis conducted under certain accuracy requirements for the subproblems. The primary distinction lies in the approach used to solve the ALM subproblems. In the following subsections, we discuss the hybrid method employed for solving these subproblems in Algorithm~\ref{alg1}. This method was originally introduced in \cite{lee2022escaping} for low-rank matrix optimization and is accompanied by a detailed convergence analysis.

\begin{remark}
    Theoretically, the update scheme in Algorithm~\ref{alg1} guarantees convergence. In practice, however, the single PG step may be skipped if solving \eqref{ALM-sub-LR-M} already provides a sufficiently accurate solution. To ensure that the rank is large enough to guarantee convergence while remaining small to reduce computational cost, a PG step is performed after the low-rank phase at every 5 outer ALM iterations or whenever singularity issues occur, as discussed in the implementation in Section~\ref{sec-experiments}.
\end{remark}

%%%%%%%%%%%%%%%%%%%%%%%%%%%%%%%%%%%%%%%%%%%%%%%%%%%%%%%%%%%%%%%%%%%%%%%%%%%%%%%

\subsection{Low-rank phase}\label{subsec-LR-phase}

In the low-rank phase, we adopt the method proposed in \cite{RiNNAL} to solve \eqref{ALM-sub-Y}. Below, we briefly summarize this approach. Let $\sigma\in\R_+$, $\lambda\in \R^{d_1}$ and $\mu\in \R^{d_2}_+$ be fixed, and assume that the subproblem (\ref{ALM-sub-Y}) has an optimal solution with rank $r>0$. By the Burer-Monteiro (BM) factorization, any optimal solution $Y\in \mathcal{F}\cap \S_+^{n+1}$ of rank $r$ can be expressed in the following factorized form:
\begin{equation}\label{lr-R}
    Y=\begin{bmatrix}
        1 & x^{\top} \\
x & X
    \end{bmatrix}=\begin{bmatrix}
        e_1^{\top} \\
    R
    \end{bmatrix}
    \begin{bmatrix}
        e_1 & R^{\top}
    \end{bmatrix}=\RR,
\end{equation}
where $R\in \R^{n\times r}$ and $\hR:=[e_1^\top;R]$ with $e_1$ being the first standard unit vector in $\R^r$.
Thus, \eqref{ALM-sub-Y} is equivalent to the following factorized nonconvex problem: 
\begin{equation}
\label{ALM-sub-general-LR}
\min_R\left\{ 
f_r(R):=L_\sigma(\hR \hR^\top;\lambda,\mu)
:\ R\in \N_r \right\},
\end{equation}
where the feasible set $\N_r$ is defined as
\begin{equation*}
%\label{algebraic-variety-N}
    \mathcal{N}_r :=  \left\{R\in\R^{n\times r} :\  ARe_1=b,\ ARR^\top=b(Re_1)^\top,\ \operatorname{diag}_B(RR^{\top})=R_B e_1 \right\}.
\end{equation*}
The set $\mathcal{N}_r$ contains a huge number of $\Omega(mn)$ quadratic equality constraints and is non-smooth everywhere, i.e., the LICQ condition does not hold at every $R\in \mathcal{N}_r$. However, an important observation made in \cite{RiNNAL} is that
\[
\hR\hR^\top \in\F\cap\S^{n+1}_+\quad\Longleftrightarrow\quad R\in \N_r\quad\Longleftrightarrow\quad R\in \M_r,
\]
where $\M_r$ is a much simpler set defined as
\begin{equation*}
%\label{algebraic-variety-M}
\mathcal{M}_r := \left\{R \in \mathbb{R}^{n \times r} :\  AR=be_1^\top,\ \operatorname{diag}_B(RR^{\top})=R_B e_1 \right\}.
\end{equation*}
The set $\M_r$  contains only $mr$ linear constraints and $p$ spherical constraints after a linear transformation.
Therefore, problem \eqref{ALM-sub-Y} can be further simplified as follows:
\begin{equation}
\label{ALM-sub-LR-M}
\min_R\left\{f_r(R) :\ R\in \M_r\right\}.
\tag{LR}
\end{equation}
Different from $\N_r$, the feasible set $\M_r$ is assured to conform to a manifold structure under some conditions \cite{RiNNAL},
thus enabling the application of Riemannian optimization methods for its solution. Define $\widehat{I}:=[0_{1\times n};I_n]$ and 
\begin{align*}
\lambda^{+}(R) :=\lambda-\sigma(\B(\hR \hR^\top)-g),\quad
\mu^{+}(R) :=\Pi_{+}(\mu-\sigma(\C(\hR \hR^\top)-h)).
\end{align*}
Then the objective function value $f_r(R)$ and its gradient can be computed by
\begin{align*}
f_r(R)&=\<C,\hR \hR^\top\>+\frac{1}{2\sigma}\|\lambda^+(R) \|^2+\frac{1}{2\sigma}\| \mu^+(R) \|^2\label{AL-val-general},\\
\nabla f_r(R) &= 2\widehat{I}\left( C-\B^*\left(\lambda^+(R)\right)- \C^*\left(\mu^+(R)\right)\right) \hR.
\end{align*}
Thus, \eqref{ALM-sub-LR-M} can be solved using the Riemannian gradient descent method with Barzilai-Borwein stepsizes. 
The implementation of the Riemannian gradient descent method to solve
\eqref{ALM-sub-LR-M} is similar to that in \cite{RiNNAL}, and we omit the 
details here.

When using this method to solve \eqref{ALM-sub-LR-M}, two important operations are the projection
and retraction \cite{manibook,intromani}.
The projection onto the tangent space of $\M_r$ involves solving an $(mr+p)$ by $(mr+p)$ symmetric positive definite linear system, whose computational cost is in general $\O((mr+p)^3).$ However, it is shown in \cite{RiNNAL} that by utilizing the special structure of $\M_r$, the computational cost of the projection can be reduced to $$\O\left(\min\left\{ p^3+m^2r+mrp, (mr)^2p+(mr)^3\right\}\right),$$ which is much smaller than $\O((mr+p)^3)$ when either $p$ or $mr$ is small.

As for retraction, it is typically more complicated than the projection onto a tangent space because of the nonlinearity and nonconvexity of $\M_r$. However, it is also demonstrated in \cite{RiNNAL} that the non-convex metric projection problem onto $\M_r$ can be equivalently transformed into a convex generalized geometric medium problem. This allows us to adapt the generalized Weiszfeld algorithm to tackle the convex problem with a convergence guarantee. Additional favorable geometric properties of $\mathcal{M}_r$ have been well studied in \cite{tang2024solving,RiNNAL}.

However, as discussed in the introduction, the above approach has certain limitations: (1) the Riemannian gradient descent method may get stuck at saddle points due to the nonconvexity of \eqref{ALM-sub-LR-M}; (2) the optimal rank $r$ for each ALM subproblem is unknown and must be tuned adaptively. Choosing a rank that is too small may cause the low-rank problem \eqref{ALM-sub-LR-M} to be nonequivalent to the ALM subproblem \eqref{ALM-sub-Y}, while a rank that is too large can increase the computational cost substantially. To address these challenges, we introduce the convex lifting phase in the next subsection.

%%%%%%%%%%%%%%%%%%%%%%%%%%%%%%%%%%%%%%%%%%%%%%%%%%%%%%%%%%%%%%%%%%%%%%%%%%%%%%%

\subsection{Convex lifting phase}\label{subsec-CVX-phase}

In the convex lifting phase, we adopt the strategy proposed in \cite{lee2022escaping} to ensure the global convergence of RiNNAL+ and 
automatically adjust the rank. Let $\sigma\in\R_+$, $\lambda\in \R^{d_1}$ and $\mu\in \R^{d_2}_+$ be fixed and define $\L(Y):=L_{\sigma}(Y,\lambda,\mu)$. Given a feasible starting point $\widehat{Y}$, the convex lifting phase runs one PG step as follows:
\begin{equation*}
    Y=\Pi_{\F\cap \mathbb{S}_{+}^{n+1}}(\widehat{Y}-t\nabla \L(\widehat{Y})),
\end{equation*}
where $t>0$ is an appropriate stepsize. We define $G:=\widehat{Y}-t\nabla \L(\widehat{Y})$ in this section with a different meaning from its earlier usage. Then, the projection subproblem can be written as
\begin{equation}\label{proj-0}
 \min_Y\Big\{ \|Y-G\|^2:\, Y \in \F\cap \mathbb{S}_{+}^{n+1} \Big\}.
 \end{equation}
We first decompose $\F\cap \mathbb{S}_{+}^{n+1}=\K\cap\mathcal{P}$, where 
\begin{equation*}
\begin{aligned}
\mathcal{K}:=&\left\{Y \in \mathbb{S}_{+}^{n+1}:\, \< PP^\top,Y\>=0\right\}
&=&\; 
\left\{\begin{bmatrix}z& x^\top \\ x & X\end{bmatrix}\in \mathbb{S}_{+}^{n+1}:\, Ax=b,\ AX=bx^\top\right\},
\\
\mathcal{P}:=&\left\{Y\in \mathbb{S}^{n+1}:\,\H(Y)=q\right\}
&=&\;
\left\{\begin{bmatrix}z& x^\top \\ x & X\end{bmatrix}\in \mathbb{S}^{n+1}:\,\operatorname{diag}_B(X)=x_B,\ z=1\right\},
\end{aligned}
\end{equation*}
and $P:=\begin{bmatrix}
b^\top\\-A^\top
\end{bmatrix}$, $\H(Y):=\begin{bmatrix}
\operatorname{diag}_B(X)-x_B\\z
\end{bmatrix}$, $q:=\begin{bmatrix}
0_{p}\\1
\end{bmatrix}$. The second equality of $\K$ follows from \cite[Theorem 1]{bomze2017fresh}.
Thus, \eqref{proj-0} is equivalent to
\begin{equation}\label{proj}
 \min_Y\left\{ \|Y-G\|^2:\, Y \in \K\cap\mathcal{P} \right\}.
 \end{equation}
To solve \eqref{proj}, we consider its dual problem:
\begin{equation}\label{proj-dual}
 \min_y\left\{ \frac{1}{2}\|\Pi_{\K}(G+\H^*(y))\|^2-\langle q,y\rangle \right\}.
 \end{equation}
The following lemma provides an efficient way to compute the projection onto $\K$, which generalizes the result of \cite[Lemma 7]{lee2022escaping}.
\begin{lemma}\label{lemma-proj-JGJ}
For any full column rank matrix $P\in\R^{n\times r}$, define
\[
\K:=\left\{G\in\S^{n}_+:\<PP^\top,G\>=0\right\}
\]
and $J:=I-P(P^\top P)^{-1}P^\top$. Then it holds that 
\[
\Pi_{\K}(G)=\Pi_{\S^{n}_+}(JGJ),\quad \forall G\in \S^n.
\]
\end{lemma}
\begin{proof}
First, for any $X\in \K$, it holds that $XP=0$ and $P^\top X=0$. Therefore, we have
\begin{equation*}
    \begin{aligned}
        &\|X-JGJ\|^2\\
        =&\|X-(I-PP^\dagger)G(I-PP^\dagger)\|^2\\
        =&\|X-G+PP^\dagger G+ GPP^\dagger -PP^\dagger GPP^\dagger\|^2\\
        %=&\|X-G+PP^\dagger G+ GPP^\dagger - PP^\dagger G PP^\dagger\|^2\\
        =&\|X-G\|^2+2\<X-G,PP^\dagger G+ GPP^\dagger - PP^\dagger G PP^\dagger\> + \|PP^\dagger G+ GPP^\dagger - PP^\dagger G PP^\dagger\|^2\\
        =&\|X-G\|^2-2\<G,PP^\dagger G+ GPP^\dagger - PP^\dagger G PP^\dagger\> + \|PP^\dagger G+ GPP^\dagger - PP^\dagger G PP^\dagger\|^2,
    \end{aligned}
\end{equation*}
where $P^\dagger=(P^\top P)^{-1}P^\top$ is the Moore–Penrose inverse. As a consequence, it holds that $\Pi_{\K}(G)=\Pi_{\K}(JGJ)$. Furthermore, we observe that
\[
\langle PP^\top, JGJ \rangle = \langle JPP^\top J, G \rangle = \langle (I - PP^\dagger) PP^\top (I - PP^\dagger), G \rangle = 0,
\]
which implies that $(JGJ)P = 0$, meaning that $JGJ$ has zero eigenvalues 
with the columns of $P$ as eigenvectors. Consequently, we find that $\Pi_{\mathbb{S}^n_+}(JGJ) \in \K$, since the projection onto $\mathbb{S}^n_+$ only sets negative eigenvalues to zero in the eigendecomposition. Given that $\K \subseteq \mathbb{S}^n_+$, it follows that $\Pi_{\K}(G) = \Pi_{\mathbb{S}^n_+}(JGJ)$.
\end{proof}
By Lemma \ref{lemma-proj-JGJ}, the dual problem \eqref{proj-dual} is equivalent to
\begin{equation}\label{proj-dual-new1}
 \min_y\left\{ \frac{1}{2}\|\Pi_{\S^{n+1}_+}(J(G+\H^*(y))J)\|^2-\langle q,y\rangle \right\}.
 \end{equation}
 Define $\widehat{G}:=JGJ$ and $\widehat{\H} (X):=\H(JXJ)$, then \eqref{proj-dual-new1} can be written as
\begin{equation}\label{proj-dual-new2}
 \min_y\left\{ \frac{1}{2}\|\Pi_{\S^{n+1}_+}(\widehat{G}+\widehat{\H}^*(y))\|^2-\langle q,y\rangle \right\},
 \end{equation}
 which can be solved by the Semismooth Newton-CG method, see, for example, \cite{qi2006quadratically,SDPNAL}. However, directly solving \eqref{proj-dual-new2} typically requires numerous SSN iterations and CG steps to reach 
 a reasonably accurate approximate solution, as indicated by our numerical experiments. In Subsection \ref{subsec-preprocess} and \ref{subsec-recover-dual}, we introduce a preprocessing technique and a warm-start technique to mitigate these issues.

%%%%%%%%%%%%%%%%%%%%%%%%%%%%%%%%%%%%%%%%%%%%%%%%%%%%%%%%%%%%%%%%%%%%%%%%%%%%%%%

\subsection{Equivalence between ALM subproblems}\label{subsec-alg-equivalence}

When applying RiNNAL+ to the two different reformulations \eqref{SDP-RLT} and \eqref{DNN}, the subproblems \eqref{ALM-sub-Y} differ in variable dimensions and constraints. However, we prove that they are equivalent under the invertible linear transformation $\Phi:\S^{n+1}\rightarrow\S^{n+l+1}$ defined in \eqref{phi}. We first denote
\begin{equation*}
\begin{aligned}
{\mathrm{F}'_R}:=\F\cap \mathbb{S}_{+}^{n+1},\quad 
{\mathrm{F}'_D}:=\F_D\cap \mathbb{S}_{+}^{n+l+1},
\end{aligned}
\end{equation*}
which are the feasible regions of the ALM subproblem \eqref{ALM-sub-Y} of \eqref{SDP-RLT} and \eqref{DNN}, respectively.
Similar to Theorem \ref{lemma-F1-F2}, we can prove the following result.
\begin{lemma}\label{lemma-F1F2sub}
    The map $\Phi$ is a bijection from ${\mathrm{F}'_R}$ to ${\mathrm{F}'_D}$.
\end{lemma}
We omit the proof here since it directly follows from the proof of Lemma \ref{lemma-F1-F2}. 
Next, consider the following ALM subproblem of \eqref{DNN}:
\begin{equation}
\label{prob-D2-ALM-sub}
\min_{Y'}\left\{\<C',Y'\>+\frac{1}{2\sigma}\|\Pi_+(\mu'-\sigma Y')\|^2:\ Y'\in {\mathrm{F}'_D}\right\},
\end{equation}
where $\mu'\in\S^{n+1}$ is the dual variable. 
Combining Lemma \ref{lemma-F1F2sub} with the relation $\langle C',\Phi(Y)\rangle=\langle C,Y\rangle$, we have that 
\eqref{prob-D2-ALM-sub} is equivalent to the following problem with respect to $Y$:
\begin{equation}
\label{prob-D1-ALM-sub}
\min_{Y}\left\{\<C,Y\>+\frac{1}{2\sigma}\|\Pi_+(\mu'-\sigma \Phi(Y))\|^2:\ Y\in {\mathrm{F}'_R} \right\}, 
\end{equation}
which is exactly the subproblem \eqref{ALM-sub-Y} when applying RiNNAL+ to \eqref{SDP-RLT}. 
\begin{remark}
Although the subproblems of RiNNAL+ for \eqref{SDP-RLT} and \eqref{DNN} are theoretically equivalent through a linear transformation, it is important to note that their numerical performance differs. This discrepancy arises because solving the equivalent problems \eqref{prob-D2-ALM-sub} and \eqref{prob-D1-ALM-sub} with different variables and constraints can lead to distinct computational behaviors. Furthermore, since the underlying algebraic varieties are different, the time required for computing the projection, retraction, and eigenvalue decomposition of the dual variable differs between the two formulations. Based on the numerical experiments in subsection \ref{subsec-pp}, applying RiNNAL+ to \eqref{SDP-RLT} is often more efficient in reducing computational costs compared to solving \eqref{DNN}.
\end{remark}

%%%%%%%%%%%%%%%%%%%%%%%%%%%%%%%%%%%%%%%%%%%%%%%%%%%%%%%%%%%%%%%%%%%%%%%%%%%%%%%
%%%%%%%%%%%%%%%%%%%%%%%%%%%%%%%%%%%%%%%%%%%%%%%%%%%%%%%%%%%%%%%%%%%%%%%%%%%%%%%

\section{Acceleration techniques}\label{sec-acceleration}

In this section, we introduce several strategies to enhance the efficiency and stability of RiNNAL+:
\begin{itemize}
    \item \textbf{Modeling perspective:} We propose a preprocessing technique to mitigate the large condition numbers of various linear systems encountered when solving the ALM subproblem using the PG method.
    \item \textbf{Riemannian gradient descent step:} We employ a random perturbation technique to avoid singularity issues in the projection and retraction subproblem within the Riemannian gradient descent method.
    
    \item \textbf{Projected gradient (PG) step:} We recover the dual variables from the low-rank phase, providing a warm-start initial point for solving the SSN subproblem in the PG step. This approach can significantly reduce the computational cost of the PG step.
\end{itemize}

%%%%%%%%%%%%%%%%%%%%%%%%%%%%%%%%%%%%%%%%%%%%%%%%%%%%%%%%%%%%%%%%%%%%%%%%%%%%%%%

\subsection{Preprocessing technique}\label{subsec-preprocess}

In this subsection, we introduce a preprocessing technique to simplify some of the constraints in $Y\in\F\cap\S^{n+1}_+$ that appears in the ALM subproblem \eqref{ALM-sub-Y} into diagonal constraints. This simplification makes the problem easier to handle and can reduce the number of preconditioned conjugate-gradient iterations needed to solve the 
linear systems in the PG step. Let $\sigma\in\R_+$, $\lambda\in \R^{d_1}$ and $\mu\in \R^{d_2}_+$ be fixed and recall that $\L(Y):=L_{\sigma}(Y,\lambda,\mu)$. Consider the following equivalent formulation of the ALM-subproblem \eqref{ALM-sub-Y}:
\begin{equation}\label{ALM-sub-1}
\min \left\{ \L(Y) :\ \langle H_0, Y  \rangle = 1,\ \langle H_k, Y  \rangle = 0 \,\, \forall k \in B,\ \langle H_{n+1},\ Y \rangle = 0,\ Y\succeq 0\right\},
\end{equation}
where 
\begin{align*}
H_0 := \begin{bmatrix} 1 & 0_{1 \times n} \\ 0_{n \times 1} & 0_{n \times n} \end{bmatrix},\ 
H_k := \begin{bmatrix} 0 & -\frac{1}{2} e_k^{\top} \\ -\frac{1}{2} e_k & \operatorname{diag}(e_k) \end{bmatrix},\
H_{n+1}:=PP^\top,
\end{align*}
$e_k \in \mathbb{R}^n$ is the vector whose $k$-th entry is 1 and all other entries are 0, and $P$ is defined after equation \eqref{proj-0}. 
Define the invertable matrix $K:=\begin{bmatrix} 1 & \frac{1}{2}e^\top \\ 0_{n \times 1} & \frac{1}{2}I_n\end{bmatrix}$. By the change of variable $\widehat{Y}=(K^\top)^{-1}YK^{-1}$, \eqref{ALM-sub-1} can be equivalently written as
\begin{equation}\label{ALM-sub-2}
\begin{aligned}
    \min \Big\{ \L(K^\top \widehat{Y}K) :  \; &\langle KH_0K^\top,  \widehat{Y}  \rangle = 1, \; \langle 4KH_kK^\top+KH_0K^\top, \widehat{Y}  \rangle = 1 \, \forall k \in B, \\
    & \langle KH_{n+1}K^\top,  \widehat{Y} \rangle = 0, \; \widehat{Y} \succeq 0 \Big\}.
\end{aligned}
\end{equation}
Note that $KH_{n+1}K^\top=KP(KP)^\top$ and
\begin{align*}
KH_0K^\top := \begin{bmatrix} 1 & 0_{1 \times n} \\ 0_{n \times 1} & 0_{n \times n} \end{bmatrix},\
4KH_kK^\top+KH_0K^\top := \begin{bmatrix} 0 & 0_{1 \times n} \\ 0_{n \times 1}& \operatorname{diag}(e_k) \end{bmatrix}.
\end{align*}
Define $\mathcal{D}: \S^{n+1}\rightarrow \R^{p+1}$ such that $\mathcal{D}(Y)=[Y_{11};\diag_B(Y_{22})]$. 
Then \eqref{ALM-sub-2} is equivalent to
\begin{equation}\label{proj-1}
\begin{aligned}
    \min \left\{ \L(K^\top \widehat{Y}K) :  \; \mathcal{D}(\widehat{Y})=e,\ \langle NN^\top, \widehat{Y} \rangle = 0, \; \widehat{Y} \succeq 0 \right\},
\end{aligned}
\end{equation}
where $N:=KP$ and $e$ is the vector of all ones. We can {apply the PG method to \eqref{proj-1}}. For a given starting point ${Y_0}$ provided by the low-rank phase, define 
\[
\widehat{Y}_0=(K^{-1})^\top {Y_0}K^{-1},\quad \widehat{G}=\widehat{Y}_0-tK\nabla \L({Y}_0)K^\top,
\]
then the PG method updates $\widehat{Y}_0$ as follows: 
\begin{equation}\label{ALM-sub-3}
\begin{aligned}
    \widehat{Y}_1 =\operatorname{argmin} \left\{ \frac{1}{2}\|\widehat{Y}-\widehat{G}\|^2 :  \; \mathcal{D}(\widehat{Y})=e,\ \langle NN^\top, \widehat{Y} \rangle = 0, \; \widehat{Y} \succeq 0 \right\}.
\end{aligned}
\end{equation}
Define $J:=I_{n+1}-NN^\dagger$. Then the dual problem of \eqref{ALM-sub-3} is
\begin{equation}\label{ALM-sub-3-dual}%\tag{KKT-3}
\mathcal{D}\big( J\Pi_{\S^{n+1}_+}\big(J(\widehat{G}+\mathcal{D}^*(y))J\big)J\big)-e=0.
\end{equation}
We can apply the SSN method to solve \eqref{ALM-sub-3-dual}. Note that a special case of the projection subproblem \eqref{ALM-sub-3} is the nearest correlation matrix problem \cite{qi2006quadratically}, where an efficient implementation of the SSN method is provided at \url{https://www.polyu.edu.hk/ama/profile/dfsun/CorrelationMatrix.m}. We modify the code to solve the more general problem \eqref{ALM-sub-3}, and use the exact diagonal preconditioner to accelerate the preconditioned conjugate gradient (PCG) method. To demonstrate the benefit of the preprocessing technique, we consider the following Shor relaxation of a binary integer quadratic (BIQ) programming problem, which is in the form of \eqref{ALM-sub-1}:
\begin{equation}\label{shor-biq-without-pre}
\min\left\{\< C,Y\>:\ \diag(X)=x,\ Y=\begin{bmatrix}
    1&x^\top\\x&X \end{bmatrix}\in\S^{n+1}_+ \right\}.
\end{equation}
After applying the preprocessing technique, it can be equivalently reformulated as
\begin{equation}\label{shor-biq-with-pre}
\begin{aligned}
    \min \left\{ \<KCK^\top, \widehat{Y}\> :  \; \diag(\widehat{X})=e,\   \widehat{Y}=\begin{bmatrix}
    1&\widehat{x}^\top\\\widehat{x}&\widehat{X} \end{bmatrix}\in\S^{n+1}_+ \right\},
\end{aligned}
\end{equation}
where the reformulation has the structure of the standard SDP relaxation 
of a max-cut problem. To verify the effectiveness of the preprocessing technique, we compare the performance of SDPNAL+ in solving \eqref{shor-biq-without-pre} and \eqref{shor-biq-with-pre}. The numerical results demonstrate that the preprocessing technique significantly reduces the total computation time from 36 seconds to 5 seconds. Additionally, the number of ADMM+ iterations, SSN iterations, SSN subproblems, and PCG iterations decreases from $(2100,30,237,8305)$ to $(400,18,43,607)$, respectively. Since the set $\F\cap\S^{n+1}_+$ is general and widely used in many relaxation problems, the technique can be used as a subroutine for many existing solvers to make the problem constraints easier to handle.

%%%%%%%%%%%%%%%%%%%%%%%%%%%%%%%%%%%%%%%%%%%%%%%%%%%%%%%%%%%%%%%%%%%%%%%%%%%%%%%

\subsection{Random perturbation technique}\label{subsec-perturnation}

In this subsection, we adopt the random perturbation technique used in \cite{tang2023feasible} to avoid the nonsmooth points in the variety $\M_r$. Recall that the algebraic variety of the low-rank phase is
\begin{equation*}
\mathcal{M}_r := \left\{R \in \mathbb{R}^{n \times r} :\  AR=be_1^\top,\ \operatorname{diag}_B(RR^{\top})=R_B e_1 \right\},
\end{equation*}
which can be expressed as
\begin{equation*}
\mathcal{M}_r = \left\{R \in \mathbb{R}^{n \times r} :\  AR=be_1^\top,\ \operatorname{diag}_B\big((2R-ee_1^\top)(2R-ee_1^\top)^{\top}\big)=e \right\}.
\end{equation*}
The LICQ condition may not hold at some points $R\in\M_r$. However, we can always add a small perturbation to the spherical constraints, and consider the following set:
\begin{equation*}
\mathcal{M}_{r,v} := \left\{R \in \mathbb{R}^{n \times r} :\  AR=be_1^\top,\ \operatorname{diag}_B\big((2R-ee_1^\top)(2R-ee_1^\top)^{\top}\big)=e+v \right\},
\end{equation*}
where $v\in\R^p$ is a random vector such that $\|v\|=\epsilon$ and $\epsilon>0$ is a given small scalar. We can ensure the generic smoothness of the new set $\M_{r,v}$ by the following lemma.

\begin{lemma}{\cite[Theorem 4]{tang2023feasible}}
    For a generic $v$, every point of $\M_{r,v}$ satisfies the LICQ property.
\end{lemma}

Thus, when RiNNAL+ encounters a nonsmooth point, we can introduce a random perturbation to $\M_r$, ensuring that the projection and retraction steps on the 
slightly perturbed manifold $\M_{r,v}$ can be successfully computed.

%%%%%%%%%%%%%%%%%%%%%%%%%%%%%%%%%%%%%%%%%%%%%%%%%%%%%%%%%%%%%%%%%%%%%%%%%%%%%%%

\subsection{Warm start technique}\label{subsec-recover-dual}

To accelerate the computation in the SSN method for solving the projection problem \eqref{ALM-sub-3-dual}, we can recover an initial point $y$ from the low-rank phase as a warm-start point, which is obtained by considering the correspondence between the KKT conditions of different problems. On the one hand, the KKT condition of the ALM subproblem \eqref{ALM-sub-1} is
\begin{equation}\label{KKT-ALM-sub}\tag{KKT-1}
\nabla \L(Y) - \alpha H_0 - \sum_{k\in B}\mu_k H_k - \beta H_{n+1}=S_1,\ \langle S_1,Y\rangle =0, \ S_1\succeq 0,
\end{equation}
where $(\alpha,\mu,\beta,S_1)\in\R\times\R^{p}\times\R\times\S^{n+1}$ are the multipliers that can be recovered from the low-rank algorithm, see \cite[Subsection 4.1]{RiNNAL} for more details. On the other hand, the KKT condition of the projection problem \eqref{ALM-sub-3} is
\begin{align}\label{KKT-2}\tag{KKT-2}
\widehat{Y}-\widehat{G}-\mathcal{D}^*(y)-y_0 NN^\top = S_2,\ \langle S_2,\widehat{Y}\rangle =0,\ S_2\succeq 0,
\end{align}
where $(y,y_0,S_2)\in\R^{p+1}\times\R\times\S^{n+1}$ are the multipliers that need to be recovered.
When approaching the optimality of the ALM subproblem \eqref{ALM-sub-1}, the starting point $Y_0$ provided by the low-rank phase is also the optimal solution of the projection problem \eqref{ALM-sub-3}, i.e.,
\[
\widehat{Y}=\widehat{Y}_0=(K^{-1})^{\top}Y_0K^{-1}=(K^{-1})^{\top}YK^{-1}.
\]
Note that $\widehat{G}=\widehat{Y}-tK\nabla \L(Y)K^\top$. Therefore, \eqref{KKT-2} can be written as:
\begin{equation}\label{KKT-proj-3}\tag{KKT-3}
tK\nabla \L(Y)K^\top-\mathcal{D}^*(y)-y_0 NN^\top = S_2,\ \langle S_2,(K^{-1})^{\top}YK^{-1}\rangle =0,\ S_2\succeq 0.
\end{equation}
Comparing \eqref{KKT-ALM-sub} and \eqref{KKT-proj-3}, we can recover the dual variable $y_0$, $y$ and $S_2$ of the projection problem by 
\[
y_0=t\beta,\quad y=t\left[\alpha-\frac{1}{4}\sum_{k=1}^{p} \mu_k;\ \frac{1}{4}\mu\right],\quad S_2=tKS_1K^{\top}.
\]
When the initial point of the PG step is near the optimal solution of the ALM subproblem \eqref{ALM-sub-Y} (equivalently \eqref{ALM-sub-1}), the dual variable $(y,y_0)$ constructed above is also near the optimal solution of the projection problem \eqref{ALM-sub-3}. Therefore, we can use $(y,y_0)$ obtained above to warm start the PG step, which can significantly decrease the number of SSN iterations.

%%%%%%%%%%%%%%%%%%%%%%%%%%%%%%%%%%%%%%%%%%%%%%%%%%%%%%%%%%%%%%%%%%%%%%%%%%%%%%%
%%%%%%%%%%%%%%%%%%%%%%%%%%%%%%%%%%%%%%%%%%%%%%%%%%%%%%%%%%%%%%%%%%%%%%%%%%%%%%%

\section{Numerical experiments}\label{sec-experiments}

In this section, we conduct numerical experiments to illustrate the effectiveness of RiNNAL+ for solving SDP 
relaxation problems of the form \eqref{prob-DNN-general}. All experiments are performed using {\sc Matlab} R2023b on a workstation equipped with Intel Xeon E5-2680 (v3) processors and 96GB of RAM.

\medskip

\noindent\textbf{Baseline Solvers}. 
We compare the performance of RiNNAL+ with the solver SDPNAL+ \cite{SDPNAL,SDPNALp1,SDPNALp2}.
Although various other ALM-based algorithms and low-rank SDP solvers exist, we exclude them from our comparison because they are either inapplicable to \eqref{prob-DNN-general} or not competitive with SDPNAL+, as observed in \cite{RiNNAL}. Similarly, we do not consider RiNNAL, as it cannot directly handle \eqref{MBQP} with inequality constraints.
\medskip

\noindent\textbf{Stopping Conditions}. Based on the KKT conditions \eqref{KKT} for \eqref{prob-DNN-general}, we define the following relative KKT residual to assess the accuracy of the solution computed by RiNNAL+:
\small{
\begin{align*}
\operatorname{R_p} := \cfrac{\sqrt{\|\mathcal{A}(Y)- {d}\|^2+\|\B(Y)- {g}\|^2+\|\Pi_+(h-\C(Y))\|^2} }{1+\sqrt{\| {d} \|^2 + \| {g} \|^2+\|h\|^2}},\
\operatorname{R_d} :=  \frac{\|\Pi_{\S_+^{n+1}}(S)\|}{1+\|S\|},\
\operatorname{R_{c}} := \frac{|\langle Y, S\rangle|}{1+\|Y\|+\|S\|}.
\end{align*}
}
\normalsize
Note that we do not include the dual KKT residual and complementarity of the constraint $Y\in\I$ because they are always equal to zero due to the update of the ALM multiplier $\mu$ in \eqref{ALM-mu}. We also omit the primal KKT residual of the constraint $Y\in \S^{n+1}$ because it is always equal to zero due to the low-rank factorization \eqref{lr-R}.
For a given tolerance $\tol> 0$, RiNNAL+ is terminated when the maximum residual satisfies $\operatorname{R_{max}}:=\max\{\operatorname{R_p},\operatorname{R_d},\operatorname{R_c}\}<\tol$ or the maximum time limit $\timelimit$ is reached. In our experiments, we set $\tol = 10^{-6}$ and $\timelimit = 3600\tt{(secs)}$ for all solvers.

\medskip

\noindent\textbf{Implementation}. 
In RiNNAL+, we employ a Riemannian gradient descent method with Barzilai-Borwein steps and non-monotone line search to solve the augmented Lagrangian subproblems (see \cite{RiNNAL,iannazzo2018riemannian, gao2021riemannian, tang2023feasible}). The penalty parameter $\sigma_k$ is initialized at $\sigma_0 = 1$ and adjusted dynamically: it is increased by a factor of 1.5 if $\operatorname{R_p}/\operatorname{R_d}\geq 2$ and decreased by a factor of 1.5 if $\operatorname{R_p}/\operatorname{R_d}\leq 1/5$. The initial rank $r_0$ is set to be $\operatorname{min}\{200,\lceil n/5 \rceil \}$. The stepsize $t_k$ of the PG step is determined as $1/\sigma_k$. The initial point $R_0$ is randomly selected from the feasible region $\M_{r_0}$. In each ALM iteration, the low-rank phase is performed with the maximum number of Riemannian gradient descent iterations capped at 50. Additionally, a PG step is executed 
after the low-rank phase at every 5 outer ALM iterations or whenever singularity issues occur.

\medskip

\noindent\textbf{Table Notations.} We use `-' to indicate that an algorithm did not achieve the required tolerance $\tol$ within the maximum time limit $\timelimit$. We report all results for problem sizes with $n \leq 1500$, even if the algorithm does not reach the desired accuracy. For $n > 1500$, SDPNAL+ consistently fails to converge and yields solutions with poor accuracy. Therefore, we do not report its results in these cases. A superscript ``$\dagger$'' is appended to the KKT residual value to indicate that the algorithm attained moderate accuracy, though not the required accuracy.
For the column labeled ``Iteration'' associated with SDPNAL+, the first entry denotes the number of outer iterations, the second entry denotes the total number of semismooth Newton inner iterations, and the third indicates the total number of ADMM+ iterations. Similarly, for the column labeled ``Iteration'' associated with RiNNAL+, the first entry corresponds to the number of ALM iterations, the second denotes the total number of Riemannian gradient descent iterations, and the third 
reports the total number of PG steps.
The column labeled ``Rank'' indicates the number of columns of the final iterate $R$ for RiNNAL+ and the final rank of the output matrix $Y$ for SDPNAL+. 
The ``Objective'' column denotes the value of the objective function, while the total computation time is listed under ``Time''. Additionally, the final column, labeled ``TPG'', reports the time (in seconds) consumed by PG steps.

%%%%%%%%%%%%%%%%%%%%%%%%%%%%%%%%%%%%%%%%%%%%%%%%%%%%%%%%%%%%%%%%%%%%%%%%%%%%%%%

\subsection{Tightness of different relaxations}\label{subsec-tightness-numerical}

In this subsection, we compare the bound tightness of various relaxations. We focus on the following binary integer quadratic (BIQ) programming problem:
\begin{equation}\label{BIQ-MBQP}
v^*:=\min\left\{x^{\top} Q x+2 c^{\top}x:\  x \in\{0,1\}^n\right\}\tag{BIQ}
\end{equation}
and its strengthened formulation:
\begin{equation}\label{BIQ-SMBQP}
\min\left\{x^{\top} Q x+2 c^{\top}x:\ x\leq e,\ x \in\{0,1\}^n\right\}.\tag{BIQ-S}
\end{equation}
To evaluate the tightness, we solve relaxation problems of \eqref{BIQ-MBQP} and \eqref{BIQ-SMBQP} to obtain the lower bound $v$, and compare the relative gap between the lower bound $v$ and the optimal solution $v^*$. The relative gap is defined as:
\begin{equation*}
    \%\text{gap} = \frac{v^* - v}{v^*} \times 100\%.
\end{equation*}
We choose the bqp100.1 instance from the ORLIB library\footnote{Dataset available at \url{http://people.brunel.ac.uk/~mastjjb/jeb/info.html}.} maintained by J.E. Beasley to generate the coefficient matrix $Q$ and $c$. Then, we use Gurobi \cite{gurobi} to solve \eqref{BIQ-MBQP} and \eqref{BIQ-SMBQP} exactly with $\tol=10^{-10}$ and employ SDPNAL+ to solve the relaxation problems with the same tolerance. The tightness results are summarized in Table \ref{tab-tightness}.
\begin{table}[h!]
    \centering
    \renewcommand{\arraystretch}{1.25} % Adjust row spacing
    \begin{tabular}{|c|cc|cc|c}
        \cline{1-5}
         & \multicolumn{2}{c|}{\eqref{BIQ-MBQP}} & \multicolumn{2}{c|}{\eqref{BIQ-SMBQP}} & \multirow{6}{*}{$\xdownarrow{1.6cm}$}\\
        \cline{1-5}
        {$v^*=$\, -7.97e+03} & lower bound {$v$} & {\%gap} & lower bound {$v$} & {\%gap} & \\
        \cline{1-5}
        \eqref{SHOR} & -8.72110529e+03 & 9.424 & -8.72110529e+03 & 9.424 &  \\
        \cline{1-5}
        \eqref{SDP-RLT} & -8.38038443e+03 & 5.149 & -8.03665843e+03 & 0.836 & \\
        \cline{1-5}
        \eqref{DNN} & -8.38038443e+03 & 5.149 & -8.03665843e+03 & 0.836 & \\\cline{1-5}
        \eqref{COMP} & \multicolumn{2}{c|}{(not applicable)} & -8.03665843e+03 & 0.836 & \\
        \cline{1-5}
        \multicolumn{5}{c}{$\xrightarrow[\hspace{11cm}]{}$} & {Tighter} \\
    \end{tabular}
    \caption{Tightness of different relaxations for the bqp100.1 instance.}
    \label{tab-tightness}
\end{table}

As shown in Table \ref{tab-tightness}, the lower bounds obtained from \eqref{SDP-RLT}, \eqref{DNN}, and \eqref{COMP} are identical and significantly tighter than that of \eqref{SHOR}. This observation aligns with the theoretical results on tightness presented in Theorem \ref{thm-1} and Theorem \ref{thm-tightness-SMBQP}. Additionally, the lower bounds of \eqref{BIQ-SMBQP} are much tighter than those of \eqref{BIQ-MBQP}, demonstrating that the strengthening technique of adding the redundant constraint $x_B\leq e$, as discussed in Subsection \ref{subsec-application}, can substantially reduce the relaxation gap.

%%%%%%%%%%%%%%%%%%%%%%%%%%%%%%%%%%%%%%%%%%%%%%%%%%%%%%%%%%%%%%%%%%%%%%%%%%%%%%%

\subsection{Performance on different relaxation reformulations}\label{subsec-pp}

In this subsection, we analyze the numerical performance of RiNNAL+ and SDPNAL+ for different relaxation formulations: \eqref{SDP-RLT}, \eqref{DNN}, and \eqref{COMP}. Although these three formulations are theoretically equivalent, their computational performance varies when solved by the same algorithm.

We conduct experiments on strengthened versions of BIQ problems \eqref{BIQ-SMBQP} and quadratic knapsack problems \eqref{prob-qkp-v2}, which will be described in detail in subsections \ref{subsec-BIQ} and \ref{subsec-QKP}, respectively. To evaluate the computational efficiency of RiNNAL+ for different reformulations, we utilized the Dolan-Moré performance profile\footnote{Performance profile script obtained from \url{https://www.mcs.anl.gov/~more/cops/}.} \cite{dolan2002benchmarking}. More specifically, suppose we benchmark $S$ solvers on $P$ problems, with $t_{i,j}$ denoting the time taken by solver $i$ to solve problem $j$. The performance ratio $r_{i,j}$ for solver $i$ on problem $j$ is defined as:
\[
r_{i,j} = \frac{t_{i,j}}{\min\{t_{l,j} : 1 \leq l \leq S\}}, \quad 1 \leq i \leq S, \; 1 \leq j \leq P.
\]
The performance profile plots the fraction of problems solved by each solver within a factor $\tau$ of the best solver, defined as:
\[
f_i(\tau) = \frac{1}{P} \sum_{1 \leq j \leq P} I_{[0,\tau]}(r_{i,j}), \quad 1 \leq i \leq S, \quad  \tau \in \mathbb{R}_{++},
\]
where $I_{[0,\tau]}$ is the indicator function of the interval $[0,\tau]$.
Thus, $f_i(\tau)$ denotes the fraction of problems solved within $\tau$ times the best solver's time. Higher curves indicate better performance. The performance profiles are shown in Figure \ref{fig-pp}, where the $\operatorname{lg}(\cdot)$ function is base 2. Since SDPNAL+ failed to solve any QKP problems within the time limit, it is excluded from Figure \ref{fig:pp-QKP}. For BIQ problems, only the shortest computation time of SDPNAL+ among the three formulations is reported. 

\begin{figure}[h!]
    \centering
    % First subfigure
    \begin{subfigure}{0.49\textwidth}
        \centering
        \includegraphics[width=\linewidth]{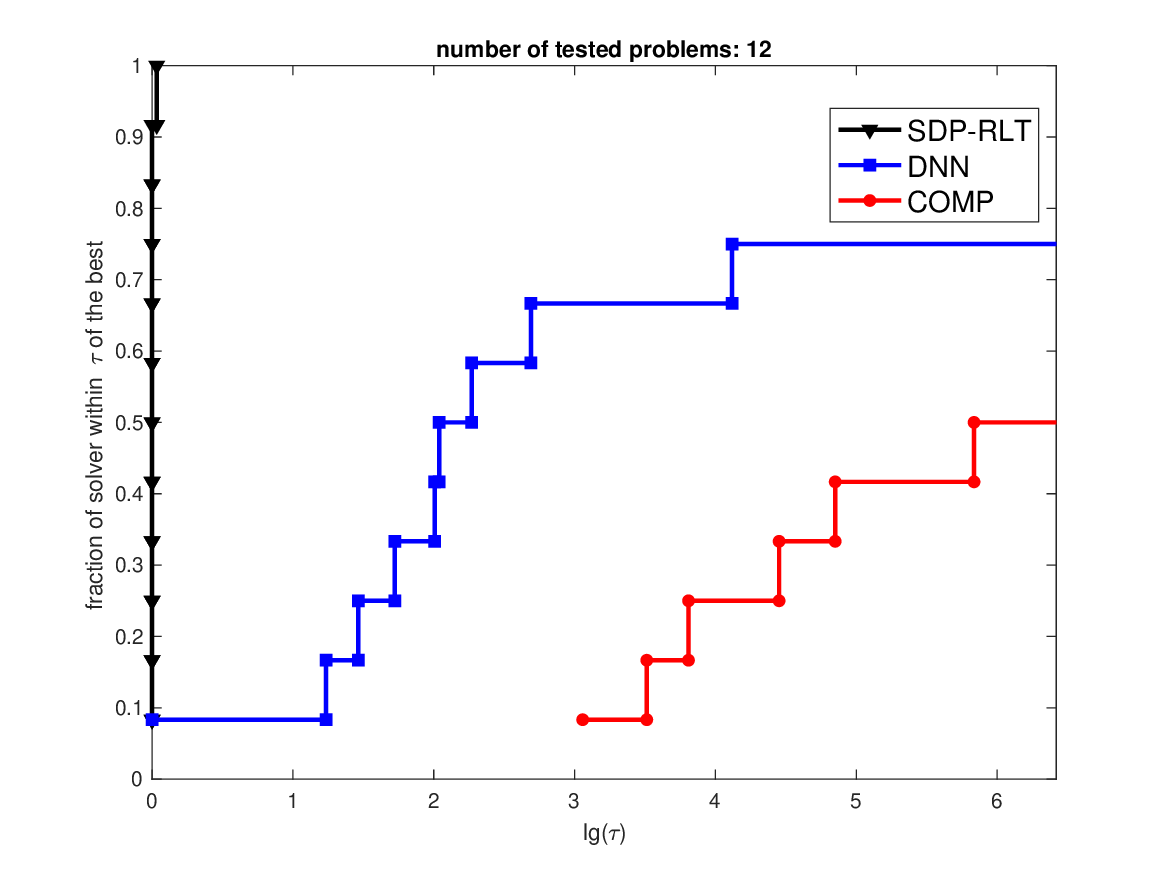}
        \caption{Strengthened QKP problems.}
        \label{fig:pp-QKP}
    \end{subfigure}
    \hfill
    % Second subfigure
    \begin{subfigure}{0.49\textwidth}
        \centering
        \includegraphics[width=\linewidth]{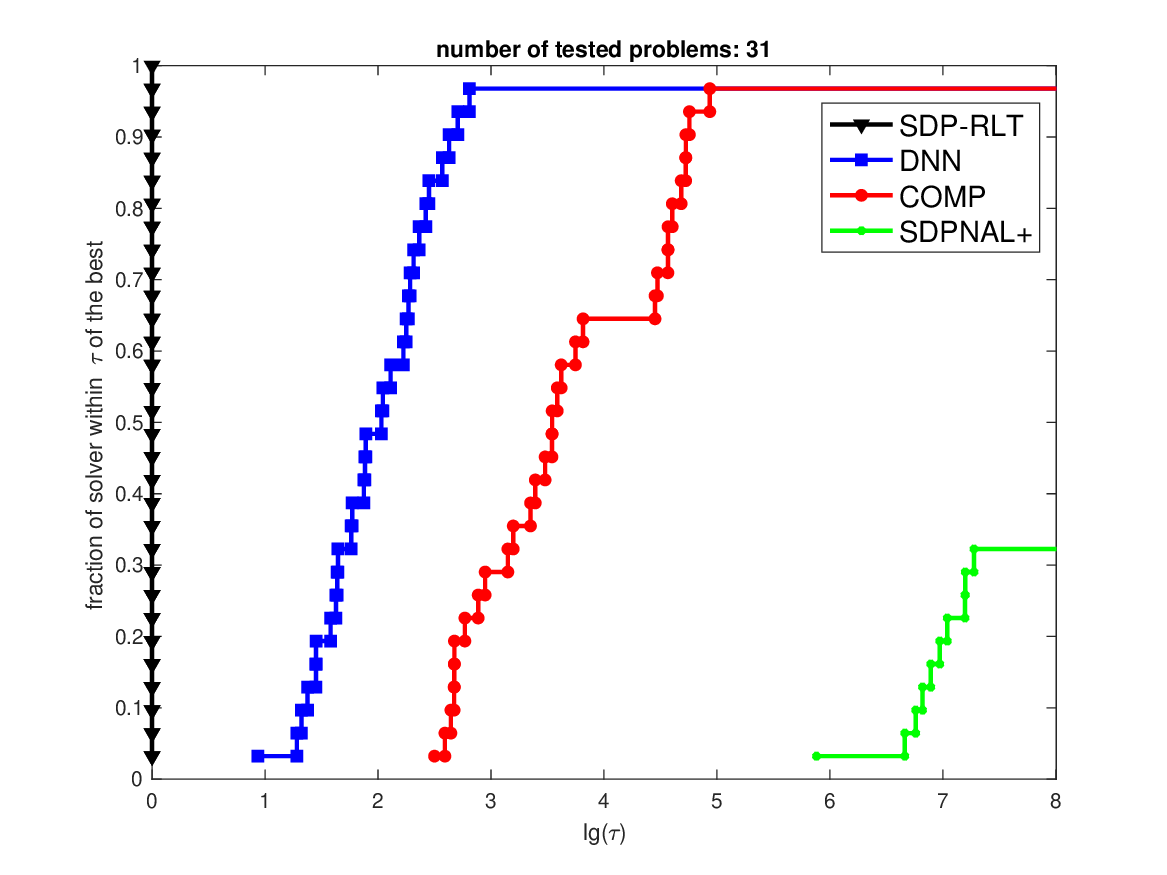}
        \caption{Strengthened BIQ problems.}
        \label{fig:pp-BIQ}
    \end{subfigure}
    \caption{Performance profile comparison for different formulations.}
    \label{fig-pp}
\end{figure}

The results indicate that the performance of RiNNAL+ on \eqref{SDP-RLT} outperforms both \eqref{DNN} and \eqref{COMP}. This advantage is expected, as the dimension of \eqref{SDP-RLT} is only half of that in \eqref{DNN} and \eqref{COMP}, and the corresponding manifold is simpler. For example, in BIQ problems, the manifold $\M_r$ corresponding to \eqref{SDP-RLT} is the oblique manifold defined as
\[
\mathcal{OB}(n, r) := \left\{ R \in \mathbb{R}^{n \times r} : \diag(RR^\top) = e \right\},
\]
which allows for straightforward projection and retraction operations. Consequently, solving \eqref{SDP-RLT} directly is generally more efficient than introducing slack variables for \eqref{DNN} or \eqref{COMP}. 
Furthermore, the results indicate that the speed of RiNNAL+ on solving \eqref{DNN} is faster than \eqref{COMP}. This is also expected, as \eqref{DNN} incorporates the constraint $\diag(X) = x$ directly into the manifold structure, while \eqref{COMP} relies on penalizing its equivalent complementarity constraints to the augmented Lagrangian function. On the one hand, the computational costs for performing retraction in \eqref{DNN} remain manageable due to the efficient convex reformulation technique used to compute the retraction subproblem, as detailed in \cite{RiNNAL}. On the other hand, embedding the constraint directly into the manifold significantly reduces the number of outer ALM and inner Riemannian gradient descent iterations, contributing to the overall speed advantage.

The results also highlight the advantage of RiNNAL+ over RiNNAL. While RiNNAL cannot be directly applied to solve \eqref{SDP-RLT}, it can be used to solve the equivalent problem \eqref{DNN}. However, as shown in Figure~\ref{fig-pp}, solving \eqref{DNN} is less efficient than solving \eqref{SDP-RLT}.

Due to the discussion above, in the following subsections, we only report the performance of RiNNAL+ on the \eqref{SDP-RLT} formulation. For SDPNAL+, we always report the fastest computational results among the three formulations \eqref{SDP-RLT}, \eqref{DNN}, and \eqref{COMP} to ensure a fair comparison.

%%%%%%%%%%%%%%%%%%%%%%%%%%%%%%%%%%%%%%%%%%%%%%%%%%%%%%%%%%%%%%%%%%%%%%%%%%%%%%%

\subsection{Binary integer nonconvex quadratic programming}\label{subsec-BIQ}

Consider the following strengthened BIQ problem mentioned in subsection \ref{subsec-tightness-numerical}, which is a specific instance of \eqref{SMBQP} without equality constraints:
\begin{equation}\label{eq-BIQ}
\min\left\{x^{\top} Q x+2 c^{\top}x:\ x\leq e,\ x \in\{0,1\}^n\right\}.\tag{BIQ-S}
\end{equation}
We select the test data for $Q$ and $c$ from the ORLIB library mentioned in subsection \ref{subsec-tightness-numerical} and consider problems of dimension $n \in \{500, 1000, 2500\}$. Each dimension includes ten instances, but we only report results for the first one in this subsection, as the performance across the others is similar. Complete results for all instances can be found in Appendix \ref{appendix-BIQ}. Since no data for larger dimensions is available, we randomly generate a dataset with the problem size $n = 5000$, following the data generation procedure outlined by J.E. Beasley in \cite{beasley1998heuristic}.

\begin{footnotesize}
\begin{longtable}[c]{crrrccrrr}
\caption{Computational results for \eqref{SDP-RLT} relaxation of \eqref{eq-BIQ} problems. \label{tab-BIQ}} \\
\toprule
Problem&Algorithm&Iteration&Rank& \multicolumn{1}{c}{$\operatorname{R_{max}}$}&Objective&\multicolumn{1}{c}{Time}&\multicolumn{1}{c}{TPG}\\
\midrule
\endfirsthead

\multicolumn{9}{c}%
{{ Table \thetable\ continued from previous page}} \\
\toprule
Problem&Algorithm&Iteration&Rank& \multicolumn{1}{c}{$\operatorname{R_{max}}$}&Objective&\multicolumn{1}{c}{Time}&\multicolumn{1}{c}{TPG}\\
\midrule
\endhead
\midrule
\multicolumn{9}{r}{{Continued on next page}} \\
\midrule
\endfoot

\bottomrule
\endlastfoot

$n=500$& RiNNAL+ & 16, 850, 3 & 50 & 8.70e-07 & -1.2259545e+05 & 8.0 &0.9 \\ 
 & SDPNAL+ & 44, 97, 1535 & 76& 9.98e-07 & -1.2259531e+05 & 613.1 & - \\ [3pt] 

$n=1000$& RiNNAL+ & 12, 650, 3 & 79 & 9.58e-07 & -3.8983061e+05 & 19.3 &2.3 \\ 
 & SDPNAL+ & 530, 549, 5574 & 995& 3.83e-02$^{\dagger}$ & -4.1472242e+05 & 3600.0 & - \\ [3pt] 

$n=2500$& RiNNAL+ & 11, 600, 2 & 134 & 7.28e-07 & -1.6096512e+06 & 133.3 &15.9 \\ 
 & SDPNAL+ & - & - & - & - & - & - \\ [3pt]

$n=5000$& RiNNAL+ & 12, 650, 3 & 177 & 8.61e-07 & -4.6838013e+06 & 1103.1 &254.8 \\ 
 & SDPNAL+ & - & - & - & - & - & - \\ 

\end{longtable}
\end{footnotesize}

As shown in Table \ref{tab-BIQ}, RiNNAL+ successfully solves all instances problems to the required accuracy, whereas SDPNAL+ fails to solve problems with dimensions $n \geq 1000$ within the 1-hour limit. For medium-size problems, RiNNAL+ can be 180 times faster than SDPNAL+. Moreover, RiNNAL+ is capable of handling problems with dimensions as large as $n = 5000$ in about 18 minutes. These results highlight the efficiency of RiNNAL+ and its potential for solving large-scale BIQ problems. It is also noteworthy that the solution ranks are typically around 50-200, which are not particularly small. Furthermore, the computational cost of the PG step constitutes only a small fraction of the total time, owing to the preprocessing and warm-start techniques proposed in Section~\ref{sec-acceleration}.

%%%%%%%%%%%%%%%%%%%%%%%%%%%%%%%%%%%%%%%%%%%%%%%%%%%%%%%%%%%%%%%%%%%%%%%%%%%%%%%

\subsection{Maximum stable set problems}\label{subsec-theta}

Consider a graph $G$ with $n$ vertices and $m$ edges, where the edge set is denoted by $E \subseteq \{(i,j) \mid 1 \leq i < j \leq n\}$. The maximum stable set problem ($\theta_+$) is defined as follows:
\begin{equation}\label{prob-theta-v1}
\max\left\{x^{\top} x:\  x\leq e,\ x_ix_j=0\,\, \forall (i,j)\in E,\  x \in\{0,1\}^n\right\}.
\tag{$\theta_+$}
\end{equation}
We choose large sparse graphs from the Gset dataset\footnote{Dataset available at \url{https://web.stanford.edu/~yyye/yyye/Gset/}.} and coding theory\footnote{Dataset available at \url{https://oeis.org/A265032/a265032.html}.}. In this subsection, we present several representative graphs with different dimensions, while the complete results for all graphs are provided in Appendix \ref{appendix-theta}. 

\begin{footnotesize}
\begin{longtable}[c]{crrrccrrr}
\caption{Computational results for \eqref{SDP-RLT} relaxation of \eqref{prob-theta-v1} problems. \label{tab-theta}} \\
\toprule
Problem&Algorithm&Iteration&Rank& \multicolumn{1}{c}{$\operatorname{R_{max}}$}&Objective&\multicolumn{1}{c}{Time}&\multicolumn{1}{c}{TPG}\\
\midrule
\endfirsthead

\multicolumn{9}{c}%
{{ Table \thetable\ continued from previous page}} \\
\toprule
Problem&Algorithm&Iteration&Rank& \multicolumn{1}{c}{$\operatorname{R_{max}}$}&Objective&\multicolumn{1}{c}{Time}&\multicolumn{1}{c}{TPG}\\
\midrule
\endhead
\midrule
\multicolumn{9}{r}{{Continued on next page}} \\
\midrule
\endfoot

\bottomrule
\endlastfoot

G11& RiNNAL+ & 56, 2850, 11 & 8 & 6.41e-07 & -4.0000014e+02 & 47.7 &6.4 \\ 
$n=800$ & SDPNAL+ & 482, 722, 9358 & 2& 2.46e-07 & -4.0000005e+02 & 3008.3 & - \\ [3pt] 

G18& RiNNAL+ & 34, 1750, 7 & 75 & 9.08e-07 & -2.7900052e+02 & 34.0 &5.5 \\ 
$n=800$ & SDPNAL+ & 118, 558, 3700 & 69& 2.38e-05$^{\dagger}$ & -2.7899998e+02 & 3600.0 & - \\ [3pt] 

G54& RiNNAL+ & 46, 2350, 10 & 168 & 9.01e-07 & -3.4100141e+02 & 68.4 &11.3 \\ 
$n=1000$ & SDPNAL+ & 100, 372, 2500 & 110& 1.93e-04$^{\dagger}$ & -3.4090921e+02 & 3600.0 & - \\ [3pt] 

G30& RiNNAL+ & 13, 700, 3 & 98 & 9.00e-07 & -5.7703872e+02 & 87.7 &10.4 \\ 
$n=2000$ & SDPNAL+ & - & - & - & - & - & - \\ [3pt] 

G50& RiNNAL+ & 24, 1250, 5 & 134 & 9.39e-07 & -1.4940617e+03 & 754.8 &149.3 \\ 
$n=3000$ & SDPNAL+ & - & - & - & - & - & - \\ [3pt] 

1tc.1024& RiNNAL+ & 116, 5850, 24 & 288 & 9.38e-07 & -2.0420520e+02 & 192.4 &28.1 \\ 
$n=1024$ & SDPNAL+ & 100, 300, 3430 & 339& 9.43e-06$^{\dagger}$ & -2.0419958e+02 & 3600.0 & - \\ [3pt] 

1tc.2048& RiNNAL+ & 94, 4750, 26 & 520 & 9.85e-07 & -3.7049061e+02 & 820.7 &162.8 \\ 
$n=2048$ & SDPNAL+ & - & - & - & - & - & - \\ [3pt] 

\end{longtable}
\end{footnotesize}

As shown in Table \ref{tab-theta}, RiNNAL+ successfully solves all instances of problem \eqref{SDP-RLT} to the required accuracy, whereas SDPNAL+ is unable to solve Gset instances with dimensions $n \geq 1000$ and coding theory instances within the 1-hour time limit. For the first four problems, RiNNAL+ is over 40 times faster than SDPNAL+. For the second instance, RiNNAL+ is at least 100 times faster than SDPNAL+. Observe that the solutions of the last two instances have relatively high ranks (approximately $n/4$ to $n/3$), yet RiNNAL+ remains about 100 times faster than SDPNAL+. These results underscore RiNNAL+’s robustness in solving general SDP relaxations—even in high-rank scenarios.

%%%%%%%%%%%%%%%%%%%%%%%%%%%%%%%%%%%%%%%%%%%%%%%%%%%%%%%%%%%%%%%%%%%%%%%%%%%%%%%

\subsection{Quadratic knapsack problems}\label{subsec-QKP}

The quadratic knapsack problem (QKP), introduced by Gallo et al. in \cite{gallo1980quadratic}, is formulated as follows:
\begin{equation}\label{prob-qkp-v1}
\max\left\{x^{\top} Q x:\ a^\top x\leq \tau,\ x\leq e,\ x \in\{0,1\}^n\right\},
\end{equation}
where $Q\in \S^n$ is a nonnegative profit matrix, $a\in \R^n_{++}$ is the weight vector, and $\tau>0$ is the knapsack capacity. To consider both equality and inequality constraints, we convert the inequality constraint into an equality constraint, leading to the modified problem:
\begin{equation}\label{prob-qkp-v2}
\max\left\{x^{\top} Q x:\ a^\top x= \tau,\ x\leq e,\ x \in\{0,1\}^n\right\}.\tag{QKP-S}
\end{equation}
The new problem \eqref{prob-qkp-v2} provides a lower bound for Problem \eqref{prob-qkp-v1}. When $a=e$ and $\tau=k$, \eqref{prob-qkp-v2} reduces to the $k$-subgraph problem. We randomly generate the profit matrix $Q$ and weight vector $a$ following the procedure proposed by Gallo et al. in \cite{gallo1980quadratic}, which has been widely adopted in the literature (see, for example, \cite{caprara1999exact,billionnet2004exact,pisinger2007quadratic,tang2024feasible}). The entries of the profit matrix $Q_{ij} = Q_{ji}$ are randomly generated as integers uniformly distributed in the range $[1, 100]$ with probability $p$, and zero otherwise. The elements of the weight vector $a$ are randomly selected integers in the range $[1, 50]$. The knapsack capacity is set to $0.9 \cdot e^\top a$, and the probability $p$ is chosen from $\{0.1, 0.5, 0.9\}$. The problem dimensions $n$ are selected from $\{500, 1000, 2000, 5000\}$.
Results for SDPNAL+ are not reported, as it fails to achieve the required accuracy within the 1-hour time limit, even for the smallest problem with $n = 500$.

\begin{footnotesize}
\begin{longtable}[c]{rrrccrrr}
\caption{Computational results for \eqref{SDP-RLT} relaxation of \eqref{prob-qkp-v2} problems. \label{tab-QKP}} \\
\toprule
\multicolumn{1}{c}{$n,p$}&Iteration&Rank& \multicolumn{1}{c}{$\operatorname{R_{max}}$}&Objective&\multicolumn{1}{c}{Time}&\multicolumn{1}{c}{TPG}\\
\midrule
\endfirsthead

\multicolumn{8}{c}%
{{ Table \thetable\ continued from previous page}} \\
\toprule
$n,p$&Iteration&Rank& \multicolumn{1}{c}{$\operatorname{R_{max}}$}&Objective&\multicolumn{1}{c}{Time}&\multicolumn{1}{c}{TPG}\\
\midrule
\endhead
\midrule
\multicolumn{8}{r}{{Continued on next page}} \\
\midrule
\endfoot

\bottomrule
\endlastfoot

500, 0.1& 87, 4400, 18 & 26 & 9.91e-07 & -1.1420448e+06 & 40.5 &4.3 \\ [1pt]
500, 0.5& 20, 1050, 4 & 16 & 8.36e-07 & -5.6675029e+06 & 9.3 &0.9 \\ [1pt]
500, 0.9& 41, 2100, 8 & 13 & 7.11e-07 & -1.0260964e+07 & 18.3 &2.0 \\ [1pt]
1000, 0.1& 27, 1400, 6 & 37 & 9.29e-07 & -4.5986965e+06 & 53.1 &5.5 \\ [1pt]
1000, 0.5& 31, 1600, 6 & 21 & 9.74e-07 & -2.2747042e+07 & 53.3 &6.0 \\ [1pt]
1000, 0.9& 16, 850, 3 & 11 & 8.86e-07 & -4.0934623e+07 & 28.8 &3.3 \\ [1pt]
2000, 0.1& 15, 800, 3 & 42 & 6.84e-07 & -1.8316334e+07 & 123.7 &22.7 \\ [1pt]
2000, 0.5& 17, 900, 4 & 22 & 9.28e-07 & -9.0934882e+07 & 140.9 &28.1 \\ [1pt]
2000, 0.9& 24, 1250, 5 & 29 & 5.08e-07 & -1.6341805e+08 & 221.8 &45.8 \\ [1pt]
5000, 0.1& 18, 950, 4 & 77 & 7.89e-07 & -1.1399134e+08 & 1761.8 &517.4 \\ [1pt]
5000, 0.5& 14, 750, 3 & 36 & 9.66e-07 & -5.6821246e+08 & 1429.3 &468.0 \\ [1pt]
5000, 0.9& 17, 900, 4 & 20 & 4.45e-07 & -1.0223365e+09 & 2005.0 &779.0 \\

\end{longtable}
\end{footnotesize}

As shown in Table \ref{tab-QKP}, RiNNAL+ successfully solves all instances. In some cases, such as for $n=500$ and $p=0.5$, RiNNAL+ is about 400 times faster than SDPNAL+. Furthermore, RiNNAL+ can handle large QKP problems with $n=5000$ in approximately 30 minutes. Note that in the low-rank phase of RiNNAL+, computing the tangent space projection and retraction onto the manifold $\M_r$ for QKP problems is generally nontrivial and can be expensive. However, by applying the strategies described in Subsection \ref{subsec-LR-phase}, the time required for these operations is reduced to approximately $10\%$ to $20\%$ of the total computation time. This small fraction highlights the efficiency of our approach in performing projection and retraction.

%%%%%%%%%%%%%%%%%%%%%%%%%%%%%%%%%%%%%%%%%%%%%%%%%%%%%%%%%%%%%%%%%%%%%%%%%%%%%%%

\subsection{Cardinality-constrained minimum sum-of-squares clustering}\label{subsec-ccMSCC}

The cardinality-constrained minimum sum-of-squares clustering problem (ccMSSC), as discussed in \cite{piccialli2023global}, aims to partition $m$ data points $p_1, \ldots, p_m \in \mathbb{R}^d$, into $k$ clusters with predetermined sizes $c_1, \ldots, c_k$ such that $\sum_{j=1}^k c_j = m$. The objective is to minimize the total sum of squared intra-cluster distances. This problem can be formulated as:
\begin{equation}\label{prob-ccMSCC-v1}
\min\left\{\sum_{j=1}^{k}\frac{1}{c_j} \sum_{s=1}^m\sum_{t=1}^m d_{st} \pi^{(s)}_j\pi_j^{(t)}:\ 
\pi\in \mathcal{S},\; \pi\leq e_{m\times k},\; \pi \in \{0,1\}^{m\times k} \right\},
\tag{ccMSSC}
\end{equation}
where $e_{m\times k}\in\R^{m\times k}$ is a matrix of all ones, and $d_{st}=\|p_s - p_t\|^2$. The variable $\pi = [\pi^{(1)}, \ldots, \pi^{(k)}]$, where each $\pi^{(j)} = [\pi^{(j)}_1, \ldots, \pi^{(j)}_m]^\top$ is the indicator vector for cluster $j$. The set $\mathcal{S}$ is defined as:
\begin{equation*}
\mathcal{S}:=\left\{\pi\in\R^{m\times k}:\  
\sum_{j=1}^k \pi^{(j)}_i= 1,\ \forall i\in[m],\ 
\sum_{i=1}^m \pi^{(j)}_i= c_j,\ \forall j\in[k] \right\}.
\end{equation*}
Additionally, we impose the constraint $\pi^{(j)}_i \leq 1$ for all $i \in [m]$ and $j \in [k]$ to \eqref{prob-ccMSCC-v1} 
to strengthen the corresponding \eqref{SDP-RLT} relaxation
as outlined in Example \ref{example-SMBQP}. Table \ref{tab-data-ccMSSC} presents 28 real-world classification datasets from the UCR2 website\footnote{Dataset available at \url{https://www.cs.ucr.edu/~eamonn/time_series_data_2018/}.}, characterized by the number of data points $m \in [180, 1370]$, the number of features $d \in [24, 2709]$, and the number of clusters $k = \{2,3,4\}$. We define $n = mk$ to denote the dimension of \eqref{MBQP}. Since generating the constraint matrix for SDPNAL+ is time-consuming for large $n$, we limit SDPNAL+ to solving relaxation problems of \eqref{MBQP} for $n \leq 1000$. Results for datasets with $\text{ID} \in\{5, 10, 15, 20, 25\}$ are presented in this subsection, and complete results for all datasets can be found in Appendix \ref{appendix-ccMSSC}.

\begin{small}%footnotesize
\begin{longtable}[c]{clrrrcr}
\caption{Real-world classification instances.} \label{tab-data-ccMSSC}\\
\toprule
 ID  & Dataset & $n$ & $m$ & $d$ & $k$ & $c_1,\dots,c_k$  \\
\midrule
\endfirsthead

\multicolumn{7}{c}%
{{ Table \thetable\ continued from previous page}} \\
\toprule
  ID  & Dataset & n & m & d & k & $c_1,\dots,c_k$  \\
\midrule
\endhead
\midrule
\multicolumn{7}{r}{{Continued on next page}} \\
\midrule
\endfoot
\bottomrule
\endlastfoot

01 & WormsTwoClass & 516 & 258 & 900 & 2 & 109,  149 \\ 
02 & ToeSegmentation1 & 536 & 268 & 277 & 2 & 140,  128 \\ 
03 & BME & 540 & 180 & 128 & 3 & 60,  60,  60 \\ 
04 & UMD & 540 & 180 & 150 & 3 & 60,  60,  60 \\ 
05 & PowerCons & 720 & 360 & 144 & 2 & 180,  180 \\ 
06 & Chinatown & 726 & 363 & 24 & 2 & 104,  259 \\ 
07 & InsectEPGSmallTrain & 798 & 266 & 601 & 3 & 95,  126,  45 \\ 
08 & Trace & 800 & 200 & 275 & 4 & 50,  50,  50,  50 \\ 
09 & GunPointAgeSpan & 902 & 451 & 150 & 2 & 228,  223 \\ 
10 & GunPointMaleVersusFemale & 902 & 451 & 150 & 2 & 237,  214 \\ 
11 & GunPointOldVersusYoung & 902 & 451 & 150 & 2 & 215,  236 \\ 
12 & Earthquakes & 922 & 461 & 512 & 2 & 368,  93 \\ 
13 & InsectEPGRegularTrain & 933 & 311 & 601 & 3 & 111,  148,  52 \\ 
14 & Computers & 1000 & 500 & 720 & 2 & 250,  250 \\ 
15 & MiddlePhalanxOutlineAgeGroup & 1662 & 554 & 80 & 3 & 92,  196,  266 \\ 
16 & DistalPhalanxOutlineCorrect & 1752 & 876 & 80 & 2 & 337,  539 \\ 
17 & ECGFiveDays & 1768 & 884 & 136 & 2 & 442,  442 \\ 
18 & MiddlePhalanxOutlineCorrect & 1782 & 891 & 80 & 2 & 337,  554 \\ 
19 & ProximalPhalanxOutlineCorrect & 1782 & 891 & 80 & 2 & 286,  605 \\ 
20 & SemgHandGenderCh2 & 1800 & 900 & 1500 & 2 & 540,  360 \\ 
21 & ProximalPhalanxOutlineAgeGroup & 1815 & 605 & 80 & 3 & 89,  227,  289 \\ 
22 & SonyAIBORobotSurface2 & 1960 & 980 & 65 & 2 & 376,  604 \\ 
23 & Strawberry & 1966 & 983 & 235 & 2 & 351,  632 \\ 
24 & LargeKitchenAppliances & 2250 & 750 & 720 & 3 & 250,  250,  250 \\ 
25 & RefrigerationDevices & 2250 & 750 & 720 & 3 & 250,  250,  250 \\ 
26 & SmallKitchenAppliances & 2250 & 750 & 720 & 3 & 250,  250,  250 \\ 
27 & TwoLeadECG & 2324 & 1162 & 82 & 2 & 581,  581 \\ 
28 & HandOutlines & 2740 & 1370 & 2709 & 2 & 495,  875 \\ 

\end{longtable}
\end{small}

\begin{footnotesize}
\begin{longtable}[c]{crrrccrrr}
\caption{Computational results for \eqref{SDP-RLT} relaxation of \eqref{prob-ccMSCC-v1} problem.} \label{tab-ccMSSC}\\
\toprule
Problem&Algorithm&Iteration&Rank& \multicolumn{1}{c}{$\operatorname{R_{max}}$}&Objective&\multicolumn{1}{c}{Time}&\multicolumn{1}{c}{TPG}\\
\midrule
\endfirsthead

\multicolumn{9}{c}%
{{ Table \thetable\ continued from previous page}} \\
\toprule
Problem&Algorithm&Iteration&Rank& \multicolumn{1}{c}{$\operatorname{R_{max}}$}&Objective&\multicolumn{1}{c}{Time}&\multicolumn{1}{c}{TPG}\\
\midrule
\endhead
\midrule
\multicolumn{9}{r}{{Continued on next page}} \\
\midrule
\endfoot

\bottomrule
\endlastfoot

5& RiNNAL+ & 9, 500, 2 & 5 & 9.27e-07 & 7.3254609e+04 & 14.3 &1.4 \\ 
$n=720$& SDPNAL+ & 201, 234, 10487 & 17& 8.09e-06$^{\dagger}$ & 7.3254549e+04 & 3600.0 & - \\ [3pt] 

10& RiNNAL+ & 3, 200, 1 & 7 & 8.78e-07 & 4.0536895e+09 & 12.6 &1.1 \\ 
$n=902$& SDPNAL+ & 54, 246, 2312 & 1& 7.91e-07 & 4.0536886e+09 & 2313.4 & - \\ [3pt] 

15& RiNNAL+ & 12, 511, 3 & 21 & 3.90e-07 & 8.5151301e+02 & 1500.3 &193.9 \\ 
$n=1662$& SDPNAL+ & - & - & - & - & - & - \\ [3pt] 

20& RiNNAL+ & 2, 150, 1 & 28 & 5.87e-07 & 7.3610392e+08 & 42.2 &13.7 \\ 
$n=1800$& SDPNAL+ & - & - & - & - & - & - \\ [3pt] 

25& RiNNAL+ & 51, 2600, 11 & 61 & 5.42e-07 & 1.0500893e+06 & 2215.9 &170.8 \\ 
$n=2250$& SDPNAL+ & - & - & - & - & - & - \\ [3pt] 

\end{longtable}
\end{footnotesize}

As shown in Table \ref{tab-ccMSSC}, RiNNAL+ outperforms SDPNAL+ on all instances that the latter can solve with $n\leq 1000$. Notably, for the first instance, RiNNAL+ is 350 times faster than SDPNAL+. Additionally, the number of equality constraints in \eqref{prob-ccMSCC-v1} is $k+m$, which is relatively large, making the structure of $\M_r$ more complex. Despite this, our approach efficiently solves all instances within the one-hour time limit.

%%%%%%%%%%%%%%%%%%%%%%%%%%%%%%%%%%%%%%%%%%%%%%%%%%%%%%%%%%%%%%%%%%%%%%%%%%%%%%%

\subsection{Sparse standard quadratic programming problems}\label{subsec-SStQP}

The sparse standard quadratic programming problem considered in \cite{bomze2024tighter} is given as follows:
\begin{equation}\label{prob-sstqp-v1}
\min\left\{x^{\top} Q x:\ e^\top x= 1,\ \|x\|_0\leq \rho,\ x\in\R^m_+\right\},
\end{equation}
where $Q\in \S^m$, $\rho \in [m]$ is the sparsity parameter. As introduced in Example \ref{example-L0}, there are two DNN relaxations of \eqref{prob-sstqp-v1}. The first one uses the Big-M reformulation to convert \eqref{prob-sstqp-v1} into the following equivalent problem:
\begin{equation}\label{prob-sstqp-v2}
\min\left\{x^{\top} Q x:\ e^\top x= 1,\ e^\top u=\rho,\ x\leq u\leq e,\ x\in\R_+^m,\ u \in\{0,1\}^m\right\}.
\tag{SStQP}
\end{equation}
The second one uses the complementarity reformulation and considers the following equivalent problem:
\begin{equation}\label{prob-sstqp-v3}
\min\left\{x^{\top} Q x:\  x^\top v=0,\ e^\top v=m-\rho,\ x\leq e,\ v\geq 0,\ v\in \{0,1\}^m \right\}.
\end{equation}
Although \eqref{prob-sstqp-v3} includes the nonlinear constraint $x^\top v = 0$ and does not fit the structure of \eqref{MBQP}, as noted in Remark \ref{remark-QCQP}, our algorithm can readily be extended to handle the SDP relaxations of a QCQP including \eqref{prob-sstqp-v3}. 
We randomly generate the matrix $Q$ and the sparsity parameter $\rho$ following the method proposed by Bomze et al. in \cite{bomze2024tighter}. Specifically, the matrix $Q$ is generated using three different approaches, referred to as COP, PSD, and SPN. For each approach, the data are generated so that the standard quadratic optimization problem (i.e., problem \eqref{prob-sstqp-v1} without the sparsity constraint) admits a unique global minimizer with support size corresponding to a target sparsity level $\rho_0 = m/4$.  
To evaluate algorithm performance under stricter sparsity constraints, we set the actual sparsity parameter to $\rho = \rho_0 / 2$. The problem dimension $m$ is chosen from the set $\{100, 200, 500\}$.
We use $n = 2m$ to denote the variable dimension of \eqref{MBQP}. We apply RiNNAL+ and SDPNAL+ to solve the \eqref{SDP-RLT} relaxation of \eqref{prob-sstqp-v2}. The results for the \eqref{SDP-RLT} relaxation of \eqref{prob-sstqp-v3} are not reported, as it yields the same bound as that for \eqref{prob-sstqp-v2} but requires more computational time. 
Additionally, we exclude the result for the SPN instance with dimension $n=1000$ since neither RiNNAL+ nor SDPNAL+ could solve it within the one-hour time limit.

\begin{footnotesize}
\begin{longtable}[c]{crrrccrrr}
\caption{Computational results for \eqref{SDP-RLT} relaxation of \eqref{prob-sstqp-v2} problems. \label{tab-sStQP}} \\
\toprule
Problem&Algorithm&Iteration&Rank& \multicolumn{1}{c}{$\operatorname{R_{max}}$}&Objective&\multicolumn{1}{c}{Time}&\multicolumn{1}{c}{TPG}\\
\midrule
\endfirsthead

\multicolumn{9}{c}%
{{ Table \thetable\ continued from previous page}} \\
\toprule
Problem&Algorithm&Iteration&Rank& \multicolumn{1}{c}{$\operatorname{R_{max}}$}&Objective&\multicolumn{1}{c}{Time}&\multicolumn{1}{c}{TPG}\\
\midrule
\endhead
\midrule
\multicolumn{9}{r}{{Continued on next page}} \\
\midrule
\endfoot

\bottomrule
\endlastfoot

COP& RiNNAL+ & 125, 6300, 25 & 102 & 6.64e-07 & -1.0190589e-01 & 30.7 &2.1 \\ 
$n=200$ & SDPNAL+ & 76, 112, 1170 & 102& 8.59e-07 & -1.0196083e-01 & 40.1 & - \\ [3pt] 

PSD& RiNNAL+ & 374, 18750, 75 & 5 & 8.59e-07 & 1.4174794e-02 & 80.7 &14.5 \\ 
$n=200$ & SDPNAL+ & 801, 1058, 17306 & 2& 8.50e-07 & 1.4174999e-02 & 468.1 & - \\ [3pt] 

SPN& RiNNAL+ & 747, 37400, 150 & 107 & 9.31e-07 & 1.3592384e-02 & 204.0 &20.8 \\ 
$n=200$ & SDPNAL+ & 801, 1032, 29187 & 12& 9.94e-07 & 1.4887842e-02 & 866.1 & - \\ [3pt] 

COP& RiNNAL+ & 300, 15050, 60 & 202 & 9.99e-07 & -1.0402158e-01 & 230.5 &15.8 \\ 
$n=400$ & SDPNAL+ & 103, 145, 1511 & 202& 9.24e-07 & -1.0445084e-01 & 170.7 & - \\ [3pt] 

PSD& RiNNAL+ & 360, 18050, 72 & 108 & 9.50e-07 & 9.0349366e-03 & 231.0 &19.4 \\ 
$n=400$ & SDPNAL+ & 801, 905, 36057 & 89& 1.84e-06$^{\dagger}$ & 9.3529559e-03 & 3600.0 & - \\ [3pt] 

SPN& RiNNAL+ & 1812, 90650, 363 & 248 & 9.97e-07 & -1.2530973e-02 & 1556.1 &145.7 \\ 
$n=400$ & SDPNAL+ & 801, 896, 38599 & 25& 1.02e-06$^{\dagger}$ & 1.3190619e-02 & 3600.0 & - \\ [3pt] 

COP& RiNNAL+ & 672, 33650, 135 & 621 & 1.20e-04$^{\dagger}$ & -4.8722936e-01 & 3600.0 &323.1 \\ 
$n=1000$ & SDPNAL+ & 200, 227, 2350 & 502& 9.28e-07 & -1.0466603e-01 & 1779.7 & - \\ [3pt] 

PSD& RiNNAL+ & 507, 25400, 102 & 234 & 9.88e-07 & 2.9188026e-03 & 1838.0 &152.6 \\ 
$n=1000$ & SDPNAL+ & 426, 430, 4726 & 160& 7.73e-06$^{\dagger}$ & 1.0912535e-02 & 3600.0 & - \\ [3pt] 

\end{longtable}
\end{footnotesize}

As shown in Table \ref{tab-sStQP},
RiNNAL+ consistently outperforms SDPNAL+ across all instances, except for the COP instance. This is expected because the solution rank of the ALM subproblem arising from the \eqref{SDP-RLT} relaxation for this type of sparse StQP problem is very high, nearly $n/2$. However, for problems with medium solution rank, such as the SPN and PSD instances, RiNNAL+ can still be faster than SDPNAL+—about 4 times faster for the SPN instance with $n=400$ and 15 times faster for the PSD instance with $n=200$. These results further highlight RiNNAL+'s effectiveness in solving \eqref{SDP-RLT}. Note that the objective function values obtained by different algorithms vary noticeably in some instances. This anomaly, rare among other problem classes, may be attributed to the unique properties of sparse QP problems. In our experiments, we observed that the objective function value stabilizes only when the KKT residual is below $10^{-8}$.

%%%%%%%%%%%%%%%%%%%%%%%%%%%%%%%%%%%%%%%%%%%%%%%%%%%%%%%%%%%%%%%%%%%%%%%%%%%%%%%

\subsection{Quadratic minimum spanning tree problem}\label{subsec-QMSTP}

The Quadratic Minimum Spanning Tree Problem (QMSTP) introduced by Assad and Xu \cite{assad1992quadratic} aims to find the minimizer of a quadratic function over all possible spanning trees of a graph. For a connected, undirected graph $G=(V,E)$ with $n=|V|$ vertices and $m=|E|$ edges, let $Q = (q_{ef}) \in \S^{m}$ denotes a matrix of interaction costs between the edges of $G$, where each $q_{ee}$ denotes the cost associated with edge $e$. Then QMSTP can be formulated as follows:
\begin{equation}\label{eq-QMSTP-org}
    \min\left\{ x^\top Qx:  x \in \mathcal{T}\right\},
\end{equation}
where $\mathcal{T}$ is the collection of all spanning trees in the graph $G$ defined by
\[
\mathcal{T} \coloneqq \left\{ x \in \{0,1\}^{m} : \sum_{e\in E} x_e = n-1, \, \sum_{e \in \partial S} x_e \geq 1, \, \forall S \subsetneq V, \, S \neq \emptyset \right\},
\]
and  $\partial S \coloneqq \big\{ \{i,j\} \in E ~ : ~ i \in S, j\notin S \big\}$ denotes the cut induced by $S$. 
The constraints
\[
\sum_{e \in \partial S} x_e \geq 1
\]
are referred to as cut-set constraints, ensuring that each subset $S$ connects to the remainder of the graph, thereby maintaining the connectivity of any subgraph in $\mathcal{T}$. If the matrix $Q$ is diagonal, the QMSTP simplifies to the classical minimum spanning tree problem, which is solvable in polynomial time~\cite{kruskal1956shortest,prim1957shortest}. 

Although \eqref{eq-QMSTP-org} follows the structure of \eqref{MBQP}, handling the $2^{n-1}$ cut-set constraints in $\mathcal{T}$ is computationally prohibitive. To reduce the number of constraints, Meijer et al. \cite{de2024spanning} proposed to consider only a subset of the cut-set constraints where $|S|=1$. This simplification leads to the following relaxed feasible set:
\[
\mathcal{T}' \coloneqq \left\{ x \in \{0,1\}^{m} : x\leq e,\ \sum_{e\in E} x_e = n-1, \, \sum_{e \in \partial S} x_e \geq 1, \, \forall S \subsetneq V, \,  |S|=1 \right\},
\]
which only has $n$ cut-set constraints. The corresponding relaxed QMSTP problem is
\begin{equation}\label{eq-QMSTP-relax}
    \min\left\{ x^\top Qx:  x \in \mathcal{T}'\right\}.
    \tag{QMSTP}
\end{equation}
Since \eqref{eq-QMSTP-relax} also conforms to the structure of \eqref{MBQP}, we can derive its \eqref{SDP-RLT} relaxation. It is worth noting that \cite{de2024spanning} introduced two DNN relaxations of \eqref{eq-QMSTP-relax}, one without and one with some RLT-type constraints, which can be viewed as partial SDP-RLT constraints. Consequently, the SDP-RLT relaxation of \eqref{eq-QMSTP-relax} proposed in this paper is stronger than those presented in \cite{de2024spanning}. Numerical experiments further demonstrate that our \eqref{SDP-RLT} relaxation can be strictly tighter than the DNN relaxation in \cite{de2024spanning}.

We test our algorithm on the OP benchmark set from the Mendeley data website\footnote{Dataset available at \url{https://data.mendeley.com/datasets/cmnh9xc6wb/1}.}, which was introduced by \"Oncan and Punnen \cite{oncan2010quadratic}. The dataset includes three different classes of complete graph instances: OPsym, OPvsym, and OPesym. The generation procedure is as follows:
\begin{enumerate}
    \item OPsym instances: the diagonal entries are chosen uniformly from $[100] := \{1,2,\ldots,100\}$, while the off-diagonal values are independently sampled uniformly from $[20]$.
    \item OPvsym instances: the diagonal entries are uniformly drawn from $[10,000]$. The off-diagonal elements $Q_{\{i,j\},\{k,l\}}$ are computed as $w(i) w(j) w(k) w(l)$, where the function $w \colon V \to [10]$ assigns a random weight uniformly from $[10]$ to each vertex in the graph.
    \item OPesym instances: consider random vertex coordinates within the box $[0,100] \times [0,100]$, with edges defined as straight-line segments connecting vertices. The cost for each edge $Q_{ee}$ is set to the edge length, while the interaction cost between two edges $e$ and $f$ is calculated as the Euclidean distance between their midpoints.
\end{enumerate}
For each of these instance types, we choose 10 random instances with $n \in \{30, 50\}$. We apply SDPNAL+ to solve the relaxation \eqref{SDP-RLT} of \eqref{eq-QMSTP-relax}. 
The average results for instances of the same dimension and data type are presented in this subsection, and complete results for all instances can be found in Appendix \ref{appendix-QMSTP}.

\begin{footnotesize}
\begin{longtable}[c]{crrrccrrr}
\caption{Computational results for \eqref{SDP-RLT} relaxation of \eqref{eq-QMSTP-relax} problems. \label{tab-QMSTP}} \\
\toprule
Problem&Algorithm&Iteration&Rank& \multicolumn{1}{c}{$\operatorname{R_{max}}$}&Objective&\multicolumn{1}{c}{Time}&\multicolumn{1}{c}{TPG}\\
\midrule
\endfirsthead

\multicolumn{9}{c}%
{{ Table \thetable\ continued from previous page}} \\
\toprule
Problem&Algorithm&Iteration&Rank& \multicolumn{1}{c}{$\operatorname{R_{max}}$}&Objective&\multicolumn{1}{c}{Time}&\multicolumn{1}{c}{TPG}\\
\midrule
\endhead
\midrule
\multicolumn{9}{r}{{Continued on next page}} \\
\midrule
\endfoot

\bottomrule
\endlastfoot

vsym& RiNNAL+ & 29, 1480, 6 & 5 & 4.86e-07 & 7.7804931e+04 & 12.9 &1.7 \\ 
$n=435$ & SDPNAL+ & 148, 230, 3713 & 1& 2.91e-07 & 7.7804694e+04 & 299.7 & - \\ [3pt] 

sym& RiNNAL+ & 103, 5190, 21 & 197 & 7.70e-07 & 5.3262719e+03 & 99.3 &2.8 \\ 
$n=435$ & SDPNAL+ & 96, 109, 1486 & 195& 1.09e-06 & 5.3262716e+03 & 130.5 & - \\ [3pt] 

esym& RiNNAL+ & 5519, 275985, 1104 & 40 & 7.53e-05$^{\dagger}$ & 7.4272770e+03 & 3600.0 &209.8 \\ 
$n=435$ & SDPNAL+ & 529, 692, 10111 & 46& 9.55e-07 & 7.4273883e+03 & 1083.3 & - \\ [3pt] 

vsym& RiNNAL+ & 96, 4850, 22 & 8 & 7.99e-07 & 1.6444786e+05 & 289.2 &64.3 \\ 
$n=1225$ & SDPNAL+ & 115, 198, 4142 & 602& 3.36e-04$^{\dagger}$ & 1.6561422e+05 & 3600.0 & - \\ [3pt] 

sym& RiNNAL+ & 84, 4275, 17 & 509 & 8.07e-07 & 1.4104759e+04 & 484.0 &21.2 \\ 
$n=1225$ & SDPNAL+ & 192, 199, 2628 & 507& 1.22e-06 & 1.4104759e+04 & 1987.4 & - \\ [3pt] 

\end{longtable}
\end{footnotesize}

As shown in Table \ref{tab-QMSTP}, RiNNAL+ fails to solve the esym instances, while SDPNAL+ fails to solve the vsym instances with dimension $m=1225$ within the one-hour time limit. 
Except for the esym instance, RiNNAL+ consistently outperforms SDPNAL+ across all problems. 
On some instances with low-rank solutions, such as the vsym instances, RiNNAL+ is about 10 to 20 times faster than SDPNAL+. 
Remarkably, RiNNAL+ is also 4 times faster than SDPNAL+ for the sym instance with $m=1225$, even though its solution has a very high rank that is close to $n/2$. 

We also compare our algorithm with the Peaceman-Rachford Splitting Method (PRSM) proposed in \cite{de2024spanning}. The computational experiments in \cite{de2024spanning} were performed on an AMD EPYC 7343 processor with 16 cores at 4.00 GHz and 1024 GB of RAM, running Debian GNU/Linux 11. PRSM uses stopping criteria based on gap closure, cut count, and iteration limits, whereas ours is based on the KKT residue. Following the criteria in \cite{de2024spanning}, we define relative gap as
\[
\%\text{gap}:=\frac{UB-LB}{UB}\times 100\%,
\]
where UB denotes the best-known upper bound from the literature\footnote{Data available at \url{https://homes.di.unimi.it/cordone/research/qmst.html}.}, and LB denotes the lower bound from different relaxations. As shown in Table \ref{tab:QMSTP-compare-PRSM}, the average lower bound obtained by our SDP-RLT relaxation is strictly tighter than that of the partial SDP-RLT relaxations proposed in \cite{de2024spanning}, particularly in cases like vsym and esym. Moreover, despite the fact that PRSM does not verify global optimality and is run on a more advanced processor with a greater core count and memory, our algorithm demonstrates significant computational advantages, achieving speed improvements of up to 800 times compared to PRSM and 23 times compared to SDPNAL+ in certain instances, such as in the vsym case. When a method fails to reach the required accuracy, we mark the corresponding LB with a dagger ``$\dagger$''. In such cases, the optimality gap is not reported in the table.

\begin{table}[h!]
    \centering
    \begin{tabular}{|c|r|r|r|r|}
    \hline
        Problem&Algorithm & \multicolumn{1}{c|}{LB} & {\%gap} &\multicolumn{1}{c|}{Time}\\
        
        \hline
        vsym&RiNNAL+ & 77804.93 &1.51 & 12.9\\
        $m=435$&SDPNAL+ & 77804.69 & 1.51 &299.7\\
        UB = 78999.90&PRSM & 76507.06 & 3.16 & 10781.1\\
        
        \hline
        sym&RiNNAL+ & 5326.27 & 11.46 & 99.3\\
        $m=435$&SDPNAL+ & 5326.27 & 11.46 & 130.5\\
        UB = 6015.90&PRSM & 5326.24 & 11.46 &102.2\\
        
        \hline
        esym&RiNNAL+ & 7427.28$^\dagger$ & - & 3600.0\\
        $m=435$&SDPNAL+ & 7427.39 & 7.81 & 1083.3 \\
        UB = 8056.70&PRSM & 7174.26 & 10.95 & 10045.1\\
        
        \hline
        vsym&RiNNAL+ & 164447.86 & 0.59 & 289.2\\
        $m=1225$&SDPNAL+ &  165614.22$^\dagger$ & - & 3600.0\\
        UB = 165419.60&PRSM & 154950.21 & 6.33 & 10835.9\\
        
        \hline
        sym&RiNNAL+ & 14104.76 & 19.94 & 484.0\\
        $m=1225$&SDPNAL+ & 14104.76 & 19.94 & 1987.4\\
        UB = 17616.90&PRSM & 14104.71 & 19.94 & 440.5\\
        \hline
    \end{tabular}
    \caption{Average performance comparison on the 
    OP benchmark dataset for \eqref{eq-QMSTP-relax}.}
    \label{tab:QMSTP-compare-PRSM}
\end{table}

%%%%%%%%%%%%%%%%%%%%%%%%%%%%%%%%%%%%%%%%%%%%%%%%%%%%%%%%%%%%%%%%%%%%%%%%%%%%%%%
%%%%%%%%%%%%%%%%%%%%%%%%%%%%%%%%%%%%%%%%%%%%%%%%%%%%%%%%%%%%%%%%%%%%%%%%%%%%%%%

\section{Conclusions}\label{sec-conclusion}

In this paper, we established the equivalence between the SDP-RLT and DNN relaxations of general mixed-binary quadratic programming problems. 
We propose RiNNAL+, a Riemannian ALM for solving general SDP relaxations \eqref{prob-DNN-general}, including the SDP-RLT and DNN relaxations. 
RiNNAL+ is an enhanced version of RiNNAL proposed in \cite{RiNNAL}.
By exploiting the low-rank structure of solutions, RiNNAL+ significantly reduces the computational complexity of solving large-scale semidefinite relaxations. 
We propose a two-phase framework that combines a low-rank decomposition phase to decrease the computational cost by utilizing the potential low-rank property of the solutions to the ALM subproblems and a convex lifting phase to guarantee convergence through projected gradient steps, as well as to automatically adjust the rank of the 
factorized variable in the low-rank phase.
Moreover, techniques such as random perturbation are employed in the low-rank phase to avoid nonsmoothness, while preprocessing is applied in the convex lifting phase to mitigate large condition numbers of the linear systems encountered when solving the ALM subproblems. A warm-start strategy is also implemented to accelerate convergence in the convex lifting phase. As a result, RiNNAL+ achieves robust numerical performance.
Extensive numerical experiments validated the computational efficiency and scalability of RiNNAL+ across various problem classes. 
These results highlight the potential of RiNNAL+ for addressing large-scale semidefinite relaxations within the branch-and-bound framework for solving challenging optimization problems involving quadratic constraints and mixed-integer variables. In the future, we will explore the effective use of 
RiNNAL+ for solving \eqref{SDP-RLT} subproblems within a branch-and-bound framework to solve \eqref{MBQP}.

%%%%%%%%%%%%%%%%%%%%%%%%%%%%%%%%%%%%%%%%%%%%%%%%%%%%%%%%%%%%%%%%%%%%%%%%%%%%%%%
%%%%%%%%%%%%%%%%%%%%%%%%%%%%%%%%%%%%%%%%%%%%%%%%%%%%%%%%%%%%%%%%%%%%%%%%%%%%%%%

%\section*{Acknowledgments}

\bibliographystyle{abbrv}
\bibliography{main}

\bigskip

%%%%%%%%%%%%%%%%%%%%%%%%%%%%%%%%%%%%%%%%%%%%%%%%%%%%%%%%%%%%%%%%%%%%%%%%%%%%%%%
%%%%%%%%%%%%%%%%%%%%%%%%%%%%%%%%%%%%%%%%%%%%%%%%%%%%%%%%%%%%%%%%%%%%%%%%%%%%%%%

\appendix

%%%%%%%%%%%%%%%%%%%%%%%%%%%%%%%%%%%%%%%%%%%%%%%%%%%%%%%%%%%%%%%%%%%%%%%%%%%%%%%

\section{Experiments on BIQ problems}
\label{appendix-BIQ}

\begin{footnotesize}
\begin{longtable}[c]{crrrccrrr}
\caption{Computational results for \eqref{SDP-RLT} relaxation of \eqref{eq-BIQ} problems. } \\
\toprule
Problem&Algorithm&Iteration&Rank& \multicolumn{1}{c}{$\operatorname{R_{max}}$}&Objective&\multicolumn{1}{c}{Time}&\multicolumn{1}{c}{TPG}\\
\midrule
\endfirsthead

\multicolumn{9}{c}%
{{ Table \thetable\ continued from previous page}} \\
\toprule
Problem&Algorithm&Iteration&Rank& \multicolumn{1}{c}{$\operatorname{R_{max}}$}&Objective&\multicolumn{1}{c}{Time}&\multicolumn{1}{c}{TPG}\\
\midrule
\endhead
\midrule
\multicolumn{9}{r}{{Continued on next page}} \\
\midrule
\endfoot

\bottomrule
\endlastfoot

1& RiNNAL+ & 16, 850, 3 & 50 & 8.70e-07 & -1.2259545e+05 & 8.0 &0.9 \\ 
$n=500$ & SDPNAL+ & 44, 97, 1535 & 76& 9.98e-07 & -1.2259531e+05 & 613.1 & - \\ [3pt] 

2& RiNNAL+ & 23, 1200, 5 & 54 & 8.93e-07 & -1.3497627e+05 & 9.2 &1.2 \\ 
$n=500$ & SDPNAL+ & 43, 85, 2028 & 64& 9.99e-07 & -1.3497612e+05 & 678.2 & - \\ [3pt] 

3& RiNNAL+ & 23, 1200, 5 & 56 & 9.12e-07 & -1.3272796e+05 & 8.8 &1.0 \\ 
$n=500$ & SDPNAL+ & 53, 134, 2078 & 68& 9.64e-07 & -1.3272779e+05 & 813.3 & - \\ [3pt] 

4& RiNNAL+ & 16, 850, 3 & 54 & 8.09e-07 & -1.3479364e+05 & 6.0 &0.6 \\ 
$n=500$ & SDPNAL+ & 52, 100, 1941 & 76& 9.77e-07 & -1.3479355e+05 & 718.1 & - \\ [3pt] 

5& RiNNAL+ & 31, 1600, 6 & 54 & 9.50e-07 & -1.3548259e+05 & 11.5 &1.3 \\ 
$n=500$ & SDPNAL+ & 42, 120, 1541 & 60& 9.96e-07 & -1.3548246e+05 & 619.3 & - \\ [3pt] 

6& RiNNAL+ & 23, 1200, 5 & 52 & 7.39e-07 & -1.3029907e+05 & 8.8 &1.0 \\ 
$n=500$ & SDPNAL+ & 52, 105, 2226 & 53& 9.87e-07 & -1.3029900e+05 & 738.8 & - \\ [3pt] 

7& RiNNAL+ & 20, 1050, 4 & 53 & 9.72e-07 & -1.2720408e+05 & 7.8 &0.9 \\ 
$n=500$ & SDPNAL+ & 53, 76, 2066 & 56& 1.03e-06 & -1.2720396e+05 & 706.1 & - \\ [3pt] 

8& RiNNAL+ & 17, 900, 4 & 53 & 9.41e-07 & -1.2793699e+05 & 6.6 &0.8 \\ 
$n=500$ & SDPNAL+ & 43, 84, 1815 & 198& 1.03e-06 & -1.2793688e+05 & 634.0 & - \\ [3pt] 

9& RiNNAL+ & 25, 1300, 5 & 50 & 8.62e-07 & -1.2956776e+05 & 9.2 &1.0 \\ 
$n=500$ & SDPNAL+ & 63, 136, 2977 & 54& 1.01e-06 & -1.2956767e+05 & 1010.1 & - \\ [3pt] 

10& RiNNAL+ & 22, 1150, 5 & 52 & 8.87e-07 & -1.2671564e+05 & 7.9 &0.8 \\ 
$n=500$ & SDPNAL+ & 46, 99, 2010 & 59& 9.99e-07 & -1.2671560e+05 & 670.0 & - \\ [3pt] 

1& RiNNAL+ & 12, 650, 3 & 79 & 9.58e-07 & -3.8983061e+05 & 19.3 &2.3 \\ 
$n=1000$ & SDPNAL+ & 530, 549, 5574 & 995& 3.83e-02$^{\dagger}$ & -4.1472242e+05 & 3600.0 & - \\ [3pt] 

2& RiNNAL+ & 16, 850, 3 & 74 & 7.03e-07 & -3.7136921e+05 & 24.8 &2.5 \\ 
$n=1000$ & SDPNAL+ & 545, 576, 5542 & 989& 1.52e-03$^{\dagger}$ & -3.7152189e+05 & 3600.0 & - \\ [3pt] 

3& RiNNAL+ & 13, 700, 3 & 76 & 7.90e-07 & -3.7607249e+05 & 20.9 &2.4 \\ 
$n=1000$ & SDPNAL+ & 480, 793, 4900 & 45& 2.37e-01$^{\dagger}$ & -1.3200132e+05 & 3600.0 & - \\ [3pt] 

4& RiNNAL+ & 17, 900, 4 & 77 & 9.53e-07 & -3.9016797e+05 & 27.2 &3.4 \\ 
$n=1000$ & SDPNAL+ & 527, 568, 5350 & 100& 2.68e-02$^{\dagger}$ & -3.9000070e+05 & 3600.0 & - \\ [3pt] 

5& RiNNAL+ & 13, 700, 3 & 78 & 6.57e-07 & -3.8989107e+05 & 21.0 &2.2 \\ 
$n=1000$ & SDPNAL+ & 527, 552, 5413 & 1001& 8.87e-02$^{\dagger}$ & -4.5097692e+05 & 3600.0 & - \\ [3pt] 

6& RiNNAL+ & 13, 700, 3 & 75 & 7.21e-07 & -3.7459095e+05 & 20.5 &2.3 \\ 
$n=1000$ & SDPNAL+ & 559, 589, 5800 & 66& 2.07e-02$^{\dagger}$ & -3.7327899e+05 & 3600.0 & - \\ [3pt] 

7& RiNNAL+ & 13, 700, 3 & 78 & 7.09e-07 & -3.7837149e+05 & 20.4 &2.2 \\ 
$n=1000$ & SDPNAL+ & 526, 561, 5500 & 70& 3.72e-02$^{\dagger}$ & -3.7843208e+05 & 3600.0 & - \\ [3pt] 

8& RiNNAL+ & 13, 700, 3 & 78 & 9.18e-07 & -3.9081229e+05 & 20.5 &2.2 \\ 
$n=1000$ & SDPNAL+ & 561, 587, 5800 & 272& 3.39e-01$^{\dagger}$ & -3.5060943e+05 & 3600.0 & - \\ [3pt] 

9& RiNNAL+ & 15, 800, 3 & 74 & 9.20e-07 & -3.7516775e+05 & 24.9 &2.6 \\ 
$n=1000$ & SDPNAL+ & 565, 603, 5674 & 999& 9.61e-03$^{\dagger}$ & -3.7990382e+05 & 3600.0 & - \\ [3pt] 

10& RiNNAL+ & 13, 700, 3 & 75 & 8.75e-07 & -3.7107976e+05 & 20.6 &2.3 \\ 
$n=1000$ & SDPNAL+ & 515, 603, 5302 & 864& 6.72e-03$^{\dagger}$ & -3.6912410e+05 & 3600.0 & - \\ [3pt] 

1& RiNNAL+ & 11, 600, 2 & 134 & 7.28e-07 & -1.6096512e+06 & 133.3 &15.9 \\ 
$n=2500$ & SDPNAL+ & - & - & - & - & - & - \\ [3pt] 

2& RiNNAL+ & 12, 650, 3 & 123 & 8.50e-07 & -1.5828341e+06 & 148.9 &20.7 \\ 
$n=2500$ & SDPNAL+ & - & - & - & - & - & - \\ [3pt] 

3& RiNNAL+ & 10, 550, 2 & 136 & 9.39e-07 & -1.5719952e+06 & 119.6 &12.1 \\ 
$n=2500$ & SDPNAL+ & - & - & - & - & - & - \\ [3pt] 

4& RiNNAL+ & 9, 500, 2 & 136 & 6.64e-07 & -1.5153039e+06 & 112.5 &14.5 \\ 
$n=2500$ & SDPNAL+ & - & - & - & - & - & - \\ [3pt] 

5& RiNNAL+ & 10, 550, 2 & 135 & 7.69e-07 & -1.5994619e+06 & 119.7 &12.2 \\ 
$n=2500$ & SDPNAL+ & - & - & - & - & - & - \\ [3pt] 

6& RiNNAL+ & 14, 750, 3 & 121 & 6.50e-07 & -1.5839251e+06 & 168.8 &20.9 \\ 
$n=2500$ & SDPNAL+ & - & - & - & - & - & - \\ [3pt] 

7& RiNNAL+ & 11, 600, 2 & 133 & 6.58e-07 & -1.5658227e+06 & 128.9 &11.5 \\ 
$n=2500$ & SDPNAL+ & - & - & - & - & - & - \\ [3pt] 

8& RiNNAL+ & 11, 600, 2 & 133 & 6.11e-07 & -1.5770182e+06 & 129.9 &12.1 \\ 
$n=2500$ & SDPNAL+ & - & - & - & - & - & - \\ [3pt] 

9& RiNNAL+ & 10, 550, 2 & 133 & 7.86e-07 & -1.5743417e+06 & 117.7 &13.1 \\ 
$n=2500$ & SDPNAL+ & - & - & - & - & - & - \\ [3pt] 

10& RiNNAL+ & 10, 550, 2 & 133 & 7.35e-07 & -1.5786588e+06 & 118.1 &11.9 \\ 
$n=2500$ & SDPNAL+ & - & - & - & - & - & - \\ [3pt] 

1& RiNNAL+ & 12, 650, 3 & 177 & 8.61e-07 & -4.6838013e+06 & 1103.1 &254.8 \\ 
$n=5000$ & SDPNAL+ & - & - & - & - & - & - \\ [3pt]

\end{longtable}
\end{footnotesize}

%%%%%%%%%%%%%%%%%%%%%%%%%%%%%%%%%%%%%%%%%%%%%%%%%%%%%%%%%%%%%%%%%%%%%%%%%%%%%%%

\section{Experiments on \texorpdfstring{$\theta_+$}{theta} problems}
\label{appendix-theta}

\begin{footnotesize}
\begin{longtable}[c]{crrrccrrr}
\caption{Computational results for \eqref{SDP-RLT} relaxation of \eqref{prob-theta-v1} problems.} \\
\toprule
Problem&Algorithm&Iteration&Rank&\multicolumn{1}{c}{$\operatorname{R_{max}}$} &Objective&\multicolumn{1}{c}{Time}&\multicolumn{1}{c}{TPG}\\
\midrule
\endfirsthead

\multicolumn{9}{c}%
{{ Table \thetable\ continued from previous page}} \\
\toprule
Problem&Algorithm&Iteration&Rank& \multicolumn{1}{c}{$\operatorname{R_{max}}$}&Objective&\multicolumn{1}{c}{Time}&\multicolumn{1}{c}{TPG}\\
\midrule
\endhead
\midrule
\multicolumn{9}{r}{{Continued on next page}} \\
\midrule
\endfoot

\bottomrule
\endlastfoot

G1& RiNNAL+ & 16, 850, 4 & 119 & 5.52e-07 & -1.4424455e+02 & 14.6 &1.6 \\ 
$n=800$ & SDPNAL+ & 64, 135, 1600 & 114& 5.04e-07 & -1.4424448e+02 & 706.0 & - \\ [3pt] 

G2& RiNNAL+ & 13, 700, 3 & 117 & 5.68e-07 & -1.4456459e+02 & 12.3 &1.4 \\ 
$n=800$ & SDPNAL+ & 66, 136, 1600 & 113& 9.27e-07 & -1.4456442e+02 & 714.9 & - \\ [3pt] 

G3& RiNNAL+ & 16, 850, 4 & 117 & 8.33e-07 & -1.4447617e+02 & 14.5 &1.6 \\ 
$n=800$ & SDPNAL+ & 65, 134, 1600 & 114& 8.39e-07 & -1.4447598e+02 & 706.5 & - \\ [3pt] 

G4& RiNNAL+ & 13, 700, 3 & 116 & 2.92e-07 & -1.4457530e+02 & 12.4 &1.4 \\ 
$n=800$ & SDPNAL+ & 65, 134, 1600 & 113& 6.37e-07 & -1.4457522e+02 & 729.9 & - \\ [3pt] 

G5& RiNNAL+ & 15, 800, 3 & 118 & 2.05e-07 & -1.4449452e+02 & 15.1 &1.6 \\ 
$n=800$ & SDPNAL+ & 65, 133, 1600 & 113& 6.69e-07 & -1.4449461e+02 & 738.3 & - \\ [3pt] 

G6& RiNNAL+ & 16, 850, 4 & 119 & 5.52e-07 & -1.4424455e+02 & 15.8 &1.7 \\ 
$n=800$ & SDPNAL+ & 64, 135, 1600 & 114& 5.04e-07 & -1.4424448e+02 & 732.1 & - \\ [3pt] 

G7& RiNNAL+ & 13, 700, 3 & 117 & 5.68e-07 & -1.4456459e+02 & 13.4 &1.6 \\ 
$n=800$ & SDPNAL+ & 66, 136, 1600 & 113& 9.27e-07 & -1.4456442e+02 & 731.8 & - \\ [3pt] 

G8& RiNNAL+ & 16, 850, 4 & 117 & 8.33e-07 & -1.4447617e+02 & 15.1 &1.8 \\ 
$n=800$ & SDPNAL+ & 65, 134, 1600 & 114& 8.39e-07 & -1.4447598e+02 & 743.0 & - \\ [3pt] 

G9& RiNNAL+ & 13, 700, 3 & 116 & 2.92e-07 & -1.4457530e+02 & 16.2 &2.1 \\ 
$n=800$ & SDPNAL+ & 65, 134, 1600 & 113& 6.37e-07 & -1.4457522e+02 & 728.4 & - \\ [3pt] 

G10& RiNNAL+ & 15, 800, 3 & 118 & 2.05e-07 & -1.4449452e+02 & 16.1 &1.6 \\ 
$n=800$ & SDPNAL+ & 65, 133, 1600 & 113& 6.69e-07 & -1.4449461e+02 & 737.0 & - \\ [3pt] 

G11& RiNNAL+ & 56, 2850, 11 & 8 & 6.41e-07 & -4.0000014e+02 & 47.7 &6.4 \\ 
$n=800$ & SDPNAL+ & 482, 722, 9358 & 2& 2.46e-07 & -4.0000005e+02 & 3008.3 & - \\ [3pt] 

G12& RiNNAL+ & 36, 1850, 7 & 8 & 6.68e-07 & -3.9999927e+02 & 30.1 &4.2 \\ 
$n=800$ & SDPNAL+ & 144, 290, 4813 & 2& 2.82e-07 & -4.0000002e+02 & 1664.0 & - \\ [3pt] 

G13& RiNNAL+ & 29, 1500, 6 & 60 & 6.47e-07 & -3.9841750e+02 & 26.8 &3.6 \\ 
$n=800$ & SDPNAL+ & 79, 390, 2326 & 65& 2.48e-07 & -3.9841665e+02 & 1627.7 & - \\ [3pt] 

G14& RiNNAL+ & 34, 1750, 7 & 75 & 9.08e-07 & -2.7900052e+02 & 33.2 &5.1 \\ 
$n=800$ & SDPNAL+ & 118, 558, 3700 & 69& 2.38e-05$^{\dagger}$ & -2.7899998e+02 & 3600.0 & - \\ [3pt] 

G15& RiNNAL+ & 102, 5150, 21 & 136 & 9.81e-07 & -2.8374968e+02 & 102.0 &17.3 \\ 
$n=800$ & SDPNAL+ & 130, 477, 4874 & 199& 7.77e-06$^{\dagger}$ & -2.8374734e+02 & 3600.0 & - \\ [3pt] 

G16& RiNNAL+ & 236, 11850, 48 & 218 & 9.82e-07 & -2.8511991e+02 & 240.5 &31.2 \\ 
$n=800$ & SDPNAL+ & 117, 493, 4700 & 132& 8.46e-06$^{\dagger}$ & -2.8511773e+02 & 3600.0 & - \\ [3pt] 

G17& RiNNAL+ & 158, 7950, 32 & 182 & 9.12e-07 & -2.8612528e+02 & 156.7 &21.1 \\ 
$n=800$ & SDPNAL+ & 131, 448, 5145 & 348& 4.90e-05$^{\dagger}$ & -2.8614605e+02 & 3600.0 & - \\ [3pt] 

G18& RiNNAL+ & 34, 1750, 7 & 75 & 9.08e-07 & -2.7900052e+02 & 34.0 &5.5 \\ 
$n=800$ & SDPNAL+ & 118, 558, 3700 & 69& 2.38e-05$^{\dagger}$ & -2.7899998e+02 & 3600.0 & - \\ [3pt] 

G19& RiNNAL+ & 102, 5150, 21 & 136 & 9.81e-07 & -2.8374968e+02 & 103.0 &17.6 \\ 
$n=800$ & SDPNAL+ & 130, 477, 5259 & 220& 6.51e-06$^{\dagger}$ & -2.8375151e+02 & 3600.0 & - \\ [3pt] 

G20& RiNNAL+ & 236, 11850, 48 & 218 & 9.82e-07 & -2.8511991e+02 & 235.6 &30.5 \\ 
$n=800$ & SDPNAL+ & 123, 512, 4700 & 125& 7.93e-06$^{\dagger}$ & -2.8511787e+02 & 3600.0 & - \\ [3pt] 

G21& RiNNAL+ & 158, 7950, 32 & 182 & 9.12e-07 & -2.8612528e+02 & 157.7 &21.5 \\ 
$n=800$ & SDPNAL+ & 132, 457, 5500 & 76& 2.54e-05$^{\dagger}$ & -2.8611971e+02 & 3600.0 & - \\ [3pt] 

G22& RiNNAL+ & 19, 1000, 4 & 96 & 9.69e-07 & -5.7739734e+02 & 132.4 &14.8 \\ 
$n=2000$ & SDPNAL+ & - & - & - & - & - \\ [3pt] 

G23& RiNNAL+ & 17, 900, 4 & 93 & 4.33e-07 & -5.7654954e+02 & 119.8 &12.4 \\ 
$n=2000$ & SDPNAL+ & - & - & - & - & - \\ [3pt] 

G24& RiNNAL+ & 16, 850, 3 & 98 & 7.28e-07 & -5.7891045e+02 & 113.2 &10.3 \\ 
$n=2000$ & SDPNAL+ & - & - & - & - & - \\ [3pt] 

G25& RiNNAL+ & 13, 700, 3 & 98 & 9.00e-07 & -5.7703872e+02 & 97.0 &11.0 \\ 
$n=2000$ & SDPNAL+ & - & - & - & - & - \\ [3pt] 

G26& RiNNAL+ & 19, 1000, 4 & 97 & 1.95e-07 & -5.7691694e+02 & 121.9 &12.5 \\ 
$n=2000$ & SDPNAL+ & - & - & - & - & - \\ [3pt] 

G27& RiNNAL+ & 19, 1000, 4 & 96 & 9.69e-07 & -5.7739734e+02 & 122.1 &13.5 \\ 
$n=2000$ & SDPNAL+ & - & - & - & - & - \\ [3pt] 

G28& RiNNAL+ & 16, 850, 3 & 101 & 9.96e-07 & -5.7683605e+02 & 101.2 &8.7 \\ 
$n=2000$ & SDPNAL+ & - & - & - & - & - \\ [3pt] 

G29& RiNNAL+ & 16, 850, 3 & 98 & 7.28e-07 & -5.7891045e+02 & 103.1 &9.3 \\ 
$n=2000$ & SDPNAL+ & - & - & - & - & - \\ [3pt] 

G30& RiNNAL+ & 13, 700, 3 & 98 & 9.00e-07 & -5.7703872e+02 & 87.7 &10.4 \\ 
$n=2000$ & SDPNAL+ & - & - & - & - & - \\ [3pt] 

G31& RiNNAL+ & 19, 1000, 4 & 97 & 1.95e-07 & -5.7691694e+02 & 118.6 &11.7 \\ 
$n=2000$ & SDPNAL+ & - & - & - & - & - \\ [3pt] 

G32& RiNNAL+ & 41, 2100, 8 & 8 & 2.74e-07 & -9.9999691e+02 & 241.3 &31.5 \\ 
$n=2000$ & SDPNAL+ & - & - & - & - & - \\ [3pt] 

G33& RiNNAL+ & 39, 2000, 8 & 124 & 9.73e-07 & -9.9604039e+02 & 239.8 &33.2 \\ 
$n=2000$ & SDPNAL+ & - & - & - & - & - \\ [3pt] 

G34& RiNNAL+ & 53, 2700, 11 & 8 & 8.83e-07 & -9.9999933e+02 & 322.7 &45.2 \\ 
$n=2000$ & SDPNAL+ & - & - & - & - & - \\ [3pt] 

G35& RiNNAL+ & 91, 4600, 19 & 359 & 9.97e-07 & -7.1824082e+02 & 652.3 &112.4 \\ 
$n=2000$ & SDPNAL+ & - & - & - & - & - \\ [3pt] 

G36& RiNNAL+ & 111, 5600, 22 & 383 & 9.80e-07 & -6.9600402e+02 & 819.5 &141.1 \\ 
$n=2000$ & SDPNAL+ & - & - & - & - & - \\ [3pt] 

G37& RiNNAL+ & 55, 2800, 11 & 294 & 9.61e-07 & -7.0800224e+02 & 404.7 &75.1 \\ 
$n=2000$ & SDPNAL+ & - & - & - & - & - \\ [3pt] 

G38& RiNNAL+ & 58, 2950, 12 & 326 & 1.00e-06 & -7.1600307e+02 & 422.5 &74.2 \\ 
$n=2000$ & SDPNAL+ & - & - & - & - & - \\ [3pt] 

G39& RiNNAL+ & 91, 4600, 19 & 359 & 9.97e-07 & -7.1824082e+02 & 653.3 &112.8 \\ 
$n=2000$ & SDPNAL+ & - & - & - & - & - \\ [3pt] 

G40& RiNNAL+ & 111, 5600, 22 & 383 & 9.80e-07 & -6.9600402e+02 & 821.8 &141.7 \\ 
$n=2000$ & SDPNAL+ & - & - & - & - & - \\ [3pt] 

G41& RiNNAL+ & 55, 2800, 11 & 294 & 9.61e-07 & -7.0800224e+02 & 405.1 &75.4 \\ 
$n=2000$ & SDPNAL+ & - & - & - & - & - \\ [3pt] 

G42& RiNNAL+ & 58, 2950, 12 & 326 & 1.00e-06 & -7.1600307e+02 & 421.1 &73.9 \\ 
$n=2000$ & SDPNAL+ & - & - & - & - & - \\ [3pt] 

G43& RiNNAL+ & 11, 600, 2 & 94 & 8.64e-07 & -2.7973479e+02 & 15.5 &1.3 \\ 
$n=1000$ & SDPNAL+ & 75, 206, 1450 & 73& 9.01e-07 & -2.7973560e+02 & 1370.3 & - \\ [3pt] 

G44& RiNNAL+ & 16, 850, 3 & 76 & 7.12e-07 & -2.7974525e+02 & 22.2 &2.0 \\ 
$n=1000$ & SDPNAL+ & 95, 196, 1764 & 73& 1.98e-06 & -2.7974536e+02 & 1417.6 & - \\ [3pt] 

G45& RiNNAL+ & 16, 850, 4 & 76 & 8.87e-07 & -2.7931758e+02 & 22.5 &2.9 \\ 
$n=1000$ & SDPNAL+ & 94, 212, 1754 & 74& 1.97e-06 & -2.7931692e+02 & 1496.1 & - \\ [3pt] 

G46& RiNNAL+ & 15, 800, 3 & 76 & 5.74e-07 & -2.7903227e+02 & 20.4 &1.8 \\ 
$n=1000$ & SDPNAL+ & 77, 191, 1450 & 72& 6.59e-07 & -2.7903199e+02 & 1268.6 & - \\ [3pt] 

G47& RiNNAL+ & 14, 750, 3 & 76 & 9.90e-07 & -2.8089199e+02 & 19.7 &2.0 \\ 
$n=1000$ & SDPNAL+ & 76, 200, 1450 & 72& 7.86e-07 & -2.8089057e+02 & 1317.1 & - \\ [3pt] 

G48& RiNNAL+ & 40, 2050, 8 & 8 & 6.34e-07 & -1.4999937e+03 & 1082.7 &179.3 \\ 
$n=3000$ & SDPNAL+ & - & - & - & - & - \\ [3pt] 

G49& RiNNAL+ & 63, 3200, 13 & 14 & 7.66e-07 & -1.4999553e+03 & 1750.2 &313.0 \\ 
$n=3000$ & SDPNAL+ & - & - & - & - & - \\ [3pt] 

G50& RiNNAL+ & 24, 1250, 5 & 134 & 9.39e-07 & -1.4940617e+03 & 754.8 &149.3 \\ 
$n=3000$ & SDPNAL+ & - & - & - & - & - \\ [3pt] 

G51& RiNNAL+ & 66, 3350, 14 & 224 & 9.41e-07 & -3.4900100e+02 & 100.8 &16.3 \\ 
$n=1000$ & SDPNAL+ & 84, 411, 1900 & 136& 6.79e-05$^{\dagger}$ & -3.4899936e+02 & 3600.0 & - \\ [3pt] 

G52& RiNNAL+ & 119, 6000, 24 & 186 & 9.51e-07 & -3.4838718e+02 & 182.1 &29.7 \\ 
$n=1000$ & SDPNAL+ & 93, 383, 2500 & 111& 8.46e-05$^{\dagger}$ & -3.4837547e+02 & 3600.0 & - \\ [3pt] 

G53& RiNNAL+ & 180, 9050, 36 & 224 & 9.83e-07 & -3.4821414e+02 & 269.0 &37.1 \\ 
$n=1000$ & SDPNAL+ & 102, 356, 2892 & 798& 7.16e-04$^{\dagger}$ & -3.4871080e+02 & 3600.0 & - \\ [3pt] 

G54& RiNNAL+ & 46, 2350, 10 & 168 & 9.01e-07 & -3.4100141e+02 & 68.4 &11.3 \\ 
$n=1000$ & SDPNAL+ & 100, 372, 2500 & 110& 1.93e-04$^{\dagger}$ & -3.4090921e+02 & 3600.0 & - \\ [3pt] 

1dc.1024& RiNNAL+ & 271, 13600, 217 & 329 & 8.41e-07 & -9.5551166e+01 & 698.3 &284.4 \\ 
$n=1024$ & SDPNAL+ & 87, 203, 2466 & 315& 9.87e-07 & -9.5551123e+01 & 2930.4 & - \\ [3pt] 

1dc.2048& RiNNAL+ & 148, 7450, 132 & 778 & 9.61e-07 & -1.7404329e+02 & 1945.3 &748.6 \\ 
$n=2048$ & SDPNAL+ & - & - & - & - & - \\ [3pt] 

1et.1024& RiNNAL+ & 51, 2600, 11 & 294 & 8.84e-07 & -1.8207222e+02 & 84.8 &10.9 \\ 
$n=1024$ & SDPNAL+ & 89, 214, 2219 & 290& 9.98e-07 & -1.8207159e+02 & 2113.9 & - \\ [3pt] 

1et.2048& RiNNAL+ & 101, 5100, 24 & 571 & 9.88e-07 & -3.3816662e+02 & 835.3 &119.1 \\ 
$n=2048$ & SDPNAL+ & - & - & - & - & - \\ [3pt] 

1tc.1024& RiNNAL+ & 116, 5850, 24 & 288 & 9.38e-07 & -2.0420520e+02 & 192.4 &28.1 \\ 
$n=1024$ & SDPNAL+ & 100, 300, 3430 & 339& 9.43e-06$^{\dagger}$ & -2.0419958e+02 & 3600.0 & - \\ [3pt] 

1tc.2048& RiNNAL+ & 94, 4750, 26 & 520 & 9.85e-07 & -3.7049061e+02 & 820.7 &162.8 \\ 
$n=2048$ & SDPNAL+ & - & - & - & - & - \\ [3pt] 

1zc.1024& RiNNAL+ & 26, 1350, 6 & 316 & 4.63e-07 & -1.2800011e+02 & 43.3 &6.0 \\ 
$n=1024$ & SDPNAL+ & 60, 140, 1010 & 653& 8.38e-07 & -1.2800009e+02 & 934.2 & - \\ [3pt] 

1zc.2048& RiNNAL+ & 36, 1850, 8 & 521 & 5.43e-07 & -2.3740022e+02 & 285.9 &39.0 \\ 
$n=2048$ & SDPNAL+ & - & - & - & - & - \\ [3pt] 

2dc.1024& RiNNAL+ & 111, 5600, 39 & 769 & 9.40e-07 & -1.7710237e+01 & 260.0 &38.0 \\ 
$n=1024$ & SDPNAL+ & 103, 271, 3434 & 774& 1.18e-05$^{\dagger}$ & -1.7708084e+01 & 3600.0 & - \\ [3pt]

\end{longtable}
\end{footnotesize}

%%%%%%%%%%%%%%%%%%%%%%%%%%%%%%%%%%%%%%%%%%%%%%%%%%%%%%%%%%%%%%%%%%%%%%%%%%%%%%%

\section{Experiments on ccMSSC problems}
\label{appendix-ccMSSC}

\begin{footnotesize}
\begin{longtable}[c]{lrrrccrrr}
\caption{Computational results for \eqref{SDP-RLT} relaxation of \eqref{prob-ccMSCC-v1} problem.} \\
\toprule
Prob ($k,n$)&Algorithm&Iteration&Rank& \multicolumn{1}{c}{$\operatorname{R_{max}}$}&Objective&\multicolumn{1}{c}{Time}&\multicolumn{1}{c}{TPG}\\
\midrule
\endfirsthead

\multicolumn{9}{c}%
{{ Table \thetable\ continued from previous page}} \\
\toprule
Prob ($k,n$)&Algorithm&Iteration&Rank& \multicolumn{1}{c}{$\operatorname{R_{max}}$}&Objective&\multicolumn{1}{c}{Time}&\multicolumn{1}{c}{TPG}\\
\midrule
\endhead
\midrule
\multicolumn{9}{r}{{Continued on next page}} \\
\midrule
\endfoot

\bottomrule
\endlastfoot

1& RiNNAL+ & 411, 20600, 83 & 23 & 9.85e-07 & 4.0610863e+05 & 248.0 &25.7 \\ 
2, 516& SDPNAL+ & 627, 1080, 18513 & 496& 6.38e-04$^{\dagger}$ & 4.0608249e+05 & 3600.0 & - \\ [3pt] 

2& RiNNAL+ & 421, 21100, 85 & 14 & 9.82e-07 & 1.2476265e+05 & 273.1 &27.4 \\ 
2, 536& SDPNAL+ & 309, 1097, 14019 & 12& 6.81e-05$^{\dagger}$ & 1.2476325e+05 & 3600.0 & - \\ [3pt] 

3& RiNNAL+ & 25, 1300, 5 & 11 & 9.93e-07 & 2.5343203e+03 & 29.2 &4.3 \\ 
3, 540& SDPNAL+ & 60, 69, 3108 & 20& 9.90e-07 & 2.5343241e+03 & 484.1 & - \\ [3pt] 

4& RiNNAL+ & 51, 2600, 11 & 15 & 8.73e-07 & 3.4886741e+03 & 54.6 &9.3 \\ 
3, 540& SDPNAL+ & 104, 143, 3301 & 29& 8.35e-07 & 3.4886788e+03 & 537.1 & - \\ [3pt] 

5& RiNNAL+ & 9, 500, 2 & 5 & 9.27e-07 & 7.3254609e+04 & 14.3 &1.4 \\ 
2, 720& SDPNAL+ & 201, 234, 10487 & 17& 8.09e-06$^{\dagger}$ & 7.3254549e+04 & 3600.0 & - \\ [3pt] 

6& RiNNAL+ & 6, 350, 1 & 5 & 7.54e-07 & 2.3820331e+08 & 41.8 &0.5 \\ 
2, 726& SDPNAL+ & 33, 177, 1360 & 1& 4.16e-08 & 2.3820380e+08 & 770.2 & - \\ [3pt] 

7& RiNNAL+ & 7, 382, 2 & 7 & 3.02e-08 & 2.9237520e+03 & 85.5 &41.2 \\ 
3, 798& SDPNAL+ & 23, 40, 850 & 1& 3.47e-07 & 2.9237725e+03 & 441.3 & - \\ [3pt] 

8& RiNNAL+ & 21, 1100, 4 & 16 & 7.31e-07 & 1.6335116e+04 & 45.2 &11.8 \\ 
4, 800& SDPNAL+ & 40, 44, 1501 & 18& 7.37e-07 & 1.6335114e+04 & 505.7 & - \\ [3pt] 

9& RiNNAL+ & 6, 331, 2 & 37 & 5.97e-07 & 4.2206142e+09 & 53.9 &3.4 \\ 
2, 902& SDPNAL+ & 56, 946, 1500 & 38& 4.54e-02$^{\dagger}$ & 4.2166306e+09 & 3600.0 & - \\ [3pt] 

10& RiNNAL+ & 3, 200, 1 & 7 & 8.78e-07 & 4.0536895e+09 & 12.6 &1.1 \\ 
2, 902& SDPNAL+ & 54, 246, 2312 & 1& 7.91e-07 & 4.0536886e+09 & 2313.4 & - \\ [3pt] 

11& RiNNAL+ & 3, 200, 1 & 16 & 8.50e-07 & 4.0728152e+09 & 14.3 &2.8 \\ 
2, 902& SDPNAL+ & 67, 299, 3092 & 2& 4.49e-07 & 4.0728127e+09 & 3022.6 & - \\ [3pt] 

12& RiNNAL+ & 31, 1600, 7 & 60 & 8.07e-07 & 4.6667138e+05 & 109.0 &5.4 \\ 
2, 922& SDPNAL+ & 93, 398, 3530 & 62& 4.70e-06$^{\dagger}$ & 4.6667139e+05 & 3600.0 & - \\ [3pt] 

13& RiNNAL+ & 8, 402, 2 & 59 & 4.13e-08 & 3.1659837e+03 & 606.9 &68.4 \\ 
3, 933& SDPNAL+ & 31, 156, 700 & 1& 9.29e-08 & 3.1660319e+03 & 1284.7 & - \\ [3pt] 

14& RiNNAL+ & 15, 800, 3 & 6 & 7.48e-07 & 6.1181510e+05 & 35.8 &3.3 \\ 
2, 1000& SDPNAL+ & 138, 169, 4990 & 125& 1.01e-05$^{\dagger}$ & 6.1181479e+05 & 3600.0 & - \\ [3pt] 

15& RiNNAL+ & 12, 511, 3 & 21 & 3.90e-07 & 8.5151301e+02 & 1500.3 &193.9 \\ 
3, 1662& SDPNAL+ & - & - & - & - & - & - \\ [3pt] 

16& RiNNAL+ & 10, 550, 2 & 7 & 6.15e-07 & 6.1040841e+03 & 86.8 &12.5 \\ 
2, 1752& SDPNAL+ & - & - & - & - & - & - \\ [3pt] 

17& RiNNAL+ & 5, 300, 1 & 5 & 8.09e-07 & 7.0548748e+04 & 45.5 &4.4 \\ 
2, 1768& SDPNAL+ & - & - & - & - & - & - \\ [3pt] 

18& RiNNAL+ & 5, 300, 1 & 6 & 8.05e-07 & 2.3371721e+03 & 47.6 &4.9 \\ 
2, 1782& SDPNAL+ & - & - & - & - & - & - \\ [3pt] 

19& RiNNAL+ & 15, 800, 3 & 7 & 6.90e-07 & 1.4107262e+03 & 131.8 &22.8 \\ 
2, 1782& SDPNAL+ & - & - & - & - & - & - \\ [3pt] 

20& RiNNAL+ & 2, 150, 1 & 28 & 5.87e-07 & 7.3610392e+08 & 42.2 &13.7 \\ 
2, 1800& SDPNAL+ & - & - & - & - & - & - \\ [3pt] 

21& RiNNAL+ & 18, 905, 8 & 11 & 1.96e-07 & 5.9631573e+02 & 1931.4 &1330.1 \\ 
3, 1815& SDPNAL+ & - & - & - & - & - & - \\ [3pt] 

22& RiNNAL+ & 15, 800, 3 & 11 & 5.17e-07 & 6.0024695e+04 & 156.7 &23.0 \\ 
2, 1960& SDPNAL+ & - & - & - & - & - & - \\ [3pt] 

23& RiNNAL+ & 9, 500, 2 & 7 & 3.26e-07 & 9.1772840e+03 & 117.6 &19.0 \\ 
2, 1966& SDPNAL+ & - & - & - & - & - & - \\ [3pt] 

24& RiNNAL+ & 86, 4350, 18 & 23 & 9.36e-07 & 9.9658646e+05 & 2375.3 &339.4 \\ 
3, 2250& SDPNAL+ & - & - & - & - & - & - \\ [3pt] 

25& RiNNAL+ & 51, 2600, 11 & 61 & 5.42e-07 & 1.0500893e+06 & 2215.9 &170.8 \\ 
3, 2250& SDPNAL+ & - & - & - & - & - & - \\ [3pt] 

26& RiNNAL+ & 21, 1100, 4 & 45 & 8.68e-07 & 1.0568137e+06 & 951.1 &96.5 \\ 
3, 2250& SDPNAL+ & - & - & - & - & - & - \\ [3pt] 

27& RiNNAL+ & 56, 2850, 11 & 6 & 9.77e-07 & 1.3311950e+04 & 1141.7 &204.7 \\ 
2, 2324& SDPNAL+ & - & - & - & - & - & - \\ [3pt] 

28& RiNNAL+ & 10, 550, 2 & 14 & 9.04e-07 & 3.6482521e+05 & 346.4 &67.3 \\ 
2, 2740& SDPNAL+ & - & - & - & - & - & - \\ [3pt] 

\end{longtable}
\end{footnotesize}

%%%%%%%%%%%%%%%%%%%%%%%%%%%%%%%%%%%%%%%%%%%%%%%%%%%%%%%%%%%%%%%%%%%%%%%%%%%%%%%

\section{Experiments on QMSTP problems}
\label{appendix-QMSTP}

\begin{footnotesize}
\begin{longtable}[c]{lrrrccrrr}
\caption{Computational results for \eqref{SDP-RLT} relaxation of \eqref{eq-QMSTP-relax} problems.} \\
\toprule
Problem&Algorithm&Iteration&Rank& \multicolumn{1}{c}{$\operatorname{R_{max}}$}&Objective&\multicolumn{1}{c}{Time}&\multicolumn{1}{c}{TPG}\\
\midrule
\endfirsthead

\multicolumn{9}{c}%
{{ Table \thetable\ continued from previous page}} \\
\toprule
Problem&Algorithm&Iteration&Rank& \multicolumn{1}{c}{$\operatorname{R_{max}}$}&Objective&\multicolumn{1}{c}{Time}&\multicolumn{1}{c}{TPG}\\
\midrule
\endhead
\midrule
\multicolumn{9}{r}{{Continued on next page}} \\
\midrule
\endfoot

\bottomrule
\endlastfoot

vsym-1& RiNNAL+ & 14, 750, 3 & 4 & 5.84e-07 & 8.2077003e+04 & 7.5 &1.1 \\ 
$n=435$ & SDPNAL+ & 115, 175, 2653 & 1& 5.15e-07 & 8.2077033e+04 & 221.4 & - \\ [3pt] 

vsym-2& RiNNAL+ & 11, 600, 2 & 6 & 8.80e-07 & 7.4235009e+04 & 5.8 &0.6 \\ 
$n=435$ & SDPNAL+ & 144, 232, 3102 & 1& 1.64e-07 & 7.4234956e+04 & 272.2 & - \\ [3pt] 

vsym-3& RiNNAL+ & 16, 850, 4 & 4 & 1.33e-08 & 7.4105014e+04 & 11.0 &1.4 \\ 
$n=435$ & SDPNAL+ & 96, 145, 2200 & 1& 2.88e-07 & 7.4104978e+04 & 169.3 & - \\ [3pt] 

vsym-4& RiNNAL+ & 142, 7150, 29 & 10 & 9.28e-07 & 8.7476205e+04 & 54.6 &7.1 \\ 
$n=435$ & SDPNAL+ & 382, 616, 10201 & 1& 2.22e-07 & 8.7475967e+04 & 808.1 & - \\ [3pt] 

vsym-5& RiNNAL+ & 26, 1350, 6 & 4 & 9.95e-07 & 8.6586018e+04 & 11.4 &1.4 \\ 
$n=435$ & SDPNAL+ & 277, 414, 6526 & 1& 9.60e-08 & 8.6585979e+04 & 534.1 & - \\ [3pt] 

vsym-6& RiNNAL+ & 11, 600, 3 & 5 & 4.13e-07 & 6.9976003e+04 & 5.8 &1.1 \\ 
$n=435$ & SDPNAL+ & 80, 108, 2051 & 1& 4.45e-07 & 6.9975772e+04 & 159.4 & - \\ [3pt] 

vsym-7& RiNNAL+ & 14, 750, 3 & 4 & 4.99e-07 & 7.8864006e+04 & 7.3 &0.9 \\ 
$n=435$ & SDPNAL+ & 106, 203, 2500 & 1& 3.01e-07 & 7.8861495e+04 & 228.7 & - \\ [3pt] 

vsym-8& RiNNAL+ & 26, 1350, 5 & 5 & 7.81e-07 & 7.3015027e+04 & 11.4 &1.4 \\ 
$n=435$ & SDPNAL+ & 82, 124, 1900 & 1& 9.18e-08 & 7.3015116e+04 & 158.9 & - \\ [3pt] 

vsym-9& RiNNAL+ & 15, 800, 3 & 4 & 1.82e-07 & 7.2562004e+04 & 7.1 &0.8 \\ 
$n=435$ & SDPNAL+ & 130, 196, 3550 & 1& 9.41e-07 & 7.2562144e+04 & 270.6 & - \\ [3pt] 

vsym-10& RiNNAL+ & 11, 600, 3 & 6 & 1.08e-08 & 7.9153020e+04 & 7.6 &1.3 \\ 
$n=435$ & SDPNAL+ & 63, 90, 2450 & 1& 1.26e-07 & 7.9153501e+04 & 173.8 & - \\ [3pt] 

sym-1& RiNNAL+ & 105, 5300, 21 & 196 & 8.70e-07 & 5.4457695e+03 & 104.5 &3.0 \\ 
$n=435$ & SDPNAL+ & 83, 85, 1377 & 194& 1.80e-06 & 5.4457692e+03 & 120.0 & - \\ [3pt] 

sym-2& RiNNAL+ & 121, 6100, 25 & 196 & 6.62e-07 & 5.3601951e+03 & 116.9 &3.2 \\ 
$n=435$ & SDPNAL+ & 200, 267, 2927 & 195& 9.64e-07 & 5.3601944e+03 & 286.7 & - \\ [3pt] 

sym-3& RiNNAL+ & 111, 5600, 23 & 201 & 8.62e-07 & 5.2587073e+03 & 115.5 &3.1 \\ 
$n=435$ & SDPNAL+ & 66, 70, 1128 & 199& 9.93e-07 & 5.2587071e+03 & 93.8 & - \\ [3pt] 

sym-4& RiNNAL+ & 106, 5350, 22 & 194 & 4.39e-07 & 5.3695585e+03 & 102.5 &2.9 \\ 
$n=435$ & SDPNAL+ & 67, 67, 1173 & 192& 7.67e-07 & 5.3695582e+03 & 93.4 & - \\ [3pt] 

sym-5& RiNNAL+ & 80, 4050, 16 & 196 & 8.19e-07 & 5.2813223e+03 & 70.7 &2.3 \\ 
$n=435$ & SDPNAL+ & 65, 68, 1112 & 194& 9.96e-07 & 5.2813222e+03 & 86.9 & - \\ [3pt] 

sym-6& RiNNAL+ & 104, 5250, 21 & 196 & 9.19e-07 & 5.2952322e+03 & 101.4 &2.9 \\ 
$n=435$ & SDPNAL+ & 152, 184, 2001 & 194& 9.83e-07 & 5.2952307e+03 & 180.8 & - \\ [3pt] 

sym-7& RiNNAL+ & 121, 6100, 24 & 198 & 9.31e-07 & 5.3275724e+03 & 128.1 &3.2 \\ 
$n=435$ & SDPNAL+ & 83, 85, 1271 & 197& 9.96e-07 & 5.3275724e+03 & 117.1 & - \\ [3pt] 

sym-8& RiNNAL+ & 93, 4700, 19 & 196 & 6.08e-07 & 5.3277159e+03 & 84.6 &2.6 \\ 
$n=435$ & SDPNAL+ & 68, 75, 1149 & 194& 9.99e-07 & 5.3277160e+03 & 94.7 & - \\ [3pt] 

sym-9& RiNNAL+ & 91, 4600, 18 & 197 & 7.97e-07 & 5.3229653e+03 & 81.2 &2.5 \\ 
$n=435$ & SDPNAL+ & 101, 105, 1443 & 197& 1.42e-06 & 5.3229654e+03 & 124.6 & - \\ [3pt] 

sym-10& RiNNAL+ & 96, 4850, 20 & 199 & 9.37e-07 & 5.2736803e+03 & 87.3 &2.7 \\ 
$n=435$ & SDPNAL+ & 78, 84, 1276 & 197& 9.83e-07 & 5.2736802e+03 & 107.3 & - \\ [3pt] 

esym-1& RiNNAL+ & 5511, 275600, 1102 & 40 & 9.02e-05$^{\dagger}$ & 6.3770159e+03 & 3600.0 &205.0 \\ 
$n=435$ & SDPNAL+ & 647, 891, 13000 & 38& 1.00e-06 & 6.3770874e+03 & 1386.6 & - \\ [3pt] 

esym-2& RiNNAL+ & 6414, 320750, 1283 & 33 & 4.81e-05$^{\dagger}$ & 6.9466174e+03 & 3600.0 &252.0 \\ 
$n=435$ & SDPNAL+ & 548, 773, 10900 & 40& 9.36e-07 & 6.9466717e+03 & 1225.0 & - \\ [3pt] 

esym-3& RiNNAL+ & 5086, 254350, 1017 & 39 & 7.47e-05$^{\dagger}$ & 7.9610653e+03 & 3600.0 &219.6 \\ 
$n=435$ & SDPNAL+ & 385, 493, 8635 & 51& 1.00e-06 & 7.9611728e+03 & 938.2 & - \\ [3pt] 

esym-4& RiNNAL+ & 6375, 318800, 1275 & 25 & 1.87e-04$^{\dagger}$ & 7.5283574e+03 & 3600.0 &268.7 \\ 
$n=435$ & SDPNAL+ & 461, 629, 9400 & 39& 9.88e-07 & 7.5284582e+03 & 998.8 & - \\ [3pt] 

esym-5& RiNNAL+ & 6031, 301600, 1206 & 41 & 4.66e-05$^{\dagger}$ & 7.6181667e+03 & 3600.0 &213.7 \\ 
$n=435$ & SDPNAL+ & 451, 630, 9100 & 48& 9.92e-07 & 7.6183146e+03 & 997.8 & - \\ [3pt] 

esym-6& RiNNAL+ & 5264, 263250, 1053 & 47 & 1.93e-05$^{\dagger}$ & 7.4736776e+03 & 3600.0 &192.0 \\ 
$n=435$ & SDPNAL+ & 331, 497, 7518 & 46& 1.00e-06 & 7.4737703e+03 & 824.0 & - \\ [3pt] 

esym-7& RiNNAL+ & 5091, 254600, 1018 & 59 & 7.48e-05$^{\dagger}$ & 8.0184786e+03 & 3600.0 &183.1 \\ 
$n=435$ & SDPNAL+ & 801, 858, 11051 & 68& 1.00e-06 & 8.0187444e+03 & 1127.6 & - \\ [3pt] 

esym-8& RiNNAL+ & 4935, 246800, 987 & 30 & 1.78e-05$^{\dagger}$ & 6.8570588e+03 & 3600.0 &172.2 \\ 
$n=435$ & SDPNAL+ & 616, 833, 12701 & 29& 9.05e-07 & 6.8570981e+03 & 1321.7 & - \\ [3pt] 

esym-9& RiNNAL+ & 5627, 281400, 1126 & 26 & 1.82e-04$^{\dagger}$ & 7.5313705e+03 & 3600.0 &235.7 \\ 
$n=435$ & SDPNAL+ & 675, 844, 12858 & 45& 9.99e-07 & 7.5315178e+03 & 1381.6 & - \\ [3pt] 

esym-10& RiNNAL+ & 4853, 242700, 971 & 64 & 1.18e-05$^{\dagger}$ & 7.9609614e+03 & 3600.0 &155.5 \\ 
$n=435$ & SDPNAL+ & 372, 476, 5950 & 60& 9.71e-07 & 7.9610475e+03 & 631.6 & - \\ [3pt] 

vsym-1& RiNNAL+ & 10, 550, 2 & 9 & 9.66e-07 & 1.7014306e+05 & 38.6 &5.4 \\ 
$n=1225$ & SDPNAL+ & 106, 220, 4200 & 2& 8.58e-06$^{\dagger}$ & 1.7018175e+05 & 3600.0 & - \\ [3pt] 

vsym-2& RiNNAL+ & 31, 1600, 6 & 7 & 8.98e-07 & 1.5515539e+05 & 93.9 &16.4 \\ 
$n=1225$ & SDPNAL+ & 108, 191, 3902 & 1183& 4.45e-04$^{\dagger}$ & 1.5638172e+05 & 3600.0 & - \\ [3pt] 

vsym-3& RiNNAL+ & 36, 1850, 12 & 7 & 8.40e-07 & 1.6989411e+05 & 121.9 &33.0 \\ 
$n=1225$ & SDPNAL+ & 181, 254, 3767 & 427& 2.47e-04$^{\dagger}$ & 1.7208258e+05 & 3600.0 & - \\ [3pt] 

vsym-4& RiNNAL+ & 77, 3900, 20 & 10 & 8.61e-07 & 1.6000112e+05 & 242.9 &62.3 \\ 
$n=1225$ & SDPNAL+ & 109, 183, 4278 & 1190& 5.88e-04$^{\dagger}$ & 1.6147984e+05 & 3600.0 & - \\ [3pt] 

vsym-5& RiNNAL+ & 426, 21350, 91 & 14 & 9.77e-07 & 1.5072543e+05 & 1227.3 &260.6 \\ 
$n=1225$ & SDPNAL+ & 93, 167, 4196 & 1189& 1.66e-03$^{\dagger}$ & 1.5231965e+05 & 3600.0 & - \\ [3pt] 

vsym-6& RiNNAL+ & 244, 12250, 62 & 6 & 9.84e-07 & 1.7481705e+05 & 747.3 &187.4 \\ 
$n=1225$ & SDPNAL+ & 108, 196, 4062 & 375& 1.73e-04$^{\dagger}$ & 1.7651275e+05 & 3600.0 & - \\ [3pt] 

vsym-7& RiNNAL+ & 16, 850, 4 & 5 & 8.15e-07 & 1.5365828e+05 & 55.8 &11.7 \\ 
$n=1225$ & SDPNAL+ & 119, 222, 4350 & 2& 8.18e-06$^{\dagger}$ & 1.5365163e+05 & 3600.0 & - \\ [3pt] 

vsym-8& RiNNAL+ & 78, 3950, 16 & 7 & 6.98e-07 & 1.7840805e+05 & 224.7 &41.8 \\ 
$n=1225$ & SDPNAL+ & 108, 184, 4116 & 1182& 2.55e-04$^{\dagger}$ & 1.7990365e+05 & 3600.0 & - \\ [3pt] 

vsym-9& RiNNAL+ & 26, 1350, 5 & 8 & 7.78e-07 & 1.5342806e+05 & 81.9 &14.5 \\ 
$n=1225$ & SDPNAL+ & 108, 170, 4104 & 466& 6.06e-04$^{\dagger}$ & 1.5519319e+05 & 3600.0 & - \\ [3pt] 

vsym-10& RiNNAL+ & 16, 850, 3 & 6 & 2.56e-07 & 1.7824801e+05 & 57.4 &9.7 \\ 
$n=1225$ & SDPNAL+ & 112, 192, 4450 & 4& 2.48e-05$^{\dagger}$ & 1.7843544e+05 & 3600.0 & - \\ [3pt] 

sym-1& RiNNAL+ & 91, 4600, 18 & 511 & 7.22e-07 & 1.4089986e+04 & 566.1 &21.5 \\ 
$n=1225$ & SDPNAL+ & 205, 218, 2653 & 509& 1.08e-06 & 1.4089986e+04 & 1968.0 & - \\ [3pt] 

sym-2& RiNNAL+ & 75, 3800, 15 & 508 & 9.86e-07 & 1.4080256e+04 & 387.3 &18.8 \\ 
$n=1225$ & SDPNAL+ & 168, 168, 2505 & 506& 1.45e-06 & 1.4080256e+04 & 1757.0 & - \\ [3pt] 

sym-3& RiNNAL+ & 97, 4900, 20 & 510 & 5.92e-07 & 1.4118326e+04 & 622.1 &25.0 \\ 
$n=1225$ & SDPNAL+ & 158, 168, 2367 & 508& 1.59e-06 & 1.4118326e+04 & 1821.3 & - \\ [3pt] 

sym-4& RiNNAL+ & 76, 3850, 16 & 508 & 9.83e-07 & 1.4041294e+04 & 403.3 &19.6 \\ 
$n=1225$ & SDPNAL+ & 211, 218, 2945 & 506& 1.10e-06 & 1.4041294e+04 & 2196.8 & - \\ [3pt] 

sym-5& RiNNAL+ & 96, 4850, 20 & 509 & 9.28e-07 & 1.4111671e+04 & 604.8 &24.4 \\ 
$n=1225$ & SDPNAL+ & 198, 210, 2661 & 507& 1.08e-06 & 1.4111670e+04 & 2068.9 & - \\ [3pt] 

sym-6& RiNNAL+ & 81, 4100, 17 & 509 & 9.09e-07 & 1.4047336e+04 & 444.9 &20.5 \\ 
$n=1225$ & SDPNAL+ & 107, 114, 1852 & 507& 1.18e-06 & 1.4047336e+04 & 1470.9 & - \\ [3pt] 

sym-7& RiNNAL+ & 84, 4250, 17 & 508 & 7.78e-07 & 1.4117649e+04 & 483.2 &21.0 \\ 
$n=1225$ & SDPNAL+ & 271, 271, 2881 & 506& 1.24e-06 & 1.4117649e+04 & 2334.6 & - \\ [3pt] 

sym-8& RiNNAL+ & 75, 3800, 15 & 507 & 6.91e-07 & 1.4152284e+04 & 382.5 &18.4 \\ 
$n=1225$ & SDPNAL+ & 168, 181, 2544 & 505& 1.38e-06 & 1.4152284e+04 & 1891.4 & - \\ [3pt] 

sym-9& RiNNAL+ & 84, 4250, 17 & 508 & 7.01e-07 & 1.4163208e+04 & 480.9 &20.6 \\ 
$n=1225$ & SDPNAL+ & 210, 216, 2950 & 506& 9.89e-07 & 1.4163208e+04 & 2136.0 & - \\ [3pt] 

sym-10& RiNNAL+ & 86, 4350, 18 & 508 & 7.85e-07 & 1.4125583e+04 & 465.2 &22.1 \\ 
$n=1225$ & SDPNAL+ & 222, 228, 2922 & 506& 1.12e-06 & 1.4125582e+04 & 2229.1 & - \\ [3pt] 

\end{longtable}
\end{footnotesize}

\end{document}